\documentclass{amsart}
\usepackage[utf8]{inputenc}

\usepackage{hyperref}

\usepackage{a4wide}
\usepackage{amsmath}
\usepackage{amsfonts}
\usepackage{amssymb}
\usepackage{amscd}
\usepackage{mathtools}
\usepackage{bbm}
\usepackage{mathrsfs}
\usepackage{amsthm}
\usepackage{verbatim}
\usepackage{upgreek}
\usepackage{relsize}
\usepackage[shortlabels]{enumitem}
\usepackage{dsfont}
\usepackage{graphicx}
\usepackage{url}

\usepackage{xcolor} 

\newcommand{\eps}{\varepsilon}
\newcommand{\ov}{\overline}
\newcommand{\id}{\textnormal{id}}
\newcommand{\mc}{\mathcal}

\newcommand{\mrm}{\mathrm}
\newcommand{\mscr}{\mathscr}
\newcommand{\mf}{\mathfrak}
\newcommand{\msf}{\mathsf}
\newcommand{\I}{\mathbbm{1}}

\newcommand{\vp}{\varphi}

\newcommand{\md}{\operatorname{d}\!}
\newcommand{\cst}{\ifmmode \mathrm{C}^* \else $\mathrm{C}^*$\fi}

\newcommand{\la}{\langle}
\newcommand{\ra}{\rangle}

\newcommand{\bbGamma}{{\mathpalette\makebbGamma\relax}}
\newcommand{\makebbGamma}[2]{%
  \raisebox{\depth}{\scalebox{1}[-1]{$\mathsurround=0pt#1\mathbb{L}$}}%
}

\DeclareSymbolFont{bbold}{U}{bbold}{m}{n}
\DeclareSymbolFontAlphabet{\mathbbold}{bbold}
\newcommand{\GGamma}{\bbGamma}

\newcommand{\NN}{\mathbb{N}}
\newcommand{\RR}{\mathbb{R}}
\newcommand{\KK}{\mathbb{K}}
\newcommand{\CC}{\mathbb{C}}
\newcommand{\ZZ}{\mathbb{Z}}
\newcommand{\GG}{\mathbb{G}}

\newcommand{\vv}{\mathrm{V}}
\newcommand{\Vv}{\mathds{V}}
\newcommand{\vV}{\text{\reflectbox{$\Vv$}}\:\!}
\newcommand{\ww}{\mathrm{W}}
\newcommand{\WW}{{\mathds{V}\!\!\text{\reflectbox{$\mathds{V}$}}}}
\newcommand{\Ww}{\mathds{W}}
\newcommand{\wW}{\text{\reflectbox{$\Ww$}}\:\!}

\newcommand{\wot}{\ifmmode \textsc{wot} \else \textsc{wot}\fi}
\newcommand{\sot}{\ifmmode \textsc{sot} \else \textsc{sot}\fi}
\newcommand{\sots}{\ifmmode \textsc{sot}^* \else \textsc{sot}$^*$\fi}
\newcommand{\ssot}{\ifmmode \sigma\textsc{-sot} \else $\sigma$-\textsc{sot }\fi}
\newcommand{\ssots}{\ifmmode \sigma\textsc{-sot}^* \else $\sigma$-\textsc{sot }$^*$\fi}
\newcommand{\swot}{\ifmmode \sigma\textsc{-wot} \else $\sigma$-\textsc{wot}\fi}

\newcommand{\Linf}{\operatorname{L}^{\infty}(\GG)}
\newcommand{\Linfd}{\operatorname{L}^{\infty}(\whG)}
\newcommand{\Lj}{\operatorname{L}^{1}(\GG)}
\newcommand{\Ljd}{\operatorname{L}^{1}(\whG)}

\newcommand{\PolG}{\Pol(\GG)}
\newcommand{\Gop}{\GG^{\operatorname{op}}}
\newcommand*{\Corep}{\mathsf{Corep}}
\newcommand{\one}{1\!\!1}
\newcommand*{\BT}{\mathbf{T}}

\newcommand{\wh}{\widehat}
\newcommand{\whG}{\widehat{\GG}}
\newcommand{\whK}{\widehat{\KK}}

\newcommand{\LdG}{\operatorname{L}^{2}(\GG)}

\newcommand{\IrrG}{\Irr(\GG)}

\newcommand{\oon}{\operatorname}

\newcommand{\tpold}{\!\!
{\scriptstyle
\text{
\raisebox{0.8pt}{
\textcircled{\raisebox{-1.7pt}{$\top$}}
} 
} 
} 
\!\!}

\newcommand{\tp}{\textnormal{$\tpold$}}

\newcommand{\stp}{\!\!\!
{\scriptscriptstyle
\text{
\raisebox{0.5pt}{
\textcircled{\raisebox{-1.2pt}{$\top$}}
} 
} 
} 
\!\!\!}

\DeclareMathOperator{\lin}{span}
\DeclareMathOperator{\Irr}{Irr}
\DeclareMathOperator{\Pol}{Pol}

\DeclareMathOperator{\Tr}{Tr}
\DeclareMathOperator{\B}{B}
\DeclareMathOperator{\A}{A}

\DeclareMathOperator{\M}{M}
\DeclareMathOperator{\K}{K}

\DeclareMathOperator{\N}{N}

\DeclareMathOperator{\LL}{L}

\DeclareMathOperator{\Rep}{Rep}
\DeclareMathOperator{\SU}{SU}

\DeclareMathOperator{\CB}{CB}

\DeclareMathOperator{\ad}{ad}
\newcommand{\conjc}{\operatorname{Cl}}

\newtheorem{theorem}{Theorem}[section]

\newtheorem{proposition}[theorem]{Proposition}
\newtheorem{lemma}[theorem]{Lemma}
\theoremstyle{definition}
\newtheorem{corollary}[theorem]{Corollary}
\newtheorem{remark}[theorem]{Remark}

\newtheorem{definition}[theorem]{Definition}
\numberwithin{equation}{section}

\begin{document}

\title{Averaging multipliers on locally compact quantum groups}

\author{Matthew Daws}
\address{Department of Mathematics and Statistics,
Lancaster University,
Lancaster,
LA1 4YF,
United Kingdom}
\email{matt.daws@cantab.net}

\author{Jacek Krajczok}
\address{Vrije Universiteit Brussel\\
Pleinlaan 2\\
1050 Brussels\\
Belgium
}
\email{jacek.krajczok@vub.be}

\author{Christian Voigt}
\address{School of Mathematics and Statistics \\
         University of Glasgow \\
         University Place \\
         Glasgow G12 8QQ \\
         United Kingdom 
}
\email{christian.voigt@glasgow.ac.uk}

\thanks{This work was supported by EPSRC grants EP/T03064X/1 and EP/T030992/1. Additionally, JK was partially supported by FWO grant 1246624N}



\date{}

\begin{abstract}
We study an averaging procedure for completely bounded multipliers on a locally compact quantum group with respect to a compact quantum subgroup. As a consequence we show that central approximation properties of discrete quantum groups are equivalent to the corresponding approximation properties of their Drinfeld doubles. 

This is complemented by a discussion of the averaging of Fourier algebra elements.  We compare the biinvariant Fourier algebra of the Drinfeld double of a discrete quantum group with the central Fourier algebra.  In the unimodular case these are naturally identified, but we show by exhibiting a family of counter-examples that they differ in general. 
\end{abstract}

\maketitle

\section{Introduction}
Approximation properties like amenability, weak amenability, the Haagerup property and the Haagerup-Kraus approximation property have been studied extensively in the setting of 
locally compact quantum groups, see \cite{Brannan} for a survey. In comparison to the case of groups, an interesting new feature in the quantum setting is the interplay between 
discrete quantum groups, their Drinfeld doubles, and the associated \cst-tensor categories. In particular, the central versions of amenability, the Haagerup property, weak amenability 
and central property (T) for discrete quantum groups can be recast at the level of \cst-tensor categories \cite{PVcstartensor}, thus building a natural bridge to the study 
of subfactors. Central approximation properties, in turn, are key to a range of fundamental results regarding the analytic structure of discrete quantum 
groups \cite{DFY_CCAP}, \cite{Freslon}. 

In this paper we show that, for all the approximation properties mentioned above, the central version of the property for a discrete quantum group 
is equivalent to the corresponding property of its Drinfeld double. Special cases of these equivalences were known previously, 
and our main point is that they actually hold in complete generality. This yields a succinct conceptual understanding of central approximation properties, and highlights the key role 
played by the Drinfeld double. 

We prove in fact a more general result. Namely, for a locally compact quantum group $\GG$ with a compact quantum subgroup $\KK \subseteq \GG$, 
we show that averaging elements of $\Linf$ with respect to left and right translations by $\KK$ 
maps completely bounded multipliers of $ \GG $ to $ \KK $-biinvariant completely bounded multipliers of $ \GG $. 
The resulting averaging map is contractive and preserves complete positivity, the Fourier algebra, and other standard properties of multipliers. 
In the special case that $ \GG = D(\GGamma) $ is the Drinfeld double of a discrete quantum group $ \GGamma $ and $ \KK = \widehat{\GGamma} $, the $\KK$-biinvariant function algebras 
may be indentified with the corresponding algebras of central functions on $\bbGamma$, and we show that the same holds for completely bounded multipliers. 
This yields our result linking the approximation properties of the Drinfeld double $D(\bbGamma)$ with the central approximation properties of $\bbGamma$.

We also discuss a number of related questions. For unimodular discrete quantum groups there is a well-known averaging procedure sending completely bounded multipliers to central 
completely bounded multipliers, \cite{Brannan, KrausRuan}.  We show that this procedure, which is different from the above averaging with respect to compact 
quantum subgroups, maps the Fourier algebra to itself. 
As a byproduct, we obtain that amenability of the Drinfeld double of a discrete quantum group 
is equivalent to its strong amenability, which in turn implies that the discrete quantum group must be unimodular. Some of these facts seem to be known to experts, but we were unable 
to locate precise references in the literature, and so we take the opportunity to use our techniques to give a self-contained account.

In addition we study the centre of the Fourier algebra for discrete quantum groups. 
Here the situation is more subtle than for completely bounded multipliers. 
We show that there is a canonical inclusion map from the biinvariant Fourier algebra of the Drinfeld double 
to the central Fourier algebra, and prove that this map is an isomorphism for unimodular discrete quantum groups. 
However, we also show that the biinvariant Fourier algebra may be strictly smaller than the central Fourier algebra in general. In fact, this happens already 
for the dual of $\SU_q(2) $. 

Let us explain how the paper is organised. In Section~\ref{section preliminaries} we collect some background material and fix 
our notation. Section~\ref{section multiplier} contains a brief review of completely bounded multipliers, 
and in Section~\ref{section aps} we record the definitions of (strong) amenability, the Haagerup property, weak amenability, and the approximation property in 
the setting of locally compact quantum groups and \cst-tensor categories. 
In Section~\ref{section averaging} we construct the averaging map for completely bounded multipliers of a locally compact quantum group with respect to a closed compact quantum subgroup. 
In Section~\ref{section double} we specialise to the setting of Drinfeld doubles of discrete quantum groups, and prove our main results concerning approximation properties. 
Section~\ref{sec:unimodular} contains some further results on approximation properties for discrete quantum groups and their Drinfeld doubles, which are related to the underlying 
discrete quantum group being unimodular. We show that central Fourier algebra elements can be approximated by central finitely-supported elements in this case, and that amenability of 
the Drinfeld double of a discrete quantum group is equivalent to the discrete quantum group being amenable and unimodular. 
Finally, in Section~\ref{sec:central_fourier} we give equivalent characterisations for the biinvariant Fourier algebra of the Drinfeld double to agree with the central Fourier algebra
of a discrete quantum group. In addition we show that the centre of the Fourier algebra of any strongly amenable, non-unimodular, and centrally weakly amenable discrete quantum group is strictly larger than the biinvariant Fourier algebra of its Drinfeld double. This applies in particular to the dual of $\SU_q(2)$ for $q\in \left]-1,1\right[\setminus\{0\}$. 

We conclude with some general remarks on notation. When $\mc{A}$ is a $\cst$-algebra we write $\M(\mc{A})$ for its multiplier algebra. The canonical pairing between a Banach space $X$ and its dual $X^*$ will be denoted via $\la \omega, x \ra = \omega(x)$ for $\omega\in X^*, x\in X$. If $\mathscr{H}$ is a Hilbert space we write $ \omega_{\xi, \eta}$ for the 
vector functional $\omega_{\xi, \eta}(T) = (\xi | T \eta) $, and abbreviate $ \omega_\xi = \omega_{\xi,\xi}$. 
We write $ \odot $ for the algebraic tensor product, $\otimes$ for the tensor product of Hilbert spaces or the minimal tensor product of \cst-algebras, and $ \M \bar{\otimes}\N $
for the spatial tensor product of von Neumann algebras $ \M,\N $. We denote by $\chi$ the flip map for tensor products of algebras. 

We would like to thank A. Skalski for several interesting comments on the first version of this paper.

\section{Preliminaries} \label{section preliminaries}

In this section we review some basic definitions and facts from the theory of locally compact quantum groups in the sense of Kustermans and Vaes \cite{KustermansUniversal, KV_LCQGs, KVVN, NeshveyevTuset}. 

By definition, a locally compact quantum group $\GG$ is given by a von Neumann algebra $\Linf$ together with a normal unital $\star$-homomorphism $\Delta_{\GG}\colon \Linf\rightarrow \Linf \bar{\otimes }\Linf$ called \emph{comultiplication}, satisfying $(\Delta_{\GG}\otimes \id) \Delta_{\GG}=(\id \otimes \Delta_{\GG}) \Delta_{\GG}$, 
and \emph{left} resp. \emph{right} \emph{Haar integrals} $\vp$ and $\psi$. These are normal, semifinite, faithful (n.s.f.) weights 
on $\Linf$ satisfying certain invariance conditions with respect to $\Delta_{\GG}$. Note that in general the von Neumann algebra $\Linf$ is non-commutative and will not be an algebra of function on a measure space. Following this notational convention, the predual of $\Linf$ is denoted by $\Lj$, and the GNS Hilbert space of $\vp$ by $\LdG$. 
We write $ \Lambda_\vp\colon \mf{N}_\vp \rightarrow \LdG $ for the GNS map, where $ \mf{N}_\vp = \{x \in \Linf\,|\, \vp(x^*x) < \infty\} $. 

With any locally compact quantum group $\GG$ one can associate the \emph{dual} locally compact quantum group $\whG$ in such a way that the correspondence between $\GG$ and $\whG$ extends 
Pontryagin duality. Furthermore, the Hilbert spaces $\LdG $ and $\LL^2(\whG)$ are identified in a canonical way.

The (left) \emph{Kac-Takesaki operator} $\ww^{\GG}\in \Linf\bar\otimes \Linfd$ is the operator on $\LdG\otimes \LdG$ defined via
\[
((\omega\otimes\id)\ww^{\GG *})\Lambda_\vp(x)=\Lambda_{\vp}((\omega\otimes \id)\Delta_{\GG}(x))\qquad(\omega\in \Lj, x\in\mf{N}_\vp).
\]
It is unitary and implements the comultiplication via $\Delta_{\GG}(x)=\ww^{\GG *}(\I\otimes x)\ww^{\GG}$ for $x\in\Linf$. By duality we also get the Kac-Takesaki operator for $\whG$, which 
is linked to $\ww^{\GG}$ via $\ww^{\whG}=\chi (\ww^{\GG *})$. Tomita-Takesaki theory yields two groups of \emph{modular automorphisms} $(\sigma^\vp_t)_{t\in \RR}, (\sigma^\psi_t)_{t\in \RR}$ 
and \emph{modular conjugations} $J_{\vp},J_{\psi}$ associated with the weights $\vp,\psi$, respectively. 
We will also use the right Kac-Takesaki operator defined by 
$\vv^{\GG} = (J_{\wh{\vp}} \otimes J_{\wh{\vp}}) \ww^{\whG} (J_{\wh{\vp}} \otimes J_{\wh{\vp}})\in \LL^{\infty}(\whG)'\bar{\otimes} \Linf$, where $ J_{\wh{\vp}} $ is the modular conjugation for the dual left Haar integral. The operator $ \vv^\GG $ implements the comultiplication via $\Delta_{\GG}(x)=\vv^{\GG}(x\otimes \I)\vv^{\GG *}$ for $x\in \Linf$.

The \emph{antipode} $S_{\GG}$ is a densely defined, typically unbounded, operator on $\Linf$ such that
\[
(\id\otimes \omega)\ww^{\GG}\in \oon{Dom}(S_{\GG})\;\textnormal{ and }S_{\GG}( (\id\otimes\omega)\ww^{\GG})=(\id\otimes \omega)\ww^{\GG *}\qquad(\omega\in \Ljd),
\]
and the \emph{unitary antipode} $R_{\GG}$ is a bounded, normal, $\star$-preserving, antimultiplicative map on $\Linf$ 
satisfying $\Delta_{\GG}R_{\GG}=\chi (R_{\GG}\otimes R_{\GG}) \Delta_{\GG}$, and given by $R_{\GG}(x) = J_\varphi x^* J_\varphi$ for $x\in \LL^\infty(\GG)$.  These maps are linked via $S_{\GG}=R_{\GG} \tau^{\GG}_{-i/2}=\tau^{\GG}_{-i/2} R_{\GG}$, 
where $(\tau^{\GG}_t)_{t\in \RR}$ is the group of \emph{scaling automorphisms} of $\Linf$. 
The left and right Haar integrals are unique up to a scalar, and we will choose these scalars such that $\vp=\psi\circ R_{\GG}$. 

We will mainly work with the weak$^*$-dense \cst-subalgebra $\mrm{C}_0(\GG)\subseteq\Linf$. It is defined as the norm-closure of $\{(\id\otimes\omega)\ww^{\GG}\,|\,\omega\in \Ljd\}$. 
After restriction, the comultiplication becomes a non-degenerate $\star$-homomorphism $\mrm{C}_{0}(\GG)\rightarrow \M(\mrm{C}_0(\GG)\otimes \mrm{C}_0(\GG))$. 
Similarly one defines $\mrm{C}_0(\whG)$, and then one obtains $\ww^{\GG}\in \M(\mrm{C}_0(\GG)\otimes \mrm{C}_0(\whG))$. 
Using the comultiplication of $\Linf$, we define a Banach algebra structure on $\Lj$ via $\omega\star\nu=(\omega\otimes\nu)\Delta_{\GG}$ for $\omega,\nu\in \Lj$. 
As $\Linf$ is the dual of $\Lj$, we have a canonical $\Lj$-bimodule structure on $\Linf$, given 
by $\omega\star x=(\id\otimes\omega)\Delta_{\GG}(x)$ and $x\star\omega=(\omega\otimes\id)\Delta_{\GG}(x)$. 
Treating $\Lj$ as the predual of the von Neumann algebra $\Linf$ yields an $\Linf$-bimodule structure on $\Lj$ defined 
via $x\omega=\omega(\cdot \,x),\omega x=\omega(x\,\cdot)$ for $x\in\Linf,\omega\in\Lj$.

There is also a universal version of $\mrm{C}_0(\GG)$ which we denote by $\mrm{C}_0^u(\GG)$, see \cite{KustermansUniversal, SW_MultUniII}.  It comes together with a comultiplication 
$\Delta^u_{\GG}\colon \mrm{C}^u_0(\GG)\rightarrow \M(\mrm{C}_0^u(\GG)\otimes \mrm{C}_0^u(\GG))$ and a surjective $\star$-homomorphism $\Lambda_{\GG}\colon \mrm{C}_0^u(\GG)\rightarrow \mrm{C}_0(\GG)$ which respects the comultiplications. The Kac-Takesaki operator $\ww^{\GG}$ admits a lift to a unitary operator $\Ww^{\GG}\in \M(\mrm{C}_0^u(\GG)\otimes \mrm{C}_0(\whG))$ 
satisfying $(\Lambda_{\GG}\otimes \id)\Ww^{\GG}=\ww^{\GG}$. Using this operator, we can introduce the half-lifted comultiplication
\[
\Delta^{u,r}_{\GG}\colon \mrm{C}_0(\GG)\ni x \mapsto
\Ww^{\GG *}(\I\otimes x)\Ww^{\GG}\in \M(\mrm{C}^u_0(\GG)\otimes \mrm{C}_0(\GG)). 
\]
It satisfies $(\id\otimes\Lambda_{\GG})\Delta^u_{\GG} = \Delta^{u,r}_{\GG} \Lambda_{\GG}$.
Similarly, writing $\mrm{C}_0(\whG')=J_{\wh{\vp}} \mrm{C}_0(\whG) J_{\wh{\vp}}$, the right Kac-Takesaki operator $\vv^{\GG}\in \M(\mrm{C}_0(\whG ')\otimes \mrm{C}_0(\GG))$ 
admits a lift to a unitary operator $\vV^{\GG}\in \M(\mrm{C}_0(\whG')\otimes \mrm{C}_0^u(\GG))$ satisfying $(\id\otimes \Lambda_{\GG})\vV^{\GG}=\vv^{\GG}$.  Recall from \cite[Section~4]{KVVN} the opposite quantum group $\Gop$, and that $\ww^{\Gop} = \chi(\vv^{\GG *})$.  It follows that one way to define $\vV^{\GG}$ is as $\chi(\Ww^{\Gop *})$.  We define
\[
\Delta^{r,u}_{\GG}\colon \mrm{C}_0(\GG)\ni x \mapsto 
\vV^{\GG} (x\otimes \I)\vV^{\GG *}\in \M( \mrm{C}_0(\GG)\otimes \mrm{C}_0^u(\GG)), 
\]
and find that $(\Lambda_\GG\otimes\id)\Delta^u_{\GG} = \Delta^{r,u}_{\GG}\Lambda_\GG$.
Indeed, one could also define $\Delta^{r,u}_{\GG}$ using $\Delta^{u,r}_{\Gop}$.
We can iterate these constructions, which we illustrate by way of an example. Define $\Delta^{u,r,u}_\GG = (\Delta^{u,r}_\GG\otimes\id)\Delta^{r,u}_\GG$, and observe that
\begin{align*}
\Delta^{u,r,u}_\GG \Lambda_\GG &= (\Delta^{u,r}_\GG\otimes\id)\Delta^{r,u}_\GG\Lambda_\GG = (\Delta^{u,r}_\GG\Lambda_\GG\otimes\id)\Delta^u_\GG 
= (\id\otimes\Lambda_\GG\otimes\id)(\Delta^u_\GG\otimes\id)\Delta^u_\GG \\
&= (\id\otimes\Lambda_\GG\otimes\id)(\id\otimes\Delta^u_\GG)\Delta^u_\GG = 
(\id\otimes \Delta_{\GG}^{r,u})(\id\otimes \Lambda_{\GG})\Delta_{\GG}^u=
(\id\otimes\Delta^{r,u}_\GG)\Delta^{u,r}_\GG\Lambda_\GG.
\end{align*}
As $\Lambda_\GG$ is onto, this shows in particular that also $\Delta^{u,r,u}_\GG = (\id\otimes\Delta^{r,u}_\GG)\Delta^{u,r}_\GG$.

We define $\lambda_{\GG}\colon \Lj\rightarrow \mrm{C}_0(\whG)$ by $\lambda_{\GG}(\omega)=(\omega\otimes\id)\ww^{\GG}$, and similarly for $\whG$. The \emph{Fourier algebra of $\GG$} is $\A(\GG)=\lambda_{\whG}(\Ljd)$. One checks that $\lambda_{\whG}$ is multiplicative, and that $\A(\GG)$ is a dense subalgebra of $\mrm{C}_0(\whG)$. As $\lambda_{\whG}$ is also injective, we can 
define an operator space structure on $\A(\GG)$ by imposing the condition that $\lambda_{\whG}\colon \Ljd\rightarrow \A(\GG)$ is completely isometric.

We will also use two larger algebras: the \emph{reduced Fourier-Stieltjes algebra} $\B_r(\GG)$ and the \emph{Fourier-Stieltjes algebra} $\B(\GG)$ defined via
\[
\B_r(\GG)=\{(\id\otimes \omega)(\ww^{\GG *})\,|\,
\omega\in \mrm{C}_0(\whG)^*\},\quad
\B(\GG)=\{(\id\otimes \omega')(\wW^{\GG *})\,|\,
\omega'\in \mrm{C}_0^u(\whG)^*\}.
\]
Together with the algebra structure of $\Linf$ and the norms $\|(\id\otimes\omega)(\ww^{\GG *})\|_{\B_r(\GG)}=\|\omega\|$ on $\B_r(\GG)$ and $\|(\id\otimes\omega')(\wW^{\GG *})\|_{\B(\GG)}=\|\omega'\|$ on $\B(\GG)$, both $\B_r(\GG)$ and $\B(\GG)$ become Banach algebras. In fact, both are dual Banach algebras with preduals given 
by $ \mrm{C}_0(\whG)$ and $\mrm{C}_0^u(\whG)$, respectively. When speaking about the weak$^*$-topology on the (reduced) Fourier-Stieltjes algebra we mean the topology arising this 
way. Furthermore we note that $\A(\GG)\subseteq\B_r(\GG)\subseteq \B(\GG)\subseteq \M(\mrm{C}_0(\GG))$, and that the first two inclusions are isometric. We will sometimes write $\lambda_{\whG}(\omega)=(\id\otimes\omega)\ww^{\GG *}$ also for $\omega\in \mrm{C}_{0}(\whG)^*$.

A locally compact quantum group $\GG$ is called \emph{compact} if $\mrm{C}_0(\GG)$ is unital, and in this case we write $\mrm{C}(\GG)$ for $\mrm{C}_0(\GG)$. 
If $\GG$ is compact then $\varphi=\psi$ is a normal state, the \emph{Haar state}, often denoted by $h \in \Lj$. 
The representation theory of compact quantum groups shares many features with the one for classical compact groups. 
In particular, every irreducible unitary representation of a compact quantum group $\GG$ is finite-dimensional, and every unitary representation decomposes into a direct sum 
of irreducibles. We write $\Irr(\GG)$ for the set of equivalence classes of irreducible unitary representations of $\GG$,  
and for each $\alpha\in\Irr(\GG)$ let $U^\alpha = [U^\alpha_{i,j}]_{i,j=1}^{\dim(\alpha)} \in \mathbb M_{\dim(\alpha)}(\mrm{C}(\GG))$ be a representative. There is a unique positive invertible 
matrix $\uprho_\alpha$ with $\Tr(\uprho_\alpha) = \Tr(\uprho_\alpha^{-1})$ which intertwines $U^\alpha$ and the double contragradient representation $(U^\alpha)^{cc}$. 
Let us also recall the Schur orthogonality relations
\begin{equation}\label{eq:orthog_matrix}
h\big( (U^\beta_{i,k})^* U^\alpha_{j,l} \big) = \delta_{\alpha,\beta} \delta_{k,l}
\frac{(\uprho_\alpha^{-1})_{j,i}}{\Tr(\uprho_\alpha)}, \qquad
h\big( U^\alpha_{i,j} (U^\beta_{k,l})^* \big) = \delta_{\alpha,\beta} \delta_{i,k} \frac{(\uprho_\alpha)_{l,j}}{\Tr(\uprho_\alpha)}
\end{equation}
for $\alpha,\beta \in \Irr(\GG), 1\le i,j\le \dim(\alpha),1\le k,l\le \dim(\beta)$.  We often write $\dim_q(\alpha) = \Tr(\uprho_\alpha)$ the quantum dimension of $\alpha$.
Recall that the collection $\{ U^\alpha_{i,j}\,|\,\alpha\in\Irr(\wh\bbGamma),1\le i,j\le \dim(\alpha)\}$ forms a basis for a dense Hopf $*$-subalgebra of $\mrm{C}(\GG)$, which we denote by $\Pol(\GG)$.  We have that $S(U^\alpha_{i,j}) = (U^\alpha_{j,i})^*$.  The matrices $\uprho_\alpha$ allow us to define functionals $f_z \colon \PolG \to \mathbb C$ by
\[ f_z(U^\alpha_{i,j}) = (\rho_\alpha^z)_{i,j} \qquad (z\in\mathbb C). \]
Then the scaling and modular automorphism groups are given by
\[
 \sigma_z = (f_{iz}\otimes\id\otimes f_{iz})\Delta^{(2)}, \quad
\tau_z = (f_{iz}\otimes\id\otimes f_{-iz})\Delta^{(2)}\qquad (z\in\mathbb C). \]
For $\alpha\in\IrrG$, the contragradient representation has matrix $[(U^\alpha_{i,j})^*]_{i,j=1}^{\dim(\alpha)}$ and is similar to the unitary representation $U^{\overline\alpha}$.

We say that $\GGamma$ is \emph{discrete} when $\wh\GGamma$ is compact, and write $\mrm{c}_0(\GGamma)$ for $\mrm{C}_0(\GGamma)$ in this case. 
The representation theory of $\wh\GGamma$ implies that $\mrm{c}_0(\GGamma)$ is the $\mrm{c}_0$-direct sum of finite-dimensional matrix algebras, and that $\ell^\infty(\GGamma) = \LL^\infty(\bbGamma)$
is the $\ell^\infty$-direct product. We write $\mrm{c}_{00}(\GGamma) \subseteq \mrm{c}_0(\GGamma)$ for the dense subalgebra of finitely-supported elements with respect to the direct 
sum structure. We can write $\ww^{\wh\GGamma} \in \LL^\infty(\wh\GGamma) \bar\otimes \ell^\infty(\GGamma)$ explicitly as
\begin{equation}\label{eq:wcmpt}
\ww^{\wh\GGamma} = \sum_{\alpha\in\Irr(\wh\bbGamma)}\sum_{i,j=1}^{\dim(\alpha)} U^\alpha_{i,j} \otimes e^{\alpha}_{i,j}
\end{equation}
where $\{e^{\alpha}_{i,j}\}_{i,j=1}^{\dim(\alpha)}$ are the matrix units of the matrix block $\mathbb M_{\dim(\alpha)}(\mathbb{C}) \subseteq \mrm{c}_0(\GGamma)$.

If there is no risk of confusion we will sometimes omit subscripts and abbreviate e.g.~$\Delta$ for $\Delta_{\GG}$ and $\wh\Delta$ for $\Delta_{\whG}$.

\section{Completely bounded multipliers} \label{section multiplier}

In this section we recall some background on multipliers of locally compact quantum groups and $ \cst $-tensor categories. 
We follow the notation and conventions in \cite{DKV_ApproxLCQG}, and refer to \cite{DawsCPMults}, \cite{CBMultipliers}, \cite{PVcstartensor} for more information. 

A \emph{left centraliser} of a locally compact quantum group $ \GG $ is linear map $T\colon \Ljd\rightarrow\Ljd$ such that
\[
T(\omega\star \omega')=T(\omega)\star \omega'\quad (\omega,\omega'\in \Ljd).
\]
We denote by $C^l_{cb}(\Ljd)$ the space of completely bounded left centralisers. Together with the completely bounded norm and composition, $C^l_{cb}(\Ljd)$ becomes a Banach algebra. 
We equip this space with an operator space structure by requiring that the embedding $C^l_{cb}(\Ljd)\hookrightarrow \CB(\Ljd)$ is completely isometric. 
This turns $C^l_{cb}(\Ljd)$ into a completely contractive Banach algebra.

An operator $b\in \Linf$ is said to be a \emph{completely bounded (CB) left multiplier} if $b \A(\GG)\subseteq \A(\GG)$ and the associated map
\[
\Theta^l(b)_*\colon \Ljd\rightarrow\Ljd
\quad\text{satisfying}\quad
b\wh\lambda(\omega)=\wh{\lambda}(\Theta^l(b)_*(\omega))\qquad(\omega\in\Ljd)
\]
is completely bounded. Since $\wh\lambda$ is multiplicative we have $\Theta^l(b)_*\in C^l_{cb}(\Ljd)$ for any $b\in \M^l_{cb}(\A(\GG))$. 
We write $\Theta^l(b)=(\Theta^l(b)_*)^*$, and denote the space of CB left multipliers by $\M^l_{cb}(\A(\GG))$. Any Fourier algebra element $\wh{\lambda}(\omega)\in\A(\GG)$ is a CB left multiplier, with $\Theta^l(\wh{\lambda}(\omega))_*\in \CB(\Ljd)$ being the left multiplication by $\omega$ and $\Theta^l(\wh{\lambda}(\omega))=(\omega\otimes\id)\wh\Delta$. 
Let us also note that $\M^l_{cb}(\A(\GG))\subseteq \M(\mrm{C}_0(\GG))$. 

If $T\in C^l_{cb}(\Ljd)$ is a left centraliser then its Banach space dual $T^*$ is a normal CB map on $\Linfd$ which is a left $\Ljd$-module homomorphism, 
i.e. $T^*\in \prescript{}{\Ljd}{\CB^{\sigma}(\Linfd)}$. Then, by \cite[Corollary 4.4]{JungeNeufangRuan}, there exists a unique CB left multiplier $b\in \M^l_{cb}(\A(\GG))$ 
satisfying $\Theta^l(b) = T^*$, that is, $\Theta^l(b)_* = T$. It follows that the map $\Theta^l(\cdot)_*\colon\M^l_{cb}(\A(\GG)) \rightarrow C^l_{cb}(\Ljd)$ is bijective. We define the operator space structure on $\M^l_{cb}(\A(\GG))$ so that these spaces become completely isometric.

If $b\in \M^l_{cb}(\A(\GG))$ is such that $\Theta^l(b)$ is completely positive, we say that $b$ is a \emph{completely positive (CP) left multiplier}. Furthermore, $b$ is \emph{normalised} if $\Theta^l(b)$ is unital. 

\begin{lemma}\label{lem:cp_fourier_pos}
Let $a=\wh\lambda(\wh\omega)\in\A(\GG)$ for some $\wh\omega\in\Ljd$.  Then $a$ is a CP multiplier if and only if $\wh\omega$ is positive, and 
$a$ is normalised if and only if $\wh\omega(\I)=1$.
\end{lemma}
\begin{proof}
We know that $\Theta^l(a) = (\wh\omega\otimes\id)\wh\Delta$ and so $\Theta^l(a)(\I)=\I$ if and only if $\wh\omega(\I)=1$. If $\wh\omega$ is positive then $\Theta^l(a)$ is CP. Hence it 
remains to prove that if $\Theta^l(a)$ is CP then $\wh\omega$ is positive. By \cite[Theorem~5.2]{DawsCPMults} there is a positive linear functional
$\mu\in\mrm{C}_0^u(\wh\GG)^*$ with $a = (\id\otimes\mu)(\wW^*)$. The Banach space adjoint $\Lambda_{\wh\GG}^*\colon \mrm{C}_0(\wh\GG)^* \rightarrow \mrm{C}^u_0(\wh\GG)^*$ is isometric. 
Let $\phi$ be the composition of $\Ljd \rightarrow \mrm{C}_0(\wh\GG)^*$ with $\Lambda_{\wh\GG}^*$, so $\phi$ is also isometric. We see that
\[ (\id\otimes\phi(\wh\omega))(\wW^*)
= (\id\otimes\wh\omega)(\id\otimes\Lambda_{\wh\GG})(\wW^*)
= (\id\otimes\wh\omega)(\ww^*) = \wh\lambda(\wh\omega) = a = (\id\otimes\mu)(\wW^*). \]
As $\{(\omega\otimes\id)(\wW^*)\,|\,\omega\in \Lj\}$ is norm dense in $\mrm{C}_0^u(\wh\GG)$, it follows that $\phi(\wh\omega)=\mu$, and so $\wh\omega$ is positive, as required.
\end{proof}

Since the inclusion $\M^l_{cb}(\A(\GG))\hookrightarrow \Linf$ is contractive, we can consider the restriction of the Banach space adjoint of this map, 
giving a map $\Lj\rightarrow \M^l_{cb}(\A(\GG))^*$. Let us define $Q^l(\A(\GG))\subseteq \M^l_{cb}(\A(\GG))^*$ as the closure of the image of this map.
According to \cite[Theorem 3.4]{CBMultipliers}, the space $Q^l(\A(\GG))$ is a predual of $\M^l_{cb}(\A(\GG))$, i.e.~we have
\[
Q^l(\A(\GG))^*\cong \M^l_{cb}(\A(\GG))
\]
completely isometrically. Whenever we speak about the weak$^*$-topology on $\M^l_{cb}(\A(\GG))$ we will have in mind this particular choice of predual.
In this way $\M^l_{cb}(\A(\GG))$ becomes a dual Banach algebra, that is, the multiplication of $\M^l_{cb}(\A(\GG))$ is separately weak$^*$-continuous.

Next we recall the definition of multipliers on rigid \cst-tensor categories from the work of Popa-Vaes \cite[Section~3]{PVcstartensor}. 
If $ \BT $ is a \cst-category and $ X, Y \in \BT $ are objects we write $ \BT(X,Y) $ for the space of morphisms from $ X $ to $ Y $. 
We denote by $ \id_X $ or $ \id $ the identity morphism in $ \BT(X,X) $. 
By definition, a \cst-tensor category is a \cst-category $ \BT $ together with a bilinear $ * $-functor $ \otimes\colon \BT \times \BT \rightarrow \BT $, 
a distinguished object $ \one \in \BT $ and unitary natural isomorphisms 
\begin{align*}
\one \otimes X \cong X \cong X \otimes \one, \qquad (X \otimes Y) \otimes Z \cong X \otimes (Y \otimes Z)  
\end{align*}
satisfying certain compatibility conditions. For simplicity we shall always assume that $ \BT $ is \emph{strict}, which means that these unitary natural isomorphisms are 
identities, and that the tensor unit $ \one $ is simple. 

We also assume that $ \BT $ is \emph{rigid}. 
Every rigid \cst-tensor category $ \BT $ is semisimple, that is, every object of $ \BT $ is isomorphic to a finite direct sum of simple objects. 
We write $ \Irr(\BT) $ for the set of isomorphism classes of simple objects in $ \BT $, and choose representatives $ X_i \in \BT $ for elements $ i = [X_i] \in \Irr(\BT) $. 

Let $ \BT $ be a rigid \cst-tensor category. By definition, a \emph{multiplier} on $ \BT $ is a family $ \theta = (\theta_{X,Y}) $ of linear 
maps $ \theta_{X,Y}\colon \BT(X \otimes Y, X \otimes Y) \rightarrow \BT(X \otimes Y, X \otimes Y) $ 
for $ X, Y \in \BT $ such that 
\begin{align*}
\theta_{X_2,Y_2}(gfh^*) &= g \theta_{X_1, Y_1}(f) h^*, \\
\theta_{X_2 \otimes X_1, Y_1 \otimes Y_2}(\id_{X_2} \otimes f \otimes \id_{Y_2}) &= \id_{X_2} \otimes \theta_{X_1,Y_1}(f) \otimes \id_{Y_2},
\end{align*}  
for all $ X_i, Y_i \in \BT $, $ f \in \BT(X_1 \otimes Y_1, X_1 \otimes Y_1) $ and $ g, h \in \BT(X_1, X_2) \otimes \BT(Y_1, Y_2) \subseteq \BT(X_1 \otimes Y_1, X_2 \otimes Y_2) $. 
A multiplier $ \theta = (\theta_{X,Y}) $ on $ \BT $ is said to be completely positive (or a \emph{CP multiplier}) if all the maps $ \theta_{X,Y} $ are completely positive. 
A multiplier $ \theta = (\theta_{X,Y}) $ on $ \BT $ is said to be completely bounded (or a \emph{CB multiplier}) if all the maps $ \theta_{X,Y} $ are completely bounded 
and $ \| \theta \|_{cb} = \sup_{X,Y \in \BT} \|\theta_{X,Y}\|_{cb} < \infty $. 

It is shown in \cite[Proposition 3.6]{PVcstartensor} that multipliers on $ \BT $ are in canonical bijection with functions $ \Irr(\BT) \rightarrow \mathbb{C} $, and we will identify 
a multiplier $\theta = (\theta_{X,Y})$ with its associated function $\theta = (\theta(k))_{k\in\Irr(\BT)}$. Note that we have $\|(\theta(k))_{k\in\Irr(\BT)}\|_\infty \leq \|\theta\|_{cb}$.

We write $ \M_{cb}(\BT) $ for the space of CB multipliers on $ \BT $. Via composition of maps and the CB norm this becomes naturally a Banach algebra, such that the product on $ \M_{cb}(\BT) $ corresponds to pointwise multiplication of functions on $ \Irr(\BT) $. In fact, $ \M_{cb}(\BT) $ is a dual Banach algebra, whose predual $ Q(\BT) $ can be constructed using 
the tube algebra of $ \BT $, see \cite[Corollary 5.3]{ARANO_DELAAT_WAHL_fourier}.  

A standard example of a rigid \cst-tensor category is the category $\BT=\Rep(\GG)$ of finite dimensional unitary representations of a compact quantum group $\GG$. 
For a discrete quantum group $\GGamma$ we shall write $\Corep(\GGamma)=\Rep(\widehat{\GGamma})$. Such $\cst$-tensor categories will be the only ones of interest to us in this paper.

\section{Approximation properties} \label{section aps}

In this section we review the definition of various approximation properties in the theory of locally compact quantum groups and rigid $ \cst $-tensor categories. 

We begin with the case of locally compact quantum groups, compare \cite{BT_Amenability, Brannan, DFSW_HAP, DKV_ApproxLCQG}. 

\begin{definition} \label{defapsqg}
Let $\GG$ be a locally compact quantum group. Then we say that $\GG$
\begin{itemize}
\item is \emph{strongly amenable} if there exists a bounded approximate identity of $\A(\GG) $ consisting of CP multipliers of $ \GG $.  
\item has the \emph{Haagerup property} if there exists a bounded approximate identity of $\mrm{C}_0(\GG)$ which consists of CP multipliers of $ \GG $.
\item is \emph{weakly amenable} if there exists a left approximate identity $(e_i)_{i\in I}$ of $\A(\GG)$ satisfying $\lim \sup_{i \in I} \|e_i\|_{cb} < \infty$. 
In this case, the smallest $M$ such that we can choose $\|e_i\|_{cb}\leq M$ for all $i \in I$ is the \emph{Cowling--Haagerup constant} of $\GG$, denoted $\Lambda_{cb}(\GG)$. 
\item has the \emph{approximation property} if there exists a net $(e_i)_{i\in I}$ in $\A(\GG)$ which converges to $\I$ in the weak$^*$ topology of $\M^l_{cb}(\A(\GG))$.
\end{itemize}
\end{definition}

While the definitions of the Haagerup property, weak amenability, and the approximation property for $\GG$ in Definition~\ref{defapsqg} are standard, this is not quite
the case for strong amenability. Our terminology here follows \cite{dqv_amen_bicrossed}, \cite{Tomatsuamenablediscrete}. Usually, strong amenability of $\GG$ 
is phrased in terms of the dual, by equivalently saying that $\whG$ is \emph{coamenable}. 

A number of equivalent characterisations of coamenability can be found in \cite{BT_Amenability}. In particular, conditions (6)--(8) in \cite[Theorem 3.1]{BT_Amenability} 
say that $\whG$ is coamenable if and only if $\Ljd\cong \A(\GG)$ has a bounded (left, right, or two-sided) approximate identity. An examination of the proof (in particular, 
that condition (5) implies conditions (6)--(8)) shows that when $\whG$ is coamenable, the Banach algebra $\Ljd$ has a bounded (two-sided) approximate identity consisting of states; 
compare also \cite[Theorem~2]{HNR_MultsNewClass}. That is, $\A(\GG)$ has a bounded approximate identity consisting of CP multipliers. 
It follows that our definition of strong amenability is compatible with \cite[Definition 3.1]{BT_Amenability}. 
Since we can in fact choose the approximate identity to consist of states the associated multipliers will be normalised. Therefore we could equivalently strengthen the CP condition 
in the definition of strong amenability to UCP.

Let us add that \emph{amenability} of $\GG$ is defined in terms of the existence of an invariant mean on $\Linf$, see Definition~\ref{defn:amenable}.  This is an a priori weaker notion than strong amenability, 
and the two concepts are known to be equivalent only in special cases, compare \cite{Brannan}, \cite{Tomatsuamenablediscrete}.

For equivalent characterisations of the Haagerup property see \cite[Theorem 6.5]{DFSW_HAP}; however be aware that some equivalent conditions in this reference may require a second 
countability assumption.

In the setting of discrete quantum groups one also considers central approximation properties. We recall that a multiplier $ a \in \M^l_{cb}(\A(\GG)) \subseteq \Linf $ 
is called central if it is contained in the centre of $\M^l_{cb}(\A(\GG))$, or equivalently, in the centre of $\Linf$.

\begin{definition}\label{defapsdqg}
Let $\GGamma$ be a discrete quantum group. Then we say that $\GGamma$
\begin{itemize}
\item is \emph{centrally strongly amenable} if there is a net $(e_i)_{i \in I}$ consisting of finitely supported central CP multipliers of $\GGamma$ converging to $\I$ pointwise. 
\item has the \emph{central Haagerup property} if there exists a bounded approximate identity of $\mrm{c}_0(\GGamma)$ which consists of central CP multipliers of $ \GGamma $.
\item is \emph{centrally weakly amenable} if there exists a net $(e_i)_{i\in I}$ of finitely supported central multipliers converging to $\I$ pointwise and 
satisfying $\lim \sup_{i \in I} \|e_i\|_{cb} < \infty$. 
In this case, the smallest $M$ such that we can choose $\|e_i\|_{cb}\leq M$ for all $i \in I$ is the \emph{central Cowling--Haagerup constant} of $\GGamma$, 
denoted $ \mc{Z}\Lambda_{cb}(\GGamma)$. 
\item has the \emph{central approximation property} if there exists a net $(e_i)_{i\in I}$ of finitely supported central multipliers which converges to $\I$ in the weak$^*$ topology 
of $\M^l_{cb}(\A(\GG))$.
\end{itemize}
\end{definition}

Note that $ \mrm{c}_{00}(\GGamma) \subseteq \A(\GGamma) $ is a dense subspace for the Fourier algebra norm. It follows that a bounded approximate identity of $ \A(\GGamma) $ is the 
same thing as a bounded net $(e_i)_{i \in I}$ in $ \A(\GGamma)$ converging to $\I$ pointwise.  If $\GGamma$ is strongly amenable then there is such a net which consists of finitely 
supported CP multipliers, as can be seen by suitably approximating the associated states in $\Ljd$. Thus central strong amenability is a natural central version of our definition of strong amenability. However, we do not know whether the a priori weaker requirement of existence of a net $(e_i)_{i \in I}$ of CP multipliers in the centre of $\A(\GGamma)$ converging to $\I$ pointwise is equivalent to central strong amenability; though this is true when $\bbGamma$ is unimodular, see Proposition~\ref{thm:kac_nofs} below.

Finally, let us review the corresponding notions for rigid \cst-tensor categories \cite{PVcstartensor}, \cite{DKV_ApproxLCQG}. 

\begin{definition} \label{defapstc}
Let $ \BT $ be a rigid \cst-tensor category. Then $\BT $
\begin{itemize}
\item is \emph{amenable} if there exists a net $(\varphi_i)_{i \in I}$ of finitely supported CP multipliers of $ \BT $ converging to $ \I $ pointwise. 
\item has the \emph{Haagerup property} if there exists a net of CP multipliers $(\varphi_i)_{i \in I}$ of $ \BT $ converging to $ \I $ pointwise such that $ \varphi_i \in \mrm{c}_0(\Irr(\BT)) $ 
for all $ i \in I$. 
\item is \emph{weakly amenable} if there exists a net of finitely supported CB multipliers $(\varphi_i)_{i \in I}$ of $ \BT $ converging to $ \I $ pointwise 
such that $ \lim \sup_{i \in I} \|\varphi_i\|_{cb} < \infty $. In this case, the smallest constant $ M $ such that we can choose $\|\varphi_i\|_{cb} \leq M $ for all $i \in I$ is 
the \emph{Cowling--Haagerup constant} of $\BT$, denoted $\Lambda_{cb}(\BT)$. 
\item has the \emph{approximation property} if there exists a net of finitely supported CB multipliers of $ \BT $ converging to $ \I $ in the weak*-topology of $ \M_{cb}(\BT) $. 
\end{itemize}
\end{definition}

If $\GGamma$ is a discrete quantum group then central multipliers for $\GGamma$ and categorical multipliers for $\Corep(\GGamma)$ can 
both be viewed as functions on $\Irr(\wh\GGamma)=\Irr(\Corep(\GGamma))$. 
It turns out that, in this way, the space of central CB-multipliers of $\GGamma$ identifies isometrically with the space of CB multipliers of $ \Corep(\GGamma) $. 
Moreover the weak$^*$-topologies agree, and this identification restricts to a bijection between the corresponding CP-multipliers 
\cite[Proposition 6.1]{PVcstartensor}, \cite[Lemma 8.6]{DKV_ApproxLCQG}. By definition, finite support, being $\mrm{c}_0$ and pointwise convergence of multipliers 
have the same meaning in either case. 

As a consequence, the categorical approximation properties for the \cst-tensor category $ \Corep(\GGamma) $ in Definition~\ref{defapstc} are equivalent to the 
corresponding central approximation properties for $\GGamma$ in Definition~\ref{defapsdqg}.

\section{The averaging construction} \label{section averaging}

In this section we discuss the averaging construction for multipliers on a locally compact quantum group with respect to a compact quantum subgroup. Let us note that averaging with 
respect to compact quantum subgroups has been widely used previously in different contexts, most notably in the study of quantum homogeneous spaces and compact quantum hypergroups, 
see for instance \cite{CHAPOVSKY_VAINERMAN}, \cite{DAS_FRANZ_WANG}. 

Let $ \GG $ be a locally compact quantum group and let $ \KK \subseteq \GG $ be a compact quantum subgroup \cite[Section~6.1]{DKSS_ClosedSub}. This means that there is a 
nondegenerate $ \star $-homomorphism $ \pi\colon \mrm{C}_0^u(\GG) \rightarrow \mrm{C}^u(\KK) \xrightarrow[]{\Lambda_{\KK}} \mrm{C}(\KK) $, compatible with comultiplications, such that the dual 
nondegenerate $ \star $-homomorphism $\mrm{c}_0(\whK) \rightarrow \M(\mrm{C}_0^u(\whG)) $ drops to a normal injective
unital $ \star $-homomorphism $ \widehat{\pi}\colon \ell^\infty(\whK) \rightarrow \LL^\infty(\whG) $.

The morphisms $ \pi $ and $ \widehat{\pi} $ can also be encoded at the level of the Kac-Takesaki operators, see for instance \cite{MRW_Hms_QGs}. 
If $ \Ww^{\GG} \in \M(\mrm{C}_0^u(\GG) \otimes \mrm{C}_0(\wh\GG)) $ is the half-universal lift of $\ww^\GG$ then we obtain a 
bicharacter $ (\pi \otimes \id)(\Ww^{\GG}) \in \M(\mrm{C}(\KK) \otimes \mrm{C}_0(\wh\GG))$ and a left action $ \lambda_{\pi}\colon \mrm{C}_0(\GG) \rightarrow \M(\mrm{C}(\KK) \otimes \mrm{C}_0(\GG))$ given by 
$$
\lambda_\pi(a)= (\pi \otimes \id)(\Ww^{\GG})^*(\I \otimes a) (\pi \otimes \id)(\Ww^{\GG}).  
$$
This is the restriction, to $\KK$, of the left translation action of $\GG$ on itself, and extends to an action $\LL^\infty(\GG) \rightarrow \LL^\infty(\KK) \bar\otimes \LL^\infty(\GG)$ 
on the level of von Neumann algebras. In a similar way, using the right Kac-Takesaki operator, one obtains a right action $\rho_\pi\colon \mrm{C}_0(\GG) \rightarrow \M(\mrm{C}_0(\GG) \otimes \mrm{C}(\KK))$ 
and the corresponding normal extension $\LL^\infty(\GG) \rightarrow \LL^\infty(\GG) \bar\otimes \LL^\infty(\KK)$.  

Combining $\lambda_\pi$ and $\rho_\pi$ with the Haar state $h$ of $\LL^\infty(\KK)$ we define the \emph{averaging map} $\Xi\colon \LL^\infty(\GG)\rightarrow \LL^\infty(\GG)$ by 
\begin{equation}\label{eq4}
\Xi = (h \otimes \id \otimes h)(\id \otimes \rho_\pi)\lambda_\pi = (h \otimes \id \otimes h)(\lambda_\pi \otimes \id)\rho_\pi. 
\end{equation}
Clearly $\Xi$ is a normal unital contractive CP map. 

This construction behaves well on the level of \cst-algebras.  For the following, recall from Section~\ref{section preliminaries} the maps $\Delta^{r,u}, \Delta^{u,r}$ and their iterated counterparts.

\begin{lemma} \label{averaginghelp}
The averaging map $\Xi$ restricts to a contractive CP map $\mrm{C}_0(\GG) \rightarrow \mrm{C}_0(\GG)$ such that
\begin{equation}\label{eq6}
\Xi(a)
= (h\pi \otimes \id \otimes h\pi)\Delta^{u,r,u}(a)
\qquad ( a\in\mrm{C}_0(\GG)).
\end{equation}
It also restricts to a strictly continuous unital contractive CP map $ \M(\mrm{C}_0(\GG)) \rightarrow \M(\mrm{C}_0(\GG))$, given by the same formula. 
\end{lemma}

\begin{proof}
For $a\in \mrm{C}_0(\GG)$ we have 
\begin{align*}
\lambda_\pi(a)
&= (\pi \otimes \id)(\Ww^{\GG})^*(\I \otimes a) (\pi \otimes \id)(\Ww^{\GG}) \\
&= (\pi \otimes \id)(\Ww^{\GG *}(\I \otimes a)\Ww^{\GG})
= (\pi\otimes\id)\Delta^{u,r}(a).
\end{align*}
Similarly, $\rho_\pi(a) = (\id\otimes\pi)\Delta^{r,u}(a)$. Thus for $a\in \mrm{C}_0(\GG)$,
\[ \Xi(a)
= (h\otimes\id\otimes h)((\pi\otimes\id)\Delta^{u,r}\otimes\id)(\id\otimes\pi)\Delta^{r,u}(a)
= (h\pi \otimes \id \otimes h\pi)\Delta^{u,r,u}(a), \]
as claimed.

From \cite[Proposition~6.1]{KustermansUniversal} we know that $(\Delta^u\otimes\id)(\Ww^{\GG}) = \Ww^{\GG}_{13}\Ww^{\GG}_{23}$.  Thus
\begin{align*}
  (\Delta^{u,r}\otimes\id)(\ww^{\GG})
&= (\Delta^{u,r}\Lambda_\GG\otimes\id)(\Ww^{\GG})
= (\id\otimes\Lambda_\GG\otimes\id)(\Delta^u\otimes\id)(\Ww^{\GG}) \\
&= (\id\otimes\Lambda_\GG\otimes\id)(\Ww^{\GG}_{13}\Ww^{\GG}_{23})
= \Ww^{\GG}_{13}\ww^{\GG}_{23}.
\end{align*}
Hence, given $a = (\id\otimes\omega)(\ww^{\GG *}) \in \A(\GG) \subseteq \mrm{C}_0(\GG)$ for some $\omega\in \Ljd$, we see that
\begin{align*}
(h\pi\otimes\id)\Delta^{u,r}(a)
&= (h\pi\otimes\id\otimes\omega)(\ww^{\GG *}_{23}\Ww^{\GG *}_{13})
= (\id\otimes\omega)\big( \ww^{\GG *} (1\otimes(h\pi\otimes\id)(\Ww^{\GG *})) \big).
\end{align*}
This equals $(\id\otimes b\omega)(\ww^{\GG *}) \in \A(\GG) \subseteq \mrm{C}_0(\GG)$ for $b = (h\pi\otimes\id)\Ww^{\GG *} \in \M(\mrm{C}_0(\whG))$.  By density, we conclude that $(h\pi\otimes\id)\Delta^{u,r}$ maps $\mrm{C}_0(\whG)$ to itself.

Analogously, $(\Delta^{r,u}\otimes\id)(\ww^\GG) = \ww^\GG_{13} \Ww^\GG_{23}$, and so
\[ (\id\otimes h\pi)\Delta^{r,u}(a) =
(\id\otimes h\pi\otimes\omega)(\Ww^{\GG *}_{23} \ww^{\GG *}_{13})
= (\id\otimes\omega b)(\ww^{\GG *}) \in \A(\GG), \]
so also $(\id\otimes h\pi)\Delta^{r,u}$ maps $\mrm{C}_0(\whG)$ to itself. By composition, $\Xi(a) \in \A(\GG)$ and $\Xi$  maps $\mrm{C}_0(\whG)$ to itself.

By \cite[Proposition~6.1]{KustermansUniversal} we know that $\Delta^u$ is a non-degenerate $*$-homomorphism, and hence the same is true of $\Delta^{u,r,u}$.  As $\pi$ is non-degenerate, it follows that $(\pi\otimes\id\otimes\pi)\Delta^{u,r,u}$ is non-degenerate, hence strictly continuous.  The strict extension to $\M(\mrm{C}_0(\GG))$ hence agrees with $(\id\otimes\rho_\pi)\lambda_\pi$ when restricted from $\Linf$ to $\M(\mrm{C}_0(\GG))$.  As $h$ is strictly continuous, and slice maps are strictly continuous, \cite[Proposition~8.3]{Lance_HilbMod}, it follows 
that $\Xi\colon \mrm{C}_0(\GG) \rightarrow \mrm{C}_0(\GG)$ is strictly continuous, and again, the strict extension to $\M(\mrm{C}_0(\GG))$ agrees with the 
restriction of $\Xi\colon \Linf \rightarrow \Linf$.
\end{proof}

\begin{remark}\label{rem:onesided1}
The structure of this proof shows that we could analogously define one-sided averaging maps $\Xi^l\colon \LL^\infty(\GG)\rightarrow \LL^\infty(\GG), \Xi^l(a)=(h\pi\otimes\id)\Delta^{u,r}(a)$ 
and $\Xi^r\colon \LL^\infty(\GG)\rightarrow \LL^\infty(\GG), \Xi^r(a)=(\id\otimes h\pi)\Delta^{r,u}(a)$. These maps $ \Xi^l, \Xi^r $ are both contractive, CP and map $\mrm{C}_0(\GG)$ to itself.
\end{remark}

From now on we will freely view $\Xi$ as an endomorphism of $\Linf,\M(\mrm{C}_0(\GG))$ or $\mrm{C}_0(\GG)$, but for the sake of clarity we shall sometimes indicate 
which version we are using. We remark that the proof of Lemma~\ref{averaginghelp} shows that $\Xi$ maps $\A(\GG)$ to itself; see Lemma~\ref{fourierhelp} below for a different 
way to verify this.

Let us write $ \mrm{C}_0(\KK\backslash \GG/\KK) \subseteq \mrm{C}_0(\GG) $ for the image of $ \mrm{C}_0(\GG) $ under $\Xi$.

\begin{lemma}\label{lemma1}
The following hold:
\begin{itemize}
\item $\mathrm{C}_0(\KK\backslash\GG/\KK)$ is a $\mrm{C}^*$-subalgebra of $\mrm{C}_0(\GG)$ and
\begin{align*}
\mathrm{C}_0(\KK\backslash\GG/\KK) &= \{a\in \mrm{C}_0(\GG) \mid (\pi\otimes \id\otimes \pi)\Delta^{u,r,u}(a)=\I\otimes a \otimes \I\} \\
&= \{ a\in\mrm{C}_0(\GG) \mid (\pi\otimes\id)\Delta^{u,r}(a) = \I\otimes a,
(\id\otimes\pi)\Delta^{r,u}(a) = a\otimes\I \}.
\end{align*}
\item
The map $ \Xi\colon \mrm{C}_0(\GG) \rightarrow \mrm{C}_0(\GG) $ is a conditional expectation onto $ \mrm{C}_0(\KK\backslash \GG/\KK) $. 
\end{itemize}
\end{lemma} 

\begin{proof}
Let $a\in \mrm{C}_0(\GG)$ so that $\Xi(a)\in \mrm{C}_0(\KK\backslash\GG/\KK)$. Then
\begin{align*}
(\pi\otimes\id)\Delta^{u,r}(\Xi(a))
&= (\pi\otimes\id)\Delta^{u,r}\big( (h\pi\otimes \id\otimes h\pi)\Delta^{u,r,u}(a) \big) \\
&= (h\pi\otimes \pi\otimes\id\otimes h\pi)\Delta^{u,u,r,u}(a)
= ((h\pi\otimes\pi)\Delta^u\otimes\id\otimes h\pi)\Delta^{u,r,u}(a) \\
&= ((h\otimes\id)\Delta_{\KK}\pi\otimes\id\otimes h\pi)\Delta^{u,r,u}(a) \\
&= \I \otimes (h\pi\otimes\id\otimes h\pi)\Delta^{u,r,u}(a)
= \I \otimes \Xi(a).
\end{align*}
An entirely analogous calculation shows that $(\id\otimes\pi)\Delta^{r,u}(\Xi(a)) = \Xi(a)\otimes \I$.  Conversely, let $a\in \mrm{C}_0(\GG)$ be such that $(\pi\otimes\id)\Delta^{u,r}(a) = \I\otimes a$ and $(\id\otimes\pi)\Delta^{r,u}(a) = a\otimes\I$.  Then
\[ (\pi\otimes\id\otimes\pi)\Delta^{u,r,u}(a) = ((\pi\otimes\id)\Delta^{u,r}\otimes\id)
\big( (\id\otimes\pi) \Delta^{r,u}(a) \big)
= ((\pi\otimes\id)\Delta^{u,r}\otimes\id) (a\otimes \I)
= \I \otimes a\otimes \I. \]
Now suppose that $a\in\mrm{C}_0(\GG)$ with $(\pi\otimes\id\otimes\pi)\Delta^{u,r,u}(a) = 
\I \otimes a\otimes \I$.  Then 
\[
a=(h\pi\otimes \id\otimes h\pi)\Delta^{u,r,u}(a)=\Xi(a).
\]
We have shown the stated forms for elements of $\mathrm{C}_0(\KK\backslash\GG/\KK)$. It follows easily that $\mrm{C}_0(\KK\backslash \GG/\KK)$ is a $\mrm{C}^*$-subalgebra.

It is clear from construction that $ \Xi $ is a contraction. For $a \in \mrm{C}_0(\KK\backslash \GG/\KK) $ we obtain, using the above obtained description,
\[
\Xi(a) = (h\pi \otimes \id \otimes h \pi)\Delta^{u,r,u}(a)=
(h\otimes \id\otimes h)(\I\otimes a\otimes \I)=a. 
\]
Hence $ \Xi $ is a contractive projection. From Tomiyama's theorem, \cite[Theorem 1.5.10]{BrownOzawa}, it follows that $\Xi$ is a conditional expectation.
\end{proof}

\begin{remark}\label{rem:onesided2}
We could analogously define $\mrm{C}_0(\KK \backslash \GG)$ to be the image of $(h\pi\otimes\id)\Delta^{u,r}$, and then the same argument shows that
\[ \mrm{C}_0(\KK \backslash \GG) = \{ a\in\mrm{C}_0(\GG) \,|\, (\pi\otimes\id)\Delta^{u,r}(a) = \I\otimes a \}. \]
Similarly, denote by $\mrm{C}_0(\GG / \KK)$ the image of $(\id\otimes h\pi)\Delta^{r,u}$, and then
\[ \mrm{C}_0(\GG / \KK) = \{ a\in\mrm{C}_0(\GG) \,|\, (\id\otimes\pi)\Delta^{r,u}(a) = a\otimes\I \}. \]
Hence we have shown that $\mathrm{C}_0(\KK\backslash\GG/\KK) = \mrm{C}_0(\KK \backslash \GG) \cap \mrm{C}_0(\GG / \KK)$.
\end{remark}

We now discuss some properties of slicing with functionals defined at the universal level. Given $\omega\in \mrm{C}_0^u(\GG)^*$ and $a\in \M(\mrm{C}_0(\GG))$ we write
\[
\omega\star a=(\id\otimes \omega)\Delta^{r,u}(a),\quad
a\star\omega=(\omega\otimes \id)\Delta^{u,r}(a).
\]
A priori both $\omega\star a$ and $a\star\omega$ belong to the multiplier algebra $\M(\mrm{C}_0(\GG))$, but as shown in Lemma~\ref{fourierhelp} below, 
we have in fact $\omega\star a,a\star\omega\in \mrm{C}_0(\GG)$ for any $a\in \mrm{C}_0(\GG)$. These operations turn $\mrm{C}_0(\GG)$ and $\M(\mrm{C}_0(\GG))$ into a 
bimodule over $\mrm{C}_0^u(\GG)^*$. 

By taking the Banach space dual of the map $\pi\colon \mrm{C}_0^u(\GG) \rightarrow \mrm{C}(\KK)$, we obtain a contractive algebra
homomorphism $\pi^*\colon \mrm{C}(\KK)^*\rightarrow \mrm{C}_0^u(\GG)^*$, with respect to the convolution product. In particular, $\pi^*(h)$ is an idempotent state in $\mrm{C}_0^u(\GG)^*$. Notice that the averaging map $\Xi$ can be written as $\Xi(a)=\pi^*(h)\star a \star \pi^*(h)$ for $a\in \M(\mrm{C}_0(\GG))$.

It is straightforward to check that $\A(\GG)$ and $\mrm{C}_0(\GG)$ are contractive left and right $\mrm{C}_0^u(\GG)^*$-modules, compare \cite[Proposition~4.5]{DKV_ApproxLCQG}, a result we now extend to the Fourier--Stieltjes algebras as defined in Section~\ref{section preliminaries}. As preparation, notice that the Banach space dual map $\Lambda_\GG^*\colon \mrm{C}_0(\GG)^* \rightarrow \mrm{C}_0^u(\GG)^*$ is an isometric algebra homomorphism since $\Lambda_\GG$ is a quotient map which intertwines the coproducts. Also recall that the canonical map $\Lj \rightarrow \mrm{C}_0(\GG)^*$ is an isometric algebra homomorphism which identifies $\Lj$ as an ideal in $\mrm{C}_0(\GG)^*$, see the discussion after \cite[Proposition~8.3]{KV_LCQGs}. The composition of these maps gives a canonical isometric homomorphism $\Lj \rightarrow \mrm{C}_0^u(\GG)^*$, whose image is again an ideal, see for example \cite[Proposition~8.3]{Daws_mults_si}. 

\begin{lemma} \label{fourierhelp}
The following hold:
\begin{enumerate}
\item[1)]
Let $a\in \A(\GG)$ and $\omega\in \mrm{C}_0^u(\GG)^*$. Then $a \star \omega \in\A(\GG)$ and $\|a \star \omega\|_{\A(\GG)} \leq \|a \|_{\A(\GG)}\|\omega\| $. Similarly, $\omega \star a \in \A(\GG)$ 
and $\|\omega \star a\|_{\A(\GG)} \leq \|a \|_{\A(\GG)}\|\omega\| $. 
\item[2)]
Let $b\in \B_r(\GG)$ and $\omega\in \mrm{C}_0^u(\GG)^*$. Then $b \star \omega \in\B_r(\GG)$ and $\|b \star \omega\|_{\B_r(\GG)} \leq \|b \|_{\B_r(\GG)}\|\omega\| $. Similarly, $\omega \star b \in \B_r(\GG)$ 
and $\|\omega \star b\|_{\B_r(\GG)} \leq \|b \|_{\B_r(\GG)}\|\omega\| $. 
\item[3)]
Let $b\in \B(\GG)$ and $\omega\in \mrm{C}_0^u(\GG)^*$. Then $b \star \omega \in\B(\GG)$ and $\|b \star \omega\|_{\B(\GG)} \leq \|b \|_{\B(\GG)}\|\omega\| $. Similarly, $\omega \star b \in \B(\GG)$ 
and $\|\omega \star b\|_{\B(\GG)} \leq \|b \|_{\B(\GG)}\|\omega\| $. 
\item[4)] Let $c\in \mrm{C}_0(\GG)$ and $\omega\in \mrm{C}_0^u(\GG)^*$. Then $c \star \omega \in\mrm{C}_0(\GG)$ and $\|c \star \omega\| \leq \|c \|\|\omega\| $. Similarly, $\omega \star c \in \mrm{C}_0(\GG)$ 
and $\|\omega \star c\| \leq \|c \|\|\omega\| $.
\end{enumerate}
\end{lemma} 

\begin{proof} 
$2)$ Write $b=\wh{\lambda}(\wh\omega)\in \B_r(\GG)$ for $\wh\omega\in \mrm{C}_0(\whG)^*$. Then
\begin{equation}\begin{split}\label{eq10a}
b\star\omega&=(\omega\otimes\id)\Delta^{u,r}\bigl((\wh\omega\otimes\id){\ww}^{\whG}\bigr)=(\wh\omega \otimes \omega \otimes\id)\bigl({\ww}^{\whG}_{13}{\wW}^{\whG}_{12}\bigr)\\
&=(\wh\omega\otimes\id)\bigl({\ww}^{\whG} ((\id\otimes\omega)\wW^{\whG}\otimes \I)\bigr)=
\wh{\lambda}\bigl(\wh\omega\bigl(\cdot\,
(\id\otimes\omega)\wW^{\whG}\bigr)\bigr)\in \B_r(\GG)
\end{split}\end{equation}
as required. We also immediately get the bound on the $\B_r(\GG)$-norm of $b\star\omega$:
\[
\|b\star \omega\|_{\B_r(\GG)}=\bigl\|\wh\omega\bigl( \cdot \,(\id\otimes \omega)\wW^{\whG}\bigr)\bigr\|\le \|\wh\omega\|\|\omega\|=
\|b\|_{\B_r(\GG)}\|\omega\|.
\]
Similarly, we get 
\begin{equation}\begin{split}\label{eq10b}
\omega\star b&=(\id \otimes\omega)\Delta^{r,u}\bigl((\wh\omega\otimes\id){\ww}^{\whG}\bigr)=
(\wh\omega \otimes \id \otimes\omega)\bigl({\wW}^{\whG}_{13}{\ww}^{\whG}_{12}\bigr)\\
&=(\wh\omega\otimes\id)\bigl(((\id\otimes\omega)\wW^{\whG}\otimes \I){\ww}^{\whG} \bigr)=
\wh{\lambda}\bigl(\wh\omega\bigl(
(\id\otimes\omega)\wW^{\whG}\;\cdot \bigr)\bigr)\in \B_r(\GG),
\end{split}\end{equation}
and $\|\omega\star b\|_{\B_r(\GG)}\le \|\omega\| \|b\|_{\B_r(\GG)}$. This proves the second point. 

$1)$ If $a=\wh{\lambda}(\wh\omega)$ for $\wh\omega\in \LL^1(\whG)\subseteq \mrm{C}_0(\whG)^*$, then since $\wh{\omega}\bigl( \cdot\,(\id\otimes \omega)\wW^{\whG}\bigr)$ and $\wh{\omega}\bigl( (\id\otimes \omega)(\wW^{\whG})\,\cdot \bigr)$ are normal functionals, the above reasoning proves also the first point.

$3)$ This is completely analogous: if $b=(\id\otimes \wh\omega)(\wW^{\GG *})\in \B(\GG)$ with $\wh\omega\in \mrm{C}_0^u(\whG)^*$, then
\begin{equation}\begin{split}\label{eq10c}
b\star\omega&=(\omega\otimes\id)\Delta^{u,r}\bigl((\wh\omega\otimes\id){\Ww}^{\whG}\bigr)=(\wh\omega \otimes \omega \otimes\id)\bigl({\Ww}^{\whG}_{13}{\WW}^{\whG}_{12}\bigr)\\
&=(\wh\omega\otimes\id)\bigl({\Ww}^{\whG} ((\id\otimes\omega)\WW^{\whG}\otimes \I)\bigr)=
\bigl(\id\otimes \wh\omega\bigl(\cdot\,
(\id\otimes\omega)\WW^{\whG}\bigr)\bigr)(\wW^{\GG *})\in \B(\GG)
\end{split}\end{equation}
and $\|b\star\omega\|_{\B(\GG)}\le \|b\|_{\B(\GG)}\|\omega\|$. The argument for $ \omega \star b $ is similar.

$4)$ Since $\A(\GG)\subseteq \mrm{C}_0(\GG)$ is norm dense, we can approximate $c$ by a sequence $(a_n)_{n\in\NN}$ of elements from the Fourier algebra $\A(\GG)$. As the map $\mrm{C}_0(\GG)\ni c'\mapsto c'\star \omega\in \M(\mrm{C}_0(\GG))$ is continuous, we obtain 
\[
c\star \omega=(\lim_{n\to\infty} a_n)\star \omega=
\lim_{n\to\infty} a_n\star\omega\in \ov{\A(\GG)}=\mrm{C}_0(\GG).
\]
The norm estimate follows as $\|c\star\omega\| = \|(\id\otimes\omega)\Delta^{r,u}(c)\| \leq \|\omega\| \| \Delta^{r,u}(c) \| \leq \|\omega\| \|c\|$. Again, the argument for $\omega\star c$ is entirely analogous.
\end{proof} 

\begin{corollary}\label{corr:xi_to_fourier}
The averaging map $\Xi$ restricts to contractive linear maps $\A(\GG) \rightarrow \A(\GG)$, $\B_r(\GG)\rightarrow \B_r(\GG)$ and $\B(\GG)\rightarrow \B(\GG)$. 
Furthermore, the maps $\B_r(\GG)\rightarrow \B_r(\GG)$ and $\B(\GG)\rightarrow \B(\GG)$ are weak$^*$-weak$^*$-continuous.
\end{corollary}

\begin{proof}
As $ \Xi(a) = \pi^*(h) \star a \star \pi^*(h) $, and $\|\pi^*(h)\| = \|h\| = 1$, this follows immediately from Lemma~\ref{fourierhelp} together with equations \eqref{eq10a}, \eqref{eq10b}, \eqref{eq10c} whose form shows weak$^*$-weak$^*$-continuity of $\Xi$.
\end{proof}

We write $\A(\KK\backslash\GG/\KK) \subseteq \A(\GG)$ (resp.~$\B_r(\KK\backslash \GG/\KK)\subseteq \B_r(\GG)$ and $\B(\KK\backslash \GG/\KK)\subseteq \B(\GG)$) for the image of $\A(\GG)$ (resp.~$\B_r(\GG)$, $\B(\GG)$) under $\Xi$.

As $\A(\GG) \cong \LL^1(\wh\GG)$, the averaging map induces a map $\Xi_1\colon \LL^1(\wh\GG) \rightarrow \LL^1(\wh\GG)$, and taking the Banach space adjoint, a normal map $\Xi_\infty\colon \LL^\infty(\wh\GG) \rightarrow \LL^\infty(\wh\GG)$.  We now explore what these maps are.  Let $p_e\in\ell^\infty(\wh\KK)$ denote the central projection onto the matrix factor corresponding to the trivial representation $e\in\Irr(\KK)$.

\begin{lemma}\label{lem:xi1xiinf}
The maps $\Xi_1\colon \LL^1(\wh\GG) \rightarrow \LL^1(\wh\GG)$ and $\Xi_\infty\colon \LL^\infty(\wh\GG) \rightarrow \LL^\infty(\wh\GG)$ are given by
\[ \Xi_1(\wh\omega) = \wh\omega(\wh\pi(p_e)\cdot\wh\pi(p_e)) \quad(\wh\omega\in\LL^1(\wh\GG)),
\qquad \Xi_\infty(\wh x) = \wh\pi(p_e) \wh x \wh\pi(p_e) \quad(\wh x\in\LL^\infty(\wh\GG)),\]
respectively. 
\end{lemma}
\begin{proof}
As in the proof of Corollary~\ref{corr:xi_to_fourier}, we know that for $a\in\A(\GG)$ we have that $\Xi(a) = \pi^*(h) \star a \star \pi^*(h)$.  Then, given $a = \wh\lambda(\wh\omega)$ for some $\wh\omega\in\LL^1(\wh\GG)$, by the proof of Lemma~\ref{fourierhelp} we know that
\[ \pi^*(h) \star a \star \pi^*(h) = \wh\lambda\big(
\wh\omega\big( (\id\otimes\pi^*(h))\wW^{\wh\GG} \cdot (\id\otimes\pi^*(h))\wW^{\wh\GG} \big)
\big). \]
As $(\pi\otimes\id)(\Ww^{\GG}) = (\id\otimes\wh\pi)(\ww^{\KK})$, equivalently, $(\id\otimes\pi)(\wW^{\wh\GG}) = (\wh\pi\otimes\id)(\ww^{\wh\KK})$, it follows that $(\id\otimes\pi^*(h))\wW^{\wh\GG} = \wh\pi\big( (\id\otimes h)\ww^{\wh\KK} \big) = \wh\pi(p_e)$.  Hence
\[ \Xi(\wh\lambda(\wh\omega)) = \Xi(a) = \wh\lambda(\wh\omega(\wh\pi(p_e)\cdot\wh\pi(p_e))), \]
and the stated formula for $\Xi_1$ follows.  The formula for $\Xi_\infty$ follows by direct calculation.
\end{proof}

Let us next discuss the compatibility of $\Xi $ with CB multipliers. 
As usual, given $\omega\in \mrm{C}^u_0(\GG)^*$ we denote by $\overline\omega$ the functional defined by $\overline{\omega}(x)=\overline{\omega(x^*)}$ for $x\in \mrm{C}_0^u(\GG)$. 
Recall from \cite[Section~7.9]{KustermansUniversal} that the scaling automorphism group and the (unitary) antipode admit lifts to the universal level $\mrm{C}_0^u(\GG)$, 
which we denote by $(\tau_t^u)_{t\in \RR}, R_u$, and $S_u$, respectively. In the same way as at the reduced level, they are connected via the formula $S_u=R_u \tau^u_{-i/2}$.

Whenever $\omega\in\mrm{C}_0^u(\GG)^*$ is such that the map $D(S_u)\ni a \mapsto 
\ov{\la \omega , S_u(a)^*\ra }\in \CC$ extends to a bounded functional on $\mrm{C}_0^u(\GG)$, we denote this extension by $\omega^\sharp$. This is in agreement with the usual definition 
of $\LL^1_{\sharp}(\GG)\subseteq\Lj$, see \cite[Definition 2.3]{KVVN}.

\begin{lemma}
We have $\ov{\pi^*(h)}=\pi^*(h)=\pi^*(h)^{\sharp}$.
\end{lemma}
\begin{proof}
The first equality is immediate since $\pi$ is a $\star$-homomorphism and $h$ is a state. For the second, recall that $\pi$ commutes with the 
antipode, see \cite[Proposition 3.10]{MRW_Hms_QGs}, and $h S_{\KK}=h$ on $D(S_{\KK})$.
\end{proof}

We wish to study convolution of multipliers by elements of $\mrm{C}_0^u(\GG)^*$. To avoid developing some theory around applying completely bounded maps to multiplier algebras, we shall use von Neumann algebra techniques. By \cite[Proposition~3.12.3]{Pedersen_Book}, for example, for a \cst-algebra $\mc A$ contained in a von Neumann algebra $\M$ we can identify $\M(\mc A)$ with $\{ x\in\M \,|\, xa,ax\in\mc A \ (a\in\mc A)\}$.  In particular, we will view the \cst-algebra $\mrm{C}_0^{u}(\GG)\otimes \mrm{C}_0(\whG)$ as being contained in the von 
Neumann algebra $\mrm{C}^{u}_0(\GG)^{**}\bar{\otimes} \Linfd$, and so consider $\Ww^{\GG}$ as an element of $\mrm{C}^{u}_0(\GG)^{**}\bar{\otimes} \Linfd$.

\begin{lemma} \label{cbhelp}
Let $a\in \M_{cb}^l(\A(\GG))$ and $\omega\in \mrm{C}_0^u(\GG)^*$. 
\begin{enumerate} 
\item[1)] We have $ a \star \omega \in \M_{cb}^l(\A(\GG)) $ and $ \|a \star \omega\|_{cb} \leq \|a\|_{cb} \|\omega\| $.  Furthermore $\Theta^l(a\star\omega)$ is the map
\begin{equation}\label{eq2}
\Linfd\ni \wh x \mapsto 
(\omega\otimes \id)\bigl( (\id\otimes \Theta^l(a)) ( (\I\otimes \wh x)\Ww^{\GG *})\Ww^{\GG }\bigr)\in \Linfd. \end{equation}
\item[2)] Take $\omega\in \mrm{C}_0^u(\GG)^*$ and assume that $\ov{\omega}^{\sharp}$ exists. Then $\overline{ \overline{\omega}^\sharp } \star a \in \M_{cb}^l(\A(\GG))$ 
and $\| \overline{ \overline{\omega}^\sharp } \star a \|_{cb} \leq \|\omega\|\|a\|_{cb} $.
Furthermore, $\overline{ \overline{\omega}^\sharp } \star a = S^{-1}(S^{-1}(a)^*\star\omega)^*
$ and $\Theta^l( \overline{ \overline{\omega}^\sharp } \star a )$ is the map
\begin{equation}\label{eq2a}
\Linfd\ni \wh x \mapsto (\ov\omega\otimes \id)\bigl( \Ww^{\GG *} (\id\otimes \Theta^l(a)) ( \Ww^{\GG}(\I\otimes \wh x) )\bigr)
\in \Linfd. \end{equation}
\end{enumerate}
\end{lemma}

\begin{proof}
$ 1) $ This claim is analogous to \cite[Proposition~4.5]{DKV_ApproxLCQG}.  The functional $\omega$ may be regarded as a normal functional on $\mrm{C}_0^u(\GG)^{**}$. 
Therefore, given the discussion above, it makes sense to define
\[  T\colon \Linfd\ni \wh x \mapsto 
(\omega\otimes \id)\bigl( (\id\otimes \Theta^l(a)) ( (\I\otimes \wh x)\Ww^{\GG *})\Ww^{\GG }\bigr)\in \Linfd. \]
Observe that $\|T\|_{cb}\le \|a\|_{cb}\|\omega\|$. Now one can proceed exactly as in \cite[Proposition 4.5]{DKV_ApproxLCQG} to show that $T$ is a module map on $\Linfd$ and that the associated multiplier is $a\star\omega\in \M^l_{cb}(\A(\GG))$ with $\Theta^l(a\star\omega)=T$.

$ 2) $ It follows from \cite[Lemma 4.8, Proposition~4.9]{DKV_ApproxLCQG} that $a\in D(S^{-1})$ and that $b = S^{-1}(a)^* $ is contained in $\M_{cb}^l(\A(\GG))$ with $\|b\|_{cb} = \|a\|_{cb}$. The previous point gives $b\star \omega\in \M^l_{cb}(\A(\GG))$, hence also $b\star\omega\in D(S^{-1})$.  We claim that 
\[
S^{-1}( b\star\omega)=\ov{\omega}^{\sharp}\star S^{-1}(b), 
\]
compare \cite[Lemma 4.11]{Daws_Bohr}.
Indeed, using \cite[Propositions 7.2, 9.2]{KustermansUniversal} one easily checks that $\chi (R_u\otimes R)\Delta^{u,r}=\Delta^{u,r} R$ and $(\tau^u_t\otimes \tau_t)\Delta^{u,r}=\Delta^{u,r} \tau_t\,(t\in\RR)$. Then
\[\begin{split}
&\quad\;
S^{-1}(b\star \omega)=
R \tau_{i/2}\bigl((\omega\otimes \id)\Delta^{u,r}(b)\bigr)=
R\bigl( (\omega\tau^u_{-i/2}\otimes \id)\Delta^{u,r}(\tau_{i/2}(b))\bigr)\\
&=
(\id\otimes \omega\tau^u_{-i/2} R^u)\Delta^{r,u} (R\tau_{i/2}(b))=
(\id\otimes \ov{\omega}^{\sharp})\Delta^{r,u}( S^{-1}(b))=
\ov{\omega}^{\sharp}\star S^{-1}(b).
\end{split}\]
In the above calculation we used the fact that $\omega S_u=\ov{\omega}^{\sharp}$ on $D(S_u)$, in particular $\omega\tau^u_{-i/2}$ is bounded. Using again the previous point and \cite[Lemma 4.8, Proposition 4.9]{DKV_ApproxLCQG} we conclude that
\[ \ov{\ov{\omega}^{\sharp}} \star a
= \ov{\ov{\omega}^{\sharp}} \star S^{-1}(b)^*
= \bigl( \ov{\omega}^{\sharp}\star S^{-1}(b)\bigr)^*
= S^{-1}(b\star\omega)^*
\in  \M^l_{cb}(\A(\GG)). \]
Furthermore, we now see that
\[ \big\| \ov{\ov{\omega}^{\sharp}} \star a \big\|_{cb}
= \big\| S^{-1}(b\star\omega)^* \big\|_{cb}
= \big\| b\star\omega \big\|_{cb}
\leq \|b\|_{cb} \|\omega\|
= \|a\|_{cb} \|\omega\|, \]
as claimed.  Finally, by \cite[Proposition 4.9]{DKV_ApproxLCQG} we know that $\Theta^l(b) = \Theta^l(S(a^*)) = \Theta^l(a)^\dagger$.  Thus, for $\wh x \in \Linfd$,
\begin{align*}
\Theta^l( \overline{ \overline{\omega}^\sharp } \star a )(\wh x)
&= \Theta^l\big( S^{-1}(b\star\omega)^* \big)(\wh x)
= \Theta^l\big( b\star\omega \big)(\wh x^*)^* \\
&= (\omega\otimes \id)\bigl( (\id\otimes \Theta^l(b)) ( (\I\otimes \wh x^*)\Ww^{\GG *})\Ww^{\GG }\bigr)^*,
\end{align*}
using \eqref{eq2}. This is then equal to
\begin{align*}
& (\ov\omega\otimes \id)\bigl( \Ww^{\GG *} (\id\otimes \Theta^l(b)^\dagger) ( \Ww^{\GG}(\I\otimes \wh x) )\bigr)
= (\ov\omega\otimes \id)\bigl( \Ww^{\GG *} (\id\otimes \Theta^l(a)) ( \Ww^{\GG}(\I\otimes \wh x) )\bigr)
\end{align*}
as claimed.
\end{proof}

We are now in a position to show that the averaging procedure with respect to a compact quantum subgroup is compatible with multipliers. 

\begin{proposition} \label{averagingcb}
Let $\GG$ be a locally compact quantum group with a compact quantum subgroup $\KK\subseteq\GG$. Then the averaging map $ \Xi\colon \M(\mrm{C}_0(\GG)) \rightarrow \M(\mrm{C}_0(\GG))$ 
restricts to a contractive weak$^*$-weak$^*$-continuous map $\M_{cb}^l(\A(\GG)) \rightarrow \M_{cb}^l(\A(\GG))$. 
\end{proposition} 

\begin{proof} 
Recall that $ \Xi(a) = \pi^*(h) \star a \star \pi^*(h) $ for $a\in \M(\mrm{C}_0(\GG))$. Since $\ov{\pi^*(h)}=\pi^*(h)=\pi^*(h)^{\sharp}$, Lemma~\ref{cbhelp} tells us that $\Xi$ restricts to a contractive linear map $\Xi|_{\M^l_{cb}(\A(\GG))}\colon \M_{cb}^l(\A(\GG)) \rightarrow \M_{cb}^l(\A(\GG))$. 
To verify weak$^*$-weak$^*$-continuity it is enough to show that $(\Xi|_{\M^l_{cb}(\A(\GG))})^*$ preserves $\Lj\subseteq Q^l(\A(\GG))$. Recall the isometric homomorphism $\phi\colon \Lj \rightarrow \mrm{C}_0^u(\GG)^*$ which identifies $\Lj$ as a closed ideal in $\mrm{C}_0^u(\GG)^*$. As the duality between $Q^l(\A(\GG))$ and $\M_{cb}^l(\A(\GG))$ is naturally compatible 
with the duality between $\Lj $ and $\Linf$, it suffices to show that $\Xi^*(\Lj) \subseteq \Lj$. Observe that for $b\in \mrm{C}_0^u(\GG),\omega\in\Lj$ we have
\[ \begin{split}
&\quad\;
\Lambda_{\GG}^*(\Xi^*(\omega))(b)=\Xi^*(\omega) (\Lambda_{\GG}(b))=
\omega ( \pi^*(h)\star \Lambda_{\GG}(b)\star \pi^*(h))\\
&=
(\pi^*(h)\otimes \omega\otimes \pi^*(h))\Delta^{u,r,u}\Lambda_{\GG}(b)= (\pi^*(h)\star \phi(\omega)\star\pi^*(h))(b).
\end{split}\]
Consequently, $\Lambda_{\GG}^*(\Xi^*(\omega))=\phi(\omega')$ for some $\omega'\in \Lj$, and since $\Lambda_{\GG}$ is surjective this shows $\Xi^*(\omega)=\omega'$ as required.
\end{proof}

\begin{remark}\label{rem:onesided3}
Exactly the same arguments show that the one-sided averaging maps $ \Xi^l, \Xi^r $ restrict to maps on $\A(\GG), \B_r(\GG)$ and $\B(\GG)$, and also to $\M_{cb}^l(\A(\GG))$.
\end{remark}

We shall write $\M_{cb}^l(\A(\KK\backslash\GG/\KK))$ for the image of $\Xi\colon\M_{cb}^l(\A(\GG)) \rightarrow \M_{cb}^l(\A(\GG))$, 
and refer to the elements of this space as $\KK$-biinvariant CB multipliers of $\GG$. Proposition~\ref{averagingcb}, combined with the fact that $\Xi$ is a projection,
shows that $\M_{cb}^l(\A(\KK\backslash\GG/\KK)) \subseteq \M_{cb}^l(\A(\GG))$ is a weak$^*$-closed subspace. We now study further properties of the map $\Xi$ on 
the level of multipliers.

\begin{proposition} \label{cppreservation}
The averaging map $ \Xi\colon \M_{cb}^l(\A(\GG)) \rightarrow \M^l_{cb}(\A(\GG)) $ maps CP multipliers to CP multipliers. 
Furthermore, for $a\in \M^l_{cb}(\A(\GG))$ and $\wh{x} \in \Linfd$, we have
\[
\Theta^l(\Xi(a))(\wh x) =
(\pi^*(h)\otimes \pi^*(h)\otimes \id)
\bigl(
\Ww^{\GG *}_{13}
(\id\otimes \id\otimes \Theta^l( a))(
\Ww^{\GG }_{13}(\I\otimes \I\otimes \wh{x})\Ww^{\GG *}_{23}
)\Ww^{\GG}_{23}
\bigr)\in\Linfd. \]
\end{proposition} 

\begin{proof}
Let $a\in \M^l_{cb}(\A(\GG))$ be arbitrary. To ease notation, write $h_0 = \pi^*(h)$ and $b = a\star h_0$. 
Since $\Xi(a) = h_0 \star (a \star h_0) = h_0 \star b$ and $h_0 = \ov{h_0}=h_0^\sharp$, formula \eqref{eq2a} from Lemma~\ref{cbhelp} shows that
\begin{align*}
\Theta^l(\Xi(a))(\wh x)
&= (h_0\otimes\id)\bigl( \Ww^{\GG *} (\id\otimes \Theta^l(b)) ( \Ww^{\GG}(\I\otimes \wh x) )\bigr) \qquad (\wh x \in \Linfd).
\end{align*}
By \eqref{eq2} we know that $\Theta^l(b)$ is the map $\wh y \mapsto (h_0\otimes\id)\bigl( (\id\otimes \Theta^l(a)) ( (\I\otimes \wh y)\Ww^{\GG *})\Ww^{\GG }\bigr)$.  Thus
\begin{align*}
(\id\otimes \Theta^l(b)) ( \Ww^{\GG}(\I\otimes \wh x) )
= (\id\otimes h_0\otimes\id)\big( (\id\otimes\id\otimes\Theta^l(a))( \Ww^{\GG}_{13}(\I\otimes\I\otimes \wh x) \Ww^{\GG *}_{23} ) \Ww^{\GG}_{23} \big),
\end{align*}
and hence
\begin{align*}
\Theta^l(\Xi(a))(\wh x)
&= (h_0\otimes\id)\bigl( \Ww^{\GG *} 
(\id\otimes h_0\otimes\id)\big( (\id\otimes\id\otimes\Theta^l(a))( \Ww^{\GG}_{13}(\I\otimes\I\otimes \wh x) \Ww^{\GG *}_{23} ) \Ww^{\GG}_{23} \big)
\bigr) \\
&= (h_0\otimes h_0\otimes\id)\big(\Ww^{\GG *}_{13} (\id\otimes\id\otimes\Theta^l(a))( \Ww^{\GG}_{13}(\I\otimes\I\otimes \wh x) \Ww^{\GG *}_{23} ) \Ww^{\GG}_{23} \big).
\end{align*}
This proves the second claim.

Assume now that $a$ is a CP multiplier, i.e.~$\Theta^l(a)\in \CB^\sigma(\Linfd)$ is CP. By Stinespring's theorem and the classification of normal $*$-homomorphisms (see for 
example \cite[Theorem I.4.4.3]{DixmiervNA}), we can find a Hilbert space $\mathscr{H}$ and $r\in \B(\LdG,\mathscr{H}\otimes \LdG)$ such 
that $\Theta^l(a)(\wh x)= r^* (\I\otimes \wh{x})r$. Consequently, we can further write
\begin{align*}
&\quad\;
\Theta^l(\Xi(a))(\wh x)\\
&=
(h_0\otimes h_0\otimes \id)
\bigl(
\Ww^{\GG *}_{13}
(\id\otimes \id\otimes \Theta^l( a))(
\Ww^{\GG }_{13}(\I\otimes \I\otimes \wh{x})\Ww^{\GG *}_{23}
)\Ww^{\GG}_{23}
\bigr)\\
&=
(h_0\otimes h_0\otimes \id)
\bigl(
\Ww^{\GG *}_{13}
(\I\otimes \I\otimes r^*)
\Ww^{\GG }_{14}(\I\otimes \I\otimes \I\otimes \wh{x})\Ww^{\GG *}_{24}
(\I\otimes \I\otimes r)\Ww^{\GG}_{23}
\bigr)
\end{align*}
Here we are regarding $\Ww^{\GG}$ as a member of $\mrm{C}_0^u(\GG)^{**} \bar\otimes \Linfd$.  Pick some Hilbert space $\mathscr{K}$ and a universal representation $\mrm{C}_0^u(\GG) \subseteq \B(\mathscr{K})$ so that $\mrm{C}_0^u(\GG)^{**} = \mrm{C}_0^u(\GG)'' \subseteq \B(\mathscr{K})$.  By choosing  $\mrm{C}_0^u(\GG) \subseteq \B(\mathscr{K})$ in a suitable way, 
we may suppose that there is a unit vector $\xi\in\mathscr{K}$ such that $\omega_\xi$ restricted to $\mrm{C}_0^u(\GG)''$ agrees with the normal extension of $h_0$.  Then $\Ww^{\GG}$ can be regarded as a member of $\B(\mathscr{K}\otimes \LdG)$, and so, for example, $\Ww^{\GG *}_{24}(\I\otimes \I\otimes r)\Ww^{\GG}_{23}$ is a member of $\B(\mathscr{K}\otimes\mathscr{K}\otimes\LdG, \mathscr{K}\otimes\mathscr{K}\otimes\mathscr{H}\otimes\LdG)$. 
Thus, for $\eta_1,\eta_2\in \LdG$,
\begin{align*}
&\quad\; \big( \eta_1 \big\vert \Theta^l(\Xi(a))(\wh x) (\eta_2) \big)  \\
&= \big( \xi\otimes\xi\otimes\eta_1 \big\vert \Ww^{\GG *}_{13}
(\I\otimes \I\otimes r^*)
\Ww^{\GG }_{14}(\I\otimes \I\otimes \I\otimes \wh{x})\Ww^{\GG *}_{24}
(\I\otimes \I\otimes r)\Ww^{\GG}_{23}
(\xi\otimes\xi\otimes\eta_2) \big)  \\
&= \big( \Ww^{\GG *}_{14} (\I\otimes \I\otimes r) \Ww^{\GG}_{13} (\xi\otimes\xi\otimes\eta_1)  \big\vert 
\xi \otimes (\I\otimes \I\otimes \wh{x})\Ww^{\GG *}_{13}
(\I\otimes r)\Ww^{\GG} (\xi\otimes\eta_2) \big)  \\
&= \big( \Sigma_{12}\big( \xi \otimes \Ww^{\GG *}_{13} (\I\otimes r) \Ww^{\GG} (\xi\otimes\eta_1) \big)  \big\vert 
\xi \otimes (\I\otimes \I\otimes \wh{x})\Ww^{\GG *}_{13}
(\I\otimes r)\Ww^{\GG} (\xi\otimes\eta_2) \big)
\end{align*}
where $\Sigma \in \B(\mathscr{K}\otimes\mathscr{K})$ is the tensor swap map.  Let $\tilde r \in \B(\LdG, \mathscr{H}\otimes\LdG)$ be the operator which satisfies
\[ \big( \alpha \big\vert \tilde r (\beta) \big)
= \big( \xi\otimes\alpha \big\vert \Ww^{\GG *}_{13} (\I\otimes r) \Ww^{\GG}(\xi\otimes\beta) )
\qquad (\alpha\in \mathscr{H}\otimes\LdG, \beta\in \LdG). \]
Thus
\begin{align*}
&\quad\; \big( \eta_1 \big\vert \Theta^l(\Xi(a))(\wh x) (\eta_2) \big)
= ( \tilde r(\eta_1) \vert (\I\otimes\wh x) \tilde r(\eta_2) ),
\end{align*}
and so $\Theta^l(\Xi(a))(\wh x) = \tilde r^* (\I\otimes\wh x)\tilde r$ is CP, as claimed.
\end{proof}

\begin{remark}\label{rem:onesided4}
In general, the one-sided averaging maps $ \Xi^l, \Xi^r $ do not map CP multipliers to CP multipliers, as we now show. 

As in Remark~\ref{rem:onesided3}, the map $\Xi^l$ restricts to $\A(\GG)$, and so induces a map $\Xi^l_1 \colon \LL^1(\wh\GG) \to \LL^1(\wh\GG)$. Moreover, given $a = \wh\lambda(\wh\omega)\in \A(\GG)$ for $\wh\omega\in\LL^1(\wh\GG)$, we have $\Xi^l(a) = (h\pi\otimes\id)\Delta^{u,r}(a) = a \star (h\pi) = \wh\lambda(\wh\omega(\cdot \wh\pi(p_e)))$, as in the proof of Lemma~\ref{lem:xi1xiinf}.
Thus $\Xi^l_1(\wh\omega) = \wh\omega(\cdot \wh\pi(p_e))$, and similarly $\Xi^r_1(\wh\omega) = \wh\omega(\wh\pi(p_e)\cdot)$, for $\wh\omega\in \LL^1(\wh\GG)$.

Suppose $\wh\omega(\wh x \wh\pi(p_e)) \geq 0$ for all positive $\wh\omega\in \LL^1(\wh\GG)$ and all positive $\wh x\in \LL^\infty(\GG)$. Then $\wh x \wh\pi(p_e)$ is positive, 
hence self-adjoint, so $\wh x \wh\pi(p_e) = \wh\pi(p_e)\wh x$ for all positive $\wh x\in \LL^\infty(\GG)$. It follows that $\wh\pi(p_e)$ is central.
Hence if $\wh\pi(p_e) \in \LL^\infty(\wh\GG)$ is not central there exists a positive $\wh\omega\in \LL^1(\wh\GG)$ such that $\wh\omega(\cdot \wh\pi(p_e))$ not positive.  
Then Lemma~\ref{lem:cp_fourier_pos} shows that $a = \wh\lambda(\wh\omega)\in \A(\GG)$ is a CP multiplier but $\Xi^l(a) = \wh\lambda(\Xi^l_1(\wh\omega))$ is not CP. 

Note that $\wh\pi(p_e)$ is central if and only if $\Xi^l $ agrees with the two-sided averaging map $\Xi$, compare Lemma~\ref{lem:xi1xiinf}. Combining the above  
argument with Proposition \ref{cppreservation}, we therefore conclude that $\Xi^l$ maps CP multipliers to CP multipliers if and only if $\wh\pi(p_e)$ is central. 
The same statement, with a similar proof, holds for $\Xi^r$.

Centrality of $\wh\pi(p_e)$ already fails for classical locally compact groups if the compact subgroup is not normal. The case of Drinfeld doubles, explored in the next section, 
gives further examples where $\wh\pi(p_e)$ is not central.
\end{remark}

\section{Drinfeld doubles of discrete quantum groups} \label{section double}

In this section $\GGamma$ is a discrete quantum group. We shall use the following simplified notation: $\Delta,\wh\Delta$ are the coproducts 
on $\ell^{\infty}(\GGamma),\LL^{\infty}(\wh\GGamma)$ respectively, $h\in \LL^1(\wh\GGamma)$ is the Haar integral, $ \ww=\ww^{\GGamma}, \wh{\ww}=\ww^{\wh\GGamma}$ 
and $\ad(\ww)(x) = \ww x \ww^*$. 

The Drinfeld double $D(\GGamma)$ of $\GGamma$ is given by $ \LL^\infty(D(\GGamma)) = \ell^\infty(\GGamma) \bar{\otimes} \LL^\infty(\wh\GGamma) $ with the coproduct 
$$
\Delta_{D(\GGamma)} = (\id \otimes \chi \otimes \id)(\id \otimes \ad(\ww) \otimes \id)(\Delta \otimes \wh{\Delta}), 
$$
compare \cite[Section~8]{Doublecrossed}. Since discrete quantum groups and their compact duals are regular, it follows from \cite[Proposition 9.5]{Doublecrossed} that
\begin{equation}\label{eq11}
\mrm{C}_0(D(\GGamma))=\mrm{c}_0(\GGamma)\otimes \mrm{C}(\wh\GGamma).
\end{equation}
The left and right Haar integral on $D(\GGamma)$ is given by $\vp\otimes h$, see \cite[Theorem 5.3, Proposition 8.1]{Doublecrossed}.  The Kac--Takesaki operator of $D(\GGamma)$ is 
\[
\ww^{D(\GGamma)}=\ww_{13} Z_{34}^* \wh{\ww}_{24}Z_{34},
\]
where $Z=\ww (J\otimes \wh{J}) \ww (J\otimes \wh{J})$ and $J,\wh{J}$ are the modular conjugations on $\ell^{\infty}(\GGamma),\LL^{\infty}(\wh{\GGamma})$, 
respectively, see \cite[Proposition 8.1]{Doublecrossed}.
The Drinfeld double $D(\GGamma)$ contains $\wh\GGamma$ naturally as a compact (open) quantum subgroup, with the 
morphism $\pi\colon  \LL^\infty(D(\GGamma)) = \ell^\infty(\GGamma) \bar{\otimes} \LL^\infty(\wh\GGamma) \rightarrow \LL^\infty(\wh\GGamma) $ 
given by $\eps \otimes \id$, where $\eps \in \ell^1(\GGamma)$ is the counit. Note that since $\wh\GGamma$ is open, the morphism $\pi$ exists at the reduced level, 
which simplifies the situation compared to the previous section.
 
We shall consider $ \GG = D(\GGamma) $ and $ \KK = \widehat{\GGamma}$. Since $ \pi^*(h) = \eps \otimes h $ is a normal functional 
on $\LL^{\infty}(D(\GGamma))$, one easily sees that in this setting the averaging map $\Xi$ defined in equation \eqref{eq4} is given by 
\begin{equation}\label{eq5}
\Xi\colon \LL^{\infty}(D(\GGamma))\ni a \mapsto
(\pi^*(h)\otimes \id\otimes \pi^*(h))\Delta^{(2)}_{D(\GGamma)}(a)\in \LL^{\infty}(D(\GGamma)),
\end{equation}
where $\Delta^{(2)}_{D(\GGamma)}\colon \LL^{\infty}(D(\GGamma))\rightarrow \LL^{\infty}(D(\GGamma))^{\bar{\otimes} 3}$ is the two-fold coproduct.

\begin{proposition} \label{averagingdd}
The image $ \Xi(\LL^\infty(D(\GGamma))) $ of the averaging map $\Xi\colon \LL^\infty(D(\GGamma)) \rightarrow \LL^\infty(D(\GGamma)) $ 
equals $ \mc{Z} \ell^\infty(\GGamma) \otimes \I $. More precisely, for $x\in \LL^\infty(D(\GGamma))$ we have 
\begin{equation}\label{eq:Xi_DD}
\Xi(x) = y\otimes\I
\quad\text{where}\quad
y = (\id\otimes h)(\ww((\id\otimes h)(x)\otimes \I)\ww^*).
\end{equation}
Similarly, $\Xi(\mrm{c}_{00}(\GGamma) \odot \Pol(\wh\GGamma)) = \mc{Z} \mrm{c}_{00}(\GGamma) \otimes \I$ and $ \Xi(\mrm{C}_0(D(\GGamma))) = \mc{Z} \mrm{c}_0(\GGamma) \otimes \I $. 
\end{proposition}

\begin{proof}
We use the description of $ \Xi $ in \eqref{eq5}. Firstly, for $x\in \LL^\infty(D(\GGamma))$,
\[
  (\id_{D(\GGamma)}\otimes \eps\otimes\id)\Delta_{D(\GGamma)}(x)=
  (\id_{D(\GGamma)}\otimes \eps\otimes \id)\chi_{23}\bigl(
    \ww_{23} (\Delta\otimes \wh\Delta)(x)\ww_{23}^*\bigr)=
    (\id\otimes \wh\Delta)(x),
\]
using that $\eps$ is a $*$-homomorphism with $(\id\otimes\eps)\Delta = \id$ and $(\eps\otimes\id)(\ww) = \I$. It follows that
\begin{align*}
  (\id_{D(\GGamma)}\otimes\pi^*(h))\Delta_{D(\GGamma)}(x)
  &= (\id_{D(\GGamma)}\otimes h)(\id\otimes\wh\Delta)(x)
  = (\id\otimes h)(x) \otimes \I. 
\end{align*}
Let $z=(\id\otimes h)(x) \in \ell^\infty(\GGamma)$, and notice that
\begin{align}\label{eq7b}
\Delta_{D(\GGamma)}(z\otimes \I) &= (\id\otimes\chi\otimes\id)(\ww_{23} \Delta(z)_{12}\ww_{23}^*).
\end{align}
Thus, using that $(\eps\otimes\id)\Delta=\id$, we get 
\begin{align*}
\Xi(x)
&= (\eps\otimes h\otimes\id_{D(\GGamma)})\Delta_{D(\GGamma)}(z\otimes\I)
= (\eps\otimes\id\otimes h\otimes\id)(\ww_{23} \Delta(z)_{12}\ww_{23}^*) \\
&= (\id\otimes h\otimes \id)(\ww_{12}(z\otimes\I\otimes\I)\ww_{12}^*)
=  (\id\otimes h)(\ww(z\otimes\I)\ww^*) \otimes \I,
\end{align*}
which shows \eqref{eq:Xi_DD}.

As $\ww\in\ell^\infty(\GGamma)\bar\otimes \LL^\infty(\wh\GGamma)$ and right slices of $\ww$ generate $\ell^\infty(\GGamma)$, it follows that an element $c\in\ell^\infty(\GGamma)$ 
is central if and only if $\ww(c\otimes\I) = (c\otimes\I)\ww$, which holds if and only if $\ww(c\otimes\I)\ww^* = c\otimes\I$. Moreover, if $ c= (\id\otimes h)(\ww(d\otimes\I)\ww^*) $ for 
some $ d\in\ell^\infty(\GGamma)$ then
\begin{align*}
\ww(c\otimes\I)\ww^* &= (\id\otimes h \otimes \id)(\ww_{13}\ww_{12}(d\otimes\I \otimes \I)\ww_{12}^*\ww_{13}^*) \\
&=(\id\otimes h \otimes \id)((\id \otimes \wh\Delta)(\ww)(d \otimes \I \otimes \I)(\id \otimes \wh\Delta)(\ww^*))\\
&= (\id\otimes h)(\ww(d\otimes\I)\ww^*) \otimes \I = c \otimes \I, 
\end{align*}
so that $ c $ is central. Combining this with the above formula for $\Xi(x)$ implies $\Xi(\LL^\infty(D(\GGamma))) =\mc Z\ell^\infty(\GGamma)\otimes\I$ as claimed.

Let $x\in\ell^\infty(\bbGamma)$ and set $y = (\id\otimes h)(\ww(x\otimes\I)\ww^*)$ as in \eqref{eq:Xi_DD}.  We use the explicit formula for $\wh{\ww} = \chi(\ww^*)$ from \eqref{eq:wcmpt} together with \eqref{eq:orthog_matrix} to see that
\begin{align}
&\quad\;
y = \sum_{\alpha,\beta\in \Irr(\wh\GGamma)}\sum_{i,j=1}^{\dim(\alpha)}\sum_{k,l=1}^{\dim(\beta)} e^\alpha_{j,i} x e^\beta_{k,l} h\big( (U^\alpha_{i,j})^* U^\beta_{k,l} \big)
\notag \\
&= \sum_{\alpha\in\Irr(\wh\GGamma)}\sum_{i,j,k=1}^{\dim(\alpha)} e^\alpha_{j,i} x e^\alpha_{k,j} \frac{(\uprho_\alpha^{-1})_{k,i}}{\Tr(\uprho_\alpha)}
= \sum_{\alpha\in\Irr(\wh\GGamma)} \Big( \sum_{i,k=1}^{\dim(\alpha)} x^\alpha_{i,k} \frac{(\uprho_\alpha^{-1})_{k,i}}{\Tr(\uprho_\alpha)} \Big) p_\alpha.
\label{eq:A_map_general}
\end{align} 
Here we use the direct-sum matrix decomposition $x=(x^\alpha)_{\alpha\in\Irr(\wh\GGamma)} \in \ell^\infty(\bbGamma)$, and $p_\alpha$ denotes the minimal projection onto the $\alpha$ block.

From this explicit formula it is clear that $\Xi$ maps $\mrm{c}_{00}(\bbGamma)\odot \Pol(\wh\GGamma)$ onto $ \mc{Z} \mrm{c}_{00}(\GGamma) \otimes \I$. 
Since $\mrm{C}_0(D(\GGamma))$ is the norm-closure of $\mrm{c}_{00}(\GGamma)\odot\Pol(\wh\GGamma)$ and $\Xi$ is a conditional expectation by Lemma~\ref{lemma1}, 
it is also immediate that $\Xi$ maps $ \mrm{C}_0(D(\GGamma)) $ onto $ \mc{Z} \mrm{c}_0(\GGamma) \otimes \I $.
\end{proof}

\begin{remark}\label{rem:dd_one_sided}
The proof shows that the one-sided averaging map $ \Xi^r = (\id\otimes\pi^*(h))\Delta_{D(\bbGamma)}$ is given simply by $ \Xi^r(x) = (\id\otimes h)(x) \otimes \I$. This maps $\mrm{c}_{00}(\GGamma) \odot \Pol(\wh\GGamma)$ onto $\mrm{c}_{00}(\bbGamma) \otimes \I$, and hence $\mrm{C}_0(D(\bbGamma))$ onto $\mrm{c}_0(\bbGamma)\otimes \I$.

The other one-sided averaging map $ \Xi^l $ looks more complicated, but if we start with an element of the form $x\otimes \I \in \ell^\infty(\bbGamma)\otimes\I$, then $(\pi^*(h)\otimes\id)\Delta_{D(\bbGamma)}(x\otimes \I) = A(x) \otimes \I$ where
\[ A(x) = (\id\otimes h)(\ww(x\otimes\I)\ww^*) = (h\otimes\id)(\wh\ww^*(\I\otimes x)\wh\ww)\quad(x\in\ell^{\infty}(\bbGamma)). \]
We shall study this map $A$ further in Section~\ref{sec:unimodular}.
\end{remark}

Recall from \cite[Lemmas 8.4, 8.5]{DKV_ApproxLCQG} that we have an isometric weak$^*$-weak$^*$-continuous 
embedding $N\colon \M_{cb}(\Corep(\GGamma)) \rightarrow \M_{cb}^l(\A(D(\GGamma))) $ given by $N(a) = a \otimes \I$. 

\begin{proposition}\label{propimagecentral}
Under the embedding $N\colon \M_{cb}(\Corep(\GGamma)) \rightarrow \M_{cb}^l(\A(D(\GGamma))) $ we have an isometric identification 
$$ 
\mc Z \M^l_{cb}(\A(\GGamma)) = \M_{cb}(\Corep(\GGamma)) \cong \M_{cb}^l(\A(\wh\GGamma\backslash D(\GGamma)/\wh\GGamma)),
$$ 
compatible with the weak$^*$-topologies. Furthermore, these identifications preserve the property of being completely positive.
\end{proposition}

\begin{proof}
The isometric identification $ \mc{Z} \M^l_{cb}(\A(\GGamma)) = \M_{cb}(\Corep(\GGamma)) $, compatible with the weak$^*$-topologies, is established in \cite[Lemma~8.6]{DKV_ApproxLCQG}.  

Let $a\in \M_{cb}^l(\A(\wh\GGamma\backslash D(\GGamma)/\wh\GGamma))$. According to Proposition~\ref{averagingdd} and Proposition~\ref{averagingcb} there is $ b\in \mc{Z} \ell^\infty(\GGamma)$ with $a= \Xi(a) = b\otimes \I \in\M^l_{cb}(\A(D(\GGamma)))$.
In fact, $b\in \mc{Z} \M^l_{cb}(\A(\GGamma))$, which can be proved following the 
argument for \cite[Lemma 8.6]{DKV_ApproxLCQG}. More precisely, it suffices to show that $b$ is a cb-multiplier, and to show this, one observes that $\Theta^l(b\otimes\I)$ leaves the image of $\LL^\infty(\wh\GGamma) \subseteq \LL^\infty(\wh{D(\GGamma)})$ invariant and so induces a centraliser whose associated multiplier must be $b$.

Conversely, if $b\in \mathcal Z \M_{cb}^l(\A(\GGamma))=\M_{cb}(\Corep(\GGamma))$, then \cite[Lemma~8.4]{DKV_ApproxLCQG} shows that 
$a=b\otimes\I\in \M_{cb}^l(\A(D(\GGamma)))$ and Proposition~\ref{averagingdd} gives $\Xi(a)=a$. Thus $a\in \M_{cb}^l(\A(\wh\GGamma\backslash D(\GGamma)/\wh\GGamma)) $.

The map $\M_{cb}^l(\A(\wh\GGamma\backslash D(\GGamma)/\wh\GGamma)) \ni a \mapsto b \in \mc{Z} \ell^\infty(\GGamma)$ is weak$^*$-weak$^*$-continuous because $b = (\id\otimes h)(a)$.  It follows that the map is a weak$^*$-weak$^*$-homeomorphism, compare \cite[Lemma~3.7]{DKV_ApproxLCQG}.

It remains to verify the claim regarding CP multipliers. 
From \cite[Proposition 6.1]{PVcstartensor} we know that $\theta\in \mc{Z}\M^l_{cb}(\A(\GGamma))$ is CP (i.e.~$\Theta^l(\theta)\in\CB^\sigma(\LL^{\infty}(\wh\GGamma))$ is CP) if 
and only if $\theta\in \M_{cb}(\Corep(\GGamma))$ is CP. 
In addition, for $\theta\in \mc{Z}\M^l_{cb}(\A(\GGamma))$ we have $ \Theta^l(\theta)(\I) = \theta(e) \I $, where $e\in \Irr(\wh\GGamma)$ stands for the trivial representation. 
Since $\Theta^l(\theta\otimes \I)$ is normal, we deduce from \cite[Lemma 8.4]{DKV_ApproxLCQG} that similarly $\Theta^l(\theta\otimes \I)(\I)=\theta(e)\I$. 
Now recall from \cite[Theorem 1.35]{Pisier} that if $\mc A$ is a unital $\mrm{C}^*$-algebra, $\mathscr H$ is a Hilbert space and $u\colon \mc A\rightarrow \B(\mathscr H)$
a completely bounded linear map, then $u$ is CP if and only if $\|u\|_{cb}=u(\I)$. 
Since $ \|\theta\|_{cb}=\|\theta\otimes \I\|_{cb} $ we conclude that $\theta\in \mc{Z}\M^l_{cb}(\A(\GGamma))$ is CP if and only if $ \theta\otimes \I\in \M^l_{cb}(\A(D(\GGamma)))$ is CP. 
\end{proof}

We are now ready to state our main result. This improves in particular \cite[Proposition~8.9]{DKV_ApproxLCQG} by removing the unimodularity condition.

\begin{theorem}\label{thm:main}
Let $ \GGamma $ be a discrete quantum group and let $ D(\GGamma) $ be its Drinfeld double. Then the following conditions are equivalent. 
\begin{enumerate}
\item[1)] $\GGamma$ is centrally strongly amenable (respectively, is centrally weakly amenable, has the central Haagerup property, has central AP). 
\item[2)] $D(\GGamma)$ is strongly amenable (respectively, is weakly amenable, has the Haagerup property, has AP). 
\end{enumerate}
Furthermore, in the weakly amenable case we have $\Lambda_{cb}(D(\GGamma))=\mc Z \Lambda_{cb}(\GGamma)$.
\end{theorem} 

\begin{proof} 
$ 1) \Rightarrow 2) $ For AP this claim is \cite[Propositions~8.7,~8.9]{DKV_ApproxLCQG}; let us sketch the argument. Let $(a_i)_{i\in I}$ be a net in $\mc{Z}\mrm{c}_{00}(\bbGamma)$ verifying Definition~\ref{defapsdqg} for central AP. From Proposition~\ref{propimagecentral} (which is based on \cite[Proposition~6.1]{PVcstartensor}) we know that $(a_i\otimes\I)_{i\in I}$ is a net in $\M^l_{cb}(D(\bbGamma))$ converging weak$^*$ to $\I$.  As each $a_i\otimes\I$ is in $\mrm{c}_{00}(\bbGamma)\odot\Pol(\wh\bbGamma)\subseteq \A(D(\bbGamma))$ it follows 
that $D(\bbGamma)$ has AP.

If $\bbGamma$ is centrally weakly amenable, the net $(a_i)_{i\in i}$ can be chosen to additionally satisfy $\|a_i\|_{cb}\leq \mc{Z} \Lambda_{cb}(\GGamma)$ for each $i$.  By Proposition~\ref{propimagecentral}, we have $\|a_i\otimes\I\|_{cb} = \|a_i\|_{cb}$, and so $(a_i\otimes\I)_{i\in I}$ is a net bounded by $\mc{Z} \Lambda_{cb}(\GGamma)$.  Invoking \cite[Proposition~5.7]{DKV_ApproxLCQG} shows that $D(\bbGamma)$ is weakly amenable with $\Lambda_{cb}(D(\GGamma)) \leq \mc{Z} \Lambda_{cb}(\GGamma)$. 

If $\bbGamma$ is centrally strongly amenable, each $a_i$ can be chosen to be a UCP multiplier, and then $a_i\otimes\I$ will be a UCP multiplier in $\A(D(\bbGamma))$ 
according to Proposition~\ref{propimagecentral}. By Lemma~\ref{lem:cp_fourier_pos}, each $a_i\otimes\I$ arises from a state, and in particular, the net $(a_i\otimes\I)_{i\in I}$ is 
bounded in $\A(D(\bbGamma))$. We now invoke \cite[Proposition~5.6]{DKV_ApproxLCQG} to conclude that $D(\bbGamma)$ is strongly amenable.

When $\bbGamma$ has the central Haagerup property, the net $(a_i)_{i\in I}$ can be chosen to consist of central CP multipliers forming a bounded approximate identity 
for $\mrm{c}_0(\bbGamma)$.  By Proposition~\ref{propimagecentral} each $a_i\otimes\I$ is a CP multiplier of $D(\bbGamma)$, and $(a_i\otimes\I)_{i\in I}$ is clearly a bounded 
approximate identity for $\mrm{C}_0(D(\bbGamma)) = \mrm{c}_0(\bbGamma)\otimes \mrm{C}(\wh\bbGamma)$.

$ 2) \Rightarrow 1) $ Assume first that $D(\GGamma)$ has AP exhibited by a net $(a_i)_{i \in I}=(\lambda_{\wh{D(\GGamma)}}(\omega_i))_{i \in I} \in \A(D(\GGamma))$. 
Since $\LL^{\infty}(\wh{D(\GGamma)})$ is in standard position on $\LL^2(D(\GGamma))$, each $\omega_i$ is a vector functional. By a straightforward approximation argument we may 
assume that each $\omega_i$ is of the form $\omega_{\xi_i, \eta_i}$ where $\xi_i, \eta_i \in \Lambda_\vp(\mrm{c}_{00}(\GGamma))\odot \Lambda_h(\Pol(\wh\GGamma))$, 
here recalling that $\vp\otimes h$ is the left Haar integral on $D(\GGamma)$. Then note that $D(\GGamma)$ arises from an algebraic quantum group, 
compare \cite{DaeleAlgebraic}, \cite[Section~3.2]{VoigtYuncken}, and observe that 
\begin{equation}\label{eq9}
(\id\otimes \omega_{\Lambda_{\vp\otimes h}(x),\Lambda_{\vp\otimes h}(y)})\ww^{D(\GGamma) *}=
(\id\otimes (\vp\otimes h))\bigl((\I\otimes x^*)\Delta_{D(\GGamma)}(y)\bigr)\in \mrm{c}_{00}(\GGamma)\odot \Pol(\wh\GGamma)
\end{equation}
for $x,y\in \mrm{c}_{00}(\GGamma)\odot \Pol(\wh\GGamma)$, see e.g.~\cite[Page 12]{DaeleLCQG}. 
Thus each $a_i$ is a member of $\mrm{c}_{00}(\GGamma)\odot \Pol(\wh\GGamma)$. By Proposition~\ref{averagingdd} we have $\Xi(a_i) = b_i\otimes\I$ for some $b_i \in \mc{Z}\mrm{c}_{00}(\GGamma)$, and Propositions~\ref{averagingcb} and~\ref{propimagecentral} show that $b_i \xrightarrow[i\in I]{} \I$ weak$^*$ in $\M^l_{cb}(\A(\GGamma))$. Hence $\GGamma$ has central AP.

If $D(\GGamma)$ is weakly amenable we can choose a net $(a_{i})_{i\in I}$ which is a left approximate unit in $\A(D(\GGamma))$ in such a way that $\|a_i\|_{cb}\le \Lambda_{cb}(D(\GGamma))$.  Again we may approximate each $a_i$ by a member of $\mrm{c}_{00}(\GGamma)\odot \Pol(\wh\GGamma)$. Since this approximation is made in the $\A(D(\GGamma))$ norm, we may assume that 
we still have $\|a_i\|_{cb}\le \Lambda_{cb}(D(\GGamma))$. Again let $\Xi(a_i) = b_i\otimes\I$, and observe that Propositions~\ref{averagingcb} and~\ref{propimagecentral} 
now show $\|b_i\|_{cb} \leq \Lambda_{cb}(D(\GGamma))$. Given $\alpha\in \Irr(\wh\GGamma)$ denote by $p_\alpha\in \mrm{c}_{00}(\GGamma) \subseteq \mrm{c}_0(\GGamma)$ the central projection 
onto the $\alpha$ component. Then $p_\alpha\otimes\I \in \mrm{c}_{00}(\GGamma)\odot \Pol(\wh\GGamma)\subseteq \A(D(\GGamma))$ and so, as $\Xi$ is a conditional expectation,
\begin{equation}\label{eq8}
b_{i}p_\alpha\otimes \I =
\Xi(a_i) (p_\alpha\otimes \I) =
\Xi(a_i (p_\alpha\otimes \I))
\xrightarrow[i\in I]{} \Xi(p_\alpha\otimes \I)
= p_\alpha\otimes \I,
\end{equation}
with the convergence in $\A(D(\GGamma))$ by Corollary~\ref{corr:xi_to_fourier}. Thus $b_i \xrightarrow[i\in I]{} \I$ pointwise in $\ell^\infty(\GGamma)$, and 
we conclude that $ \GGamma $ is centrally weakly amenable with $\mc {Z}\Lambda_{cb}(\GGamma)\le \Lambda_{cb}(D(\GGamma))$.

If $D(\GGamma)$ is strongly amenable it follows from the discussion after Definition~\ref{defapsqg} that $\A(D(\GGamma))$ has a bounded approximate 
identity $(a_i)_{i\in I}$ where $a_i = \lambda_{\wh{D(\GGamma)}}(\omega_i)$ for some state $\omega_i$. By approximation, we may suppose again that 
$\omega_i = \omega_{\xi_i}$ for $\xi_i \in \Lambda_\vp(\mrm{c}_{00}(\GGamma))\odot \Lambda_h(\Pol(\wh\GGamma))$ for all $ i $. Let $\Xi(a_i) = b_i\otimes\I$, 
so that by Propositions~\ref{cppreservation} and~\ref{propimagecentral} each $b_i$ is a CP multiplier in $\mc{Z} \mrm{c}_{00}(\GGamma)$. As before, $b_i \xrightarrow[i\in I]{} \I$ 
pointwise in $\ell^\infty(\GGamma)$, which shows that $\GGamma$ is centrally strongly amenable.

Finally, consider the situation where $D(\GGamma)$ has the Haagerup property. Let $(a_i)_{i\in I}$ be a bounded approximate identity in $\mrm{C}_0(D(\GGamma))$ which consists of CP multipliers. By Proposition~\ref{averagingdd} we obtain $b_i\in \mc{Z} \mrm{c}_0(\GGamma)$ with $\Xi(a_i)=b_i\otimes \I$ for each $i\in I$. By Propositions~\ref{cppreservation} and~\ref{propimagecentral}, each $b_{i}\in\M^l_{cb}(\A(\GGamma))$ is a CP multiplier, and we have $\sup_{i\in I}\|b_i\|\le \sup_{i\in I}\|a_i\| <+\infty$. 
Furthermore, as $p_\alpha\otimes \I\in \mc{Z}\mrm{C}_0(D(\GGamma))$ for $\alpha\in \Irr(\wh\GGamma)$, and the averaging map $\Xi$ is $\LL^{\infty}(D(\GGamma))$-norm continuous, 
we deduce as in \eqref{eq8} that $b_i\xrightarrow[i\in I]{}\I$ pointwise. Since the net $(b_i)_{i\in I}$ is bounded in norm it forms an approximate identity for $\mrm{c}_0(\GGamma)$, 
thus showing that $\GGamma$ has the central Haagerup property.
\end{proof}

\section{Further results for discrete quantum groups and their Drinfeld doubles}\label{sec:unimodular}

In this section we complement the discussion in Section~\ref{section double} with some further analysis related to unimodularity. 

\subsection{Approximation in the unimodular case}\label{sec:unimodularsub}

Let us first review the averaging procedure for discrete quantum groups $\bbGamma$ with respect to the coadjoint action, which allows one to 
compare central approximation properties with their non-central counterparts in the unimodular case, see \cite[Section 5]{KrausRuan}, \cite[Section~6.3.2]{Brannan}, 
and which was already mentioned in Remark \ref{rem:dd_one_sided}. 

Let $\bbGamma$ be a discrete quantum group. By definition, the coadjoint action of $ \wh\bbGamma$ on $\ell^\infty(\bbGamma)$ is the map $ \gamma\colon \ell^\infty(\bbGamma) \rightarrow \LL^\infty(\wh\bbGamma) \bar{\otimes} \ell^\infty(\bbGamma) $
defined by $\gamma(x) = \wh{\ww}^* (\I \otimes x) \wh{\ww}$. 
Combining $ \gamma $ with the Haar state $h$ of $\LL^\infty(\wh\bbGamma)$ we obtain the averaging map $A\colon  \ell^\infty(\bbGamma) \rightarrow \ell^\infty(\bbGamma) $ 
by setting
$$
A = (h \otimes \id)\gamma; \quad x\mapsto (h\otimes\id)\bigl(\wh{\ww}^* (\I \otimes x) \wh{\ww}\bigr).
$$
Clearly this is a normal unital CP map. It is straightforward to check, as in the proof of Proposition~\ref{averagingdd}, that the algebra of invariant elements 
$ \{x \in \ell^\infty(\bbGamma) \mid \gamma(x) = \I \otimes x \} \subseteq \ell^\infty(\bbGamma)$ is equal to the centre $ \mc Z \ell^\infty(\bbGamma) $ of $\ell^\infty(\bbGamma)$. 
We calculated above in \eqref{eq:A_map_general} the form of $A$, and from this calculation it is clear that $A$ is a conditional expectation onto $ \mc Z \ell^\infty(\bbGamma) $, and that $A$ restricts to a conditional expectation $\mrm{c}_0(\bbGamma) \to \mc Z\mrm{c}_0(\bbGamma)$.  In the case when $\bbGamma$ is unimodular, things simplify, and we obtain
\begin{equation}\label{eq:defn_A}
A(e^\alpha_{m,n}) = \delta_{m,n} \frac{1}{\dim(\alpha)} \,p_\alpha. 
\end{equation}
In particular, the above definition of $A$ coincides with the construction in \cite[Section 6]{DKV_ApproxLCQG}. 

Recall that the \emph{character} of $\alpha\in\Irr(\wh\bbGamma)$ 
is $\chi_\alpha = \sum_{i=1}^{\dim(\alpha)} U^\alpha_{i,i} \in \Pol(\wh\bbGamma) \subseteq \LL^\infty(\wh\bbGamma)$. Motivated by the case of classical compact 
groups we shall call
\[ \mathscr{C}_{\wh\bbGamma} = \{ \chi_\alpha \,|\, \alpha\in\Irr(\wh\bbGamma) \}'' \subseteq \LL^\infty(\wh\bbGamma), \]
the von Neumann algebra of \emph{class functions} of $ \wh\bbGamma $, compare \cite[Lemma~1.2]{KW_ClassFuncs}. If $\GGamma$ is unimodular, the Haar state $ h $ on $\LL^\infty(\wh\bbGamma)$ is a trace and there is a unique normal conditional expectation $F\colon \LL^\infty(\wh\bbGamma) \rightarrow \mathscr{C}_{\wh\bbGamma}$ which preserves $h$, see \cite[Theorem~4.2, Chapter~IX]{TakesakiII}. 

The next proposition summarises key properties of $A$. Most of these are well-known, but we make the new observation that $A$ maps the Fourier algebra into itself.

\begin{proposition}\label{unimodularaveraging}
Let $\bbGamma$ be an unimodular discrete quantum group. The averaging map $ A\colon \ell^\infty(\GGamma)$ $\rightarrow \mc Z \ell^\infty(\bbGamma) $ restricts to a 
weak$^*$-weak$^*$-continuous contraction $\M_{cb}^l(\A(\bbGamma)) \rightarrow \mc Z \M_{cb}^l(\A(\bbGamma))$. It maps CP multipliers to CP multipliers and preserves finite support. 
Moreover it restricts to a completely contractive map $\A(\bbGamma) \rightarrow \mc Z \A(\bbGamma)$. Under the identification $\A(\GGamma)\cong\LL^1(\wh\GGamma)$, it agrees with the predual $F_*$ of the unique $h$-preserving conditional expectation $F\colon \LL^{\infty}(\wh\GGamma)\rightarrow\mathscr{C}_{\wh\GGamma}$.
\end{proposition}

\begin{proof}
The first assertions are already contained in the discussion around \cite[Proposition~6.8]{DKV_ApproxLCQG}, see also \cite[Section~6.3.2]{Brannan}; we recall some of the details.
Let $E\colon\LL^\infty(\wh\bbGamma) \bar\otimes \LL^\infty(\wh\bbGamma) \rightarrow \wh\Delta(\LL^\infty(\wh\bbGamma))$ be the unique conditional expectation which preserves the state $h\otimes h$. 
It is normal and satisfies
\begin{equation}
E(U^\alpha_{i,j} \otimes U^\beta_{k,l}) = \delta_{\alpha,\beta} \delta_{j,k} \frac{1}{\dim(\alpha)} \wh\Delta(U^\alpha_{i,l}). \label{eq:defn_E}
\end{equation}
We set $\wh\Delta^\sharp = \wh\Delta^{-1}E\colon\LL^\infty(\wh\bbGamma) \bar\otimes \LL^\infty(\wh\bbGamma) \rightarrow \LL^\infty(\wh\bbGamma)$, and 
for $T \in \CB^\sigma(\LL^\infty(\wh\bbGamma))$ we define
\[ 
\Psi(T) = \wh\Delta^\sharp(\id\otimes T)\wh\Delta \in \mrm{CB}^\sigma(\LL^\infty(\wh\bbGamma)). 
\]
The proof of \cite[Proposition~6.8]{DKV_ApproxLCQG} shows that $\Psi( \Theta^l(a)) = \Theta^l( A(a) )$ for $a\in\M^l_{cb}(\A(\bbGamma))$. In particular, $\Psi$ maps duals of CB 
centralisers to duals of CB centralisers, and hence $A$ maps CB multipliers to central CB multipliers. It follows also that CP multipliers are mapped to central CP multipliers, 
and it is evident from \eqref{eq:defn_A} that $A$ preserves finite support. 

It remains to prove the claim about the Fourier algebra. Using the $\mathscr{C}_{\wh\bbGamma}$-bimodule property of $ F $ we obtain
\[ 
h(F(U^\alpha_{i,j}) \chi_\beta^* ) = h(F(U^\alpha_{i,j}\chi_\beta^*))
= h(U^\alpha_{i,j} \chi_\beta^*)
= \sum_{k=1}^{\dim(\beta)} h(U^\alpha_{i,j} (U^\beta_{k,k})^*)
= \delta_{\alpha,\beta} \delta_{i,j} \frac{1}{\dim(\alpha)}
\]
for all $\alpha,\beta\in\Irr(\wh\bbGamma),1\le i,j\le \dim(\alpha)$, 
and since $h(\chi_\alpha \chi_\beta^*) = \delta_{\alpha,\beta}$ and $h$ is faithful it follows that
\begin{equation}
F(U^\alpha_{i,j}) = \frac{1}{\dim(\alpha)} \delta_{i,j} \chi_\alpha.
\label{eq:cond_exp_classfns}
\end{equation}
As $F\colon \LL^{\infty}(\wh\GGamma)\rightarrow \mathscr{C}_{\wh\GGamma}\subseteq\LL^{\infty}(\wh\GGamma)$ is normal, it has a Banach space pre-adjoint $F_*\colon \LL^1(\wh\bbGamma) \rightarrow \LL^1(\wh\bbGamma)$ which is completely contractive. 
Let $\wh\omega\in \LL^1(\wh\bbGamma)$ and set $a = \wh\lambda(\wh\omega) \in \A(\bbGamma) \subseteq \M^l_{cb}(\A(\bbGamma))$. Since $\Theta^l(a) = (\wh\omega\otimes\id)\wh\Delta$, 
we obtain by the definition of $\Psi$, and using \eqref{eq:defn_E},
\begin{equation}\begin{split}\label{eq12a}
\Psi(\Theta^l(a));\, U^\alpha_{i,j} \mapsto &
\wh\Delta^\sharp\Big( \sum_{k=1}^{\dim(\alpha)} U^\alpha_{i,k} \otimes \Theta^l(a)(U^\alpha_{k,j}) \Big)
= \wh\Delta^\sharp\Big( \sum_{k,l=1}^{\dim(\alpha)} U^\alpha_{i,k} \otimes \wh\omega(U^\alpha_{k,l}) U^\alpha_{l,j} \Big) \\
&= \sum_{k=1}^{\dim(\alpha)} \frac{ \wh\omega(U^\alpha_{k,k}) }{ \dim(\alpha) } U^\alpha_{i,j}
= \frac{ \wh\omega(\chi_\alpha) }{ \dim(\alpha) } U^\alpha_{i,j}.
\end{split}
\end{equation}
Set $b = \wh\lambda(F_*(\wh\omega)) \in \A(\bbGamma)$, so that, using \eqref{eq:cond_exp_classfns},
\begin{equation}\label{eq12b}
 \Theta^l(b) ;\, U^\alpha_{i,j} \mapsto \sum_{k=1}^{\dim(\alpha)} F_*(\wh\omega)(U^\alpha_{i,k}) U^\alpha_{k,j}
= \sum_{k=1}^{\dim(\alpha)} \wh\omega(F(U^\alpha_{i,k})) U^\alpha_{k,j}
= \frac{1}{\dim(\alpha)} \wh\omega(\chi_\alpha) U^\alpha_{i,j}.
\end{equation}
Since the expressions \eqref{eq12a}, \eqref{eq12b} agree we conclude $\Psi(\Theta^l(a)) = \Theta^l(b)$ and $A(a) = b \in \A(\bbGamma)$. This means in particular that $ A $ maps $\A(\bbGamma)$ 
to itself. 

In fact, we see that $A\colon \A(\bbGamma) \rightarrow \A(\bbGamma)$ identifies with $F_*\colon \LL^1(\wh\bbGamma) \rightarrow \LL^1(\wh\bbGamma)$ via 
the canonical isomorphism $\A(\bbGamma)\cong \LL^1(\wh\bbGamma)$. In particular it is completely contractive as claimed. 
\end{proof}

\begin{remark}\label{rem:A_bad_fourier_nonunimod}
Without the unimodularity assumption, the averaging map $A$ can behave quite badly. More precisely, one can show that there is a non-unimodular discrete quantum group $\bbGamma$ and $x\in \A(\bbGamma)$ completely positive, such that $A(x)\notin \M^l_{cb}(\A(\bbGamma))$. As we will not make use of this fact, we will only sketch the proof.

Firstly, fix $0<q<1$. Let $\alpha,\gamma$ be the standard generators of $\Pol(\SU_q(2))$ and define $\nu=i(q+q^{-1}) (q^{-1} h(\alpha^*\cdot) - q h(\alpha\, \cdot))\in \LL^1(\SU_q(2))$. By direct calculation one can check that $\ov{\nu}=\nu$, but after averaging we obtain $A(\wh\lambda(\nu))=\wh\lambda(\omega)$ for $\omega=i(q-q^{-1})(q^{-1} h(\alpha^*\cdot)+q h(\alpha\,\cdot))$ with $\ov{\omega}\neq \omega$. Next consider $\theta=h+\tfrac{\nu}{2 (q+q^{-1})}$ and $y=\wh\lambda(\theta)\in \mrm{c}_{00}(\wh{\SU_q(2)})$. Using the direct integral picture from \cite[Section 7.1]{KrajczokModular}, one can show that $y$ is completely positive, but $A(y)$ is not. Since both $\Theta^l(y)$ and $\Theta^l(A(y))$ are unital, we have $\|A(y)\|_{cb}>\|y\|_{cb}=1$.  To complete the argument, set $\wh\bbGamma=\prod_{m=1}^{\infty} \SU_q(2)$ and consider the positive linear functional $\rho=\sum_{m=1}^{\infty}\tfrac{1}{(1+\delta)^m} \theta^{\otimes m}\otimes h^{\otimes\infty}\in \LL^1(\wh\bbGamma)$ for appropriately chosen small $\delta>0$. Then $x=\wh\lambda(\rho)\in \A(\bbGamma)$ is CP, but $A(x)$ is not a left CB multiplier.
\end{remark}

Next we record a lemma of independent interest, which we will use later; compare also Theorem~\ref{prop:centralAP_density} below.

\begin{lemma}\label{lemma:dense_in_centre_fourier}
Let $\bbGamma$ be an unimodular discrete quantum group.  Then $\mc Z\mrm{c}_{00}(\bbGamma)$ is dense in $\mc Z\A(\bbGamma)$ for the $\A(\bbGamma)$ norm. Furthermore, CP multipliers in $\mc{Z}\A(\GGamma)$ can be approximated by CP multipliers in $\mc{Z}\mrm{c}_{00}(\bbGamma)$.
\end{lemma}

\begin{proof}
Take $a\in\mc{Z}\A(\GGamma)$ and write $a=\wh\lambda(\omega_{\xi,\eta})$ for some $\xi,\eta\in \LL^2(\wh\GGamma)$. For $\eps>0$ choose vectors $\xi_\eps,\eta_\eps\in \Lambda_{h}(\Pol(\wh\GGamma))$ such that $\|\xi-\xi_\eps\|,\|\eta-\eta_\eps\|\le \eps$, next define $b_\eps=\wh\lambda(\omega_{\xi_\eps,\eta_\eps})\in \mrm{c}_{00}(\GGamma)$ and $a_\eps=A(b_\eps)$, where $A$ is the averaging map. By Proposition \ref{unimodularaveraging} we have $a_\eps\in \mc{Z}\mrm{c}_{00}(\GGamma)$ and furthermore
\[
\|a-a_\eps\|_{\A(\GGamma)}=
\|A(a) - A(b_\eps)\|_{\A(\GGamma)}\le 
\|a-b_\eps \|_{\A(\GGamma)}\xrightarrow[\eps\to 0]{}0,
\]
which proves the first claim. If $a$ is a CP multiplier then $\omega_{\xi,\eta}\ge 0$ on $\LL^{\infty}(\wh\GGamma)$ by Lemma \ref{lem:cp_fourier_pos} and we can take $\xi=\eta$. Then $a_\eps=A(\wh{\lambda}(\omega_{\xi_\eps,\xi_\eps}))$ is also CP by Proposition \ref{unimodularaveraging}.
\end{proof}

For unimodular discrete quantum groups, we can show using Lemma \ref{lemma:dense_in_centre_fourier} that we can replace $\mc{Z} \mrm{c}_{00}(\bbGamma)$ 
by $\mc{Z}\A(\bbGamma)$ in the definition of the central approximation properties, Definition \ref{defapsdqg}. 
We state the following result in a way which draws attention to this viewpoint, but 
we note that the overall conclusion can also be deduced from the known result that the central approximation properties are equivalent to their non-central versions in the unimodular case, 
see \cite[Theorem 7.3]{Brannan}, \cite[Proposition 6.8]{DKV_ApproxLCQG} and their (analogous) proofs. 

\begin{proposition}\label{thm:kac_nofs}
Let $\bbGamma$ be an unimodular discrete quantum group. Then $\bbGamma$ 
\begin{enumerate}
\item[1)] is centrally strongly amenable if and only if there is a net $(e_i)_{i \in I}$ of CP multipliers in $\mc{Z} \A(\bbGamma)$ converging to $\I$ pointwise,
\item[2)] is centrally weakly amenable if and only if there is a net $(e_i)_{i \in I}$ in $\mc{Z} \A(\bbGamma)$ with $\sup_{i\in I}\|e_i\|_{cb}$\\$<+\infty$, 
converging to $\I$ pointwise,
\item[3)] has the central approximation property if and only if there is a net $(e_i)_{i \in I}$ in $\mc{Z} \A(\bbGamma)$ converging to $\I$ in the weak$^*$ topology 
of $\M^l_{cb}(\A(\bbGamma))$.
\end{enumerate}
\end{proposition}

\begin{proof}
1) Let $(e_i)_{i\in I}$ be a net of central CP multipliers in $\A(\bbGamma)$ converging to $\I$ pointwise. Using Lemma \ref{lemma:dense_in_centre_fourier} we choose elements $e_{i,n}\in \mc{Z}\mrm{c}_{00}(\GGamma)$ for $n\in\NN$ which are CP and satisfy $\|e_i-e_{i,n}\|_{\A(\GGamma)}\le \tfrac{1}{n}$. Then the net $(e_{i,n})_{(i,n)\in I\times \NN}$ also converges to $\I$ pointwise, hence $\bbGamma$ is centrally strongly amenable. The reverse implication is obvious. 

2) This is analogous: given a net $(e_i)_{i\in I}$ in $\mc{Z}\A(\GGamma)$ we use Lemma \ref{lemma:dense_in_centre_fourier} to choose elements $e_{i,n}\in \mc{Z}\mrm{c}_{00}(\GGamma)$ with $\|e_i-e_{i,n}\|_{\A(\GGamma)}\le\tfrac{1}{n}$. The new net $(e_{i,n})_{(i,n)\in I\times\NN}$ still converges to $\I$ pointwise, and since the CB norm is dominated by the Fourier algebra norm we have $\sup_{(i,n)\in I\times \NN}\|e_{i,n}\|_{cb}\le \sup_{(i,n)\in I\times \NN} ( \|e_i\|_{cb} + \|e_{i}-e_{i,n}\|_{cb})<+\infty$.

3) We use again Lemma \ref{lemma:dense_in_centre_fourier}, noting that norm convergence in $\A(\GGamma)$ implies weak$^*$ convergence in $\M_{cb}^l(\A(\bbGamma))$. 
More precisely, given $(e_i)_{i\in I}$ as in the claim and choosing $(e_{i,n})_{(i,n)\in I\times \NN}$ in $\mc{Z}\mrm{c}_{00}(\GGamma)$ as before we have we 
obtain $e_{i,n} = e_{i} + (e_{i,n} -e_{i})\xrightarrow[(i,n)\in I\times \NN]{}\I $ weak$^*$ in $\M^l_{cb}(\A(\GGamma))$.
\end{proof}

\begin{remark}
We do not know whether the conclusions of Proposition~\ref{thm:kac_nofs} hold for general $\bbGamma$.
\end{remark}

\subsection{Amenability of Drinfeld doubles}

Let us now complement Theorem~\ref{thm:main} with a discussion of amenability for Drinfeld doubles of discrete quantum groups.  While some of these results are known to experts, we have been unable to find references.

\begin{definition}\label{defn:amenable}
A locally compact quantum group $\GG$ is \emph{amenable} if there exists a \emph{left invariant mean} $m$ on $\LL^\infty(\GG)$, that is, there is a state $m$ on $\LL^\infty(\GG)$ with
\[ \langle m, x \star \omega \rangle = \langle m,x\rangle \langle \I,\omega\rangle
\qquad (\omega\in\LL^1(\GG), x\in\LL^\infty(\GG)). \]
\end{definition}

Composing a left invariant mean on $\LL^\infty(\GG)$ with the unitary antipode one can equivalently require the existence of a right invariant mean, i.e. a state $m$ 
on $\LL^\infty(\GG)$ such that $\langle m, \omega \star x \rangle = \langle m,x\rangle \langle \I,\omega\rangle$ for all $\omega\in\LL^1(\GG), x\in\LL^\infty(\GG)$. 
 
\begin{lemma}\label{prop:drinbdlamen_kac}
Let $\bbGamma$ be a discrete quantum group such that $D(\bbGamma)$ is amenable. Then $\bbGamma$ is unimodular.
\end{lemma}

\begin{proof}
Let $m$ be a right invariant mean on $\LL^\infty(D(\bbGamma))$. Define a state $n\in\LL^\infty(\wh\bbGamma)^*$ by $n(\wh x) = m(\I\otimes\wh x)$ for $\wh x \in \LL^\infty(\wh\bbGamma)$. 
For $\omega\in\ell^1(\bbGamma), \wh\omega\in \LL^1(\wh\bbGamma)$ and $\wh x\in \LL^\infty(\wh\bbGamma)$ we get 
\begin{align}
\langle \I\otimes \I, \omega\otimes\wh\omega \rangle n(\wh x)
&= \langle m, (\id\otimes\id\otimes\omega\otimes\wh\omega)\Delta_{D(\bbGamma)}(\I\otimes\wh x) \rangle \notag \\
&= \langle m, (\id\otimes\omega\otimes\id\otimes\wh\omega)(\I\otimes \ww\otimes \I)(\I\otimes \I\otimes\wh\Delta(\wh x))(\I\otimes \ww^*\otimes \I) \rangle \notag \\
&= \langle n, (\omega\otimes\id\otimes\wh\omega)(\ww_{12} \wh\Delta(\wh x)_{23} \ww_{12}^*) \rangle, \label{eq:inv-calc}
\end{align}
since $m$ is right invariant. 
Now let $\omega=\eps\in\ell^1(\bbGamma)$ be the counit. Since $ \eps $ is multiplicative and $(\eps\otimes\id)(\ww)=\I$, formula \eqref{eq:inv-calc} gives
\[ \langle \I,\wh\omega\rangle n(\wh x)
= n\big( (\id\otimes\wh\omega)\wh\Delta(\wh x) \big) \qquad (\wh x\in\LL^\infty(\wh\bbGamma)). \]
With $\wh\omega=h\in\LL^1(\wh\bbGamma)$ the Haar state, it follows that $n(\wh x) = n(\I) h(\wh x) = h(\wh x)$ for all $\wh x\in\LL^\infty(\wh\bbGamma)$, and so $n=h$. 
Then \eqref{eq:inv-calc} becomes equivalent to
\[
 h(\wh x) \I\otimes \I=
(\id\otimes h\otimes \id)(\ww_{12} \wh\Delta(\wh x)_{23} \ww_{12}^*) 
\qquad (\wh x\in\LL^\infty(\wh\bbGamma)).
\]
As $(\id\otimes\wh\Delta)(\ww) = \ww_{13} \ww_{12}$, this gives
\begin{align*} h(\wh x)\I\otimes \I &= (\id\otimes h\otimes\id)\big( \ww_{13}^* (\id\otimes\wh\Delta)(\ww(\I\otimes\wh x)\ww^*) \ww_{13} \big). \end{align*}
Conjugating by $\ww$, and using that $(h\otimes\id)\wh\Delta = \I\, h(\cdot)$, we obtain
\begin{equation}\label{eq13}
h(\wh x)\I = (\id\otimes h)(\ww(\I\otimes\wh x)\ww^*).
\end{equation}

Recall that the unitary antipode $R$ on $\ell^\infty(\bbGamma)$ is implemented by $\wh J$ as $R(x)=\wh J x^* \wh J$, and similarly for $\wh R$.  We also know that $(R\otimes\wh R)(\ww)=\ww$.  As $h\circ\wh R=h$, it hence follows from \eqref{eq13} that
\begin{align}
h(\wh x)\I &= h(\wh R(\wh x)) R(\I)
= (\id\otimes h)\big( (R\otimes\wh R)(\ww(\I\otimes\wh R(\wh x))\ww^*)  \big) \notag \\
&= (\id\otimes h)\big( (\wh J\otimes J)\ww(\I\otimes\wh R(\wh x)^*)\ww^*(\wh J\otimes J) \big) \notag \\
&= (\id\otimes h)\big( (\wh J\otimes J)\ww(\wh J\otimes J)(\I\otimes\wh x)(\wh J\otimes J)\ww^*(\wh J\otimes J) \big) \notag \\
&= (\id\otimes h)\big( \ww^*(\I\otimes\wh x)\ww \big). \label{eq13a}
\end{align}

Equation \eqref{eq13a} with $\wh x \in \mrm{C}(\wh\bbGamma)$ means that $h\in \mrm{C}(\wh\bbGamma)^*$ is an invariant state for the coadjoint action $\bbGamma\curvearrowright \mrm{C}(\wh\bbGamma)$ given by $\mrm{C}(\wh\bbGamma)\ni \wh x \mapsto \ww^*(\I\otimes \wh x)\ww\in \M(\mrm{c}_0(\bbGamma)\otimes \mrm{C}(\wh\bbGamma))$. It follows from \cite[Lemma~5.2]{KalantarKasprzakSkalskiVergnioux} that the Haar integral $h$ is a trace, consequently $\wh\bbGamma$ is of Kac type and $\bbGamma$ is unimodular.
\end{proof}

Let us now show that amenability and strong amenability coincide for the Drinfeld double $D(\bbGamma)$ of a discrete quantum group $\bbGamma$.

\begin{theorem}\label{thm:drindblamen}
Let $\bbGamma$ be a discrete quantum group. The following conditions are equivalent:
\begin{enumerate}
\item[1)] $D(\bbGamma)$ is strongly amenable,
\item[2)] $D(\bbGamma)$ is amenable,
\item[3)] $\bbGamma$ is unimodular and amenable,
\item[4)] $\bbGamma$ is unimodular and strongly amenable,
\item[5)] $\bbGamma$ is unimodular and centrally strongly amenable,
\item[6)] $\bbGamma$ is centrally strongly amenable.
\end{enumerate}
\end{theorem}

\begin{proof}
$ 1) \Leftrightarrow 6) $ is a part of Theorem~\ref{thm:main}. 

$ 1) \Rightarrow 2) $ It is known that strong amenability implies amenability in general, see \cite[Theorem~3.2]{BT_Amenability}. 

$ 2) \Rightarrow 3) $ Lemma~\ref{prop:drinbdlamen_kac} shows that $\bbGamma$ is unimodular.  From \cite[Theorem~5.3]{Doublecrossed}, see also \cite[Lemma~7.12]{DKV_ApproxLCQG}, we know that $\bbGamma$ is a closed quantum subgroup of $D(\bbGamma)$.  By \cite[Theorem~3.2]{CrannHereditary} amenability passes from $D(\bbGamma)$ to the closed quantum subgroup $\bbGamma$. 

$ 3) \Rightarrow 4) $ This is a consequence of \cite[Theorem~4.5]{Ruan_AmenHopf}, see also \cite{Tomatsuamenablediscrete}. 

$ 4) \Rightarrow 5) $ This follows using averaging for unimodular discrete quantum groups, compare \cite[Theorem~7.3]{Brannan}.  Indeed, if $(e_i)_{i \in I}$ is a net of CP multipliers in $\A(\bbGamma)$ converging to $\I$ pointwise then by Proposition~\ref{unimodularaveraging}, the net $ (A(e_i))_{i \in I} $ in $ \mc Z \A(\bbGamma)$ obtained from averaging consists again of CP multipliers and converges to $\I$ pointwise.  According to Proposition~\ref{thm:kac_nofs} part 1), this means that $\bbGamma$ is centrally strongly amenable. 

$ 5) \Rightarrow 6) $ is trivial. 
\end{proof}

\begin{remark}
By \cite{Tomatsuamenablediscrete} we know that amenability and strong amenability are equivalent for discrete quantum groups. Clearly we cannot drop unimodularity in condition $3)$ or $4)$ of Theorem~\ref{thm:drindblamen}, as if $\bbGamma$ being amenable implied that $D(\bbGamma)$ was amenable, then the theorem would show that $\bbGamma$ was in particular unimodular, and there are amenable non-unimodular $\bbGamma$. Indeed, it is well-known that the dual of $\SU_q(2)$ is strongly amenable but not centrally strongly amenable, compare \cite{DFY_CCAP}, \cite{Freslon}. 
In view of Theorem~\ref{thm:main}, this is equivalent to the fact that the quantum Lorentz group $D(\SU_q(2)) = \oon{SL}_q(2,\mathbb C)$ is not strongly amenable, compare for instance the discussion in \cite[Section 7]{VYplancherel}.
\end{remark}

\section{Biinvariance and centrality for the Fourier algebra}\label{sec:central_fourier}

Given a discrete quantum group $\GGamma$, recall that $\A(\wh\GGamma\backslash D(\GGamma)/\wh\GGamma)$ is the image of $\A(D(\GGamma))$ under the averaging map $\Xi$. Similarly, $\B_r(\wh\GGamma\backslash D(\GGamma)/\wh\GGamma)$ is the image of the reduced Fourier-Stieltjes algebra $\B_r(D(\GGamma))$ under $\Xi$ (c.f.~Corollary~\ref{corr:xi_to_fourier}).

In view of Propositions~\ref{averagingdd} and~\ref{propimagecentral}, one might wonder whether $\A(\wh\GGamma\backslash D(\GGamma)/\wh\GGamma)$ is equal to $\mc Z \A(\bbGamma) \otimes \I$, and similarly whether $\B_r(\wh\GGamma\backslash D(\GGamma)/\wh\GGamma)$ is equal to $\mc Z \B_r(\bbGamma) \otimes \I$. Neither of these equalities turn out to hold in general, a surprising fact that we explore more in this section. We will also study the question of density of $\mc{Z}\mrm{c}_{00}(\bbGamma)$ in $\mc{Z}\A(\bbGamma)$ (and weak$^*$-density 
of $\mc{Z}\mrm{c}_{00}(\bbGamma)$ in $\mc{Z} \B_r(\bbGamma)$), and prove that these properties hold for a very large class of discrete quantum groups, namely those which are unimodular or have the central AP. This is complemented by some observations regarding one-sided averaging on the level of the Fourier algebra.

\subsection{The centre of the Fourier algebra} 

Let $\bbGamma$ be a discrete quantum group and let $D(\bbGamma)$ be its Drinfeld double, following the notation of Section~\ref{section double}. From the form of $\ww^{D(\bbGamma)}$ we see that
\begin{equation}\label{eq:lambdadd}
\lambda_{D(\bbGamma)} \colon \LL^1(D(\bbGamma)) = \ell^1(\bbGamma) \wh\otimes \LL^1(\wh\bbGamma) \to \LL^\infty(\wh{D(\bbGamma)}); \quad
\omega\otimes\wh\omega \mapsto (\lambda(\omega)\otimes\I) Z^* (\I\otimes\wh\lambda(\wh\omega)) Z.
\end{equation}
Hence
\[ \mrm{C}_0(\wh{D(\bbGamma)}) = \overline{\lin} \big\{  (\wh x\otimes\I)Z^*(\I\otimes x)Z \,|\, x\in\mrm{c}_0(\bbGamma), \wh x\in\mrm{C}(\wh\bbGamma) \big\}, \] 
and so
\[ \LL^\infty(\wh{D(\bbGamma)}) = \overline{\lin}^{ \textnormal{weak}^*} \big\{  (\wh x\otimes\I)Z^*(\I\otimes x)Z \,|\, x\in\ell^\infty(\bbGamma), \wh x\in\LL^\infty(\wh\bbGamma) \big\}. \]
As $\pi = \eps\otimes\id \colon \LL^\infty(D(\bbGamma)) \to \LL^\infty(\wh\bbGamma)$, and
\[ (\id\otimes\wh\pi)(\wh\ww) = (\pi\otimes\id)(\ww^{D(\bbGamma)})
= (\eps\otimes\id\otimes\id_{D(\bbGamma)})(\ww_{13} Z_{34}^* \wh{\ww}_{24}Z_{34})
= Z^*_{23} \wh\ww_{13} Z_{23}, \]
it follows that
\[ \wh\pi \colon \ell^\infty(\bbGamma) \to \LL^\infty(\wh{D(\bbGamma)}); \quad
x \mapsto Z^*(\I\otimes x) Z. \]
This is the natural map identifying $\wh\bbGamma$ as a closed quantum subgroup of $D(\bbGamma)$, see \cite[Theorem~5.3]{Doublecrossed} and \cite[Lemma~7.12]{DKV_ApproxLCQG}.  Define
\[ p_z = Z^*(\I\otimes p_e)Z = \wh\pi(p_e)
= \lambda_{D(\bbGamma)}(\eps\otimes h) \in \LL^\infty(\wh{D(\bbGamma)}). \]

\begin{proposition}\label{prop:xidd}
The maps $\Xi_1\colon \LL^1(\wh{D(\bbGamma)}) \to \LL^1(\wh{D(\bbGamma)})$ and $\Xi_\infty \colon \LL^\infty(\wh{D(\bbGamma)}) \to \LL^\infty(\wh{D(\bbGamma)})$ are given by
\[ \Xi_1(\mu) = \mu\big(p_z \cdot p_z) \quad (\mu\in \LL^1(\wh{D(\bbGamma)})), \qquad
\Xi_\infty(y) = p_z y p_z \quad (y\in \LL^\infty(\wh{D(\bbGamma)})). \]
For $x\in\ell^\infty(\bbGamma)$ let $x_e\in\mathbb C$ be the entry of $x$ in the one-dimensional matrix block corresponding to $e\in\Irr(\wh\bbGamma)$.  For $\wh x\in\LL^\infty(\wh\bbGamma)$ we have
\begin{equation}
\Xi_\infty((\wh x\otimes \I)Z^*(\I\otimes x)Z) = x_e p_z (\wh x\otimes \I) p_z,
\label{eq:formxiinfty}
\end{equation}
and so the image of $\Xi_\infty$ is the weak$^*$-closure of $p_z(\LL^\infty(\wh\bbGamma)\otimes\I)p_z$.  Furthermore, $\Xi_\infty$ restricts to a map $\Xi_0 \colon \mrm{C}_0(\wh{D(\bbGamma)}) \to \mrm{C}_0(\wh{D(\bbGamma)})$ given by the same formula as \eqref{eq:formxiinfty}, and the image of $\Xi_0$ is the norm-closure of $p_z(\mrm{C}(\wh\bbGamma)\otimes \I)p_z$.
\end{proposition}
\begin{proof}
The forms of $\Xi_1$ and $\Xi_\infty$ follow immediately from Lemma~\ref{lem:xi1xiinf} and the definition of $p_z$.  Notice then that
\[ Z^*(\I\otimes x)Z p_z = Z^*(\I\otimes x)(\I\otimes p_e)Z
= x_e Z^*(\I\otimes p_e)Z = x_e p_z, \]
and from this \eqref{eq:formxiinfty} follows.  As $p_z$ is a projection, the image of $\Xi_\infty$, namely $p_z \LL^\infty(\wh{D(\bbGamma)}) p_z$, is weak$^*$-closed. Consequently \eqref{eq:formxiinfty} shows that the image of $\Xi_\infty$ is the weak$^*$-closure of the space $p_z(\LL^\infty(\wh\bbGamma)\otimes\I)p_z$, because $\LL^\infty(\wh{D(\bbGamma)})$ is the weak$^*$-closed linear span of elements of the form $(\wh x\otimes \I)Z^*(\I\otimes x)Z$.

Finally, as $p_z \in \mrm{C}_0(\wh{D(\bbGamma)})$, \eqref{eq:formxiinfty} also shows that $\Xi_\infty$ restricts to a bounded map on $\mrm{C}_0(\wh{D(\bbGamma)})$, and that the resulting map $\Xi_0$ has image equal to the norm-closure of $p_z(\mrm{C}(\wh\bbGamma)\otimes\I)p_z$.
\end{proof}

Recall from Corollary~\ref{corr:xi_to_fourier} that the averaging map $\Xi $ restricts to a contractive linear map on the level of the (reduced) Fourier-Stieltjes algebra. 

\begin{lemma}\label{lem:Xi0adj}
Consider $\Xi_0^* \colon \mrm{C}_0(\wh{D(\bbGamma)})^* \to \mrm{C}_0(\wh{D(\bbGamma)})^*$. Under the isomorphism $\mrm{C}_0(\wh{D(\bbGamma)})^* \cong \B_r(D(\bbGamma))$, this map 
identifies with the averaging map $\Xi\colon \B_r(D(\bbGamma))\to \B_r(D(\bbGamma))$.  
\end{lemma}

\begin{proof}
We wish to show that for $\mu \in \mrm{C}_0(\wh{D(\bbGamma)})^*$, we have that
\[ (\id\otimes\Xi_0^*(\mu))(\ww^{D(\bbGamma) *}) = \Xi\big( (\id\otimes\mu)(\ww^{D(\bbGamma) *}) \big). \]
There is a bounded net $(\omega_i)_{i\in I}$ in $\LL^1(\wh{D(\bbGamma)})$ which converges weak$^*$ to $\mu$.  As $\Xi_0$ is the restriction of $\Xi_\infty = \Xi_1^*$, it follows that $\Xi_0^*(\mu)$ is the weak$^*$-limit of the net $(\Xi_1(\omega_i))_{i\in I}$.  As left slices of $\ww^{D(\bbGamma) *}$ land in $\mrm{C}_0(\wh{D(\bbGamma)})$, taking right slices is a weak$^*$-continuous operation, and hence
\begin{align*}
(\id\otimes\Xi_0^*(\mu))(\ww^{D(\bbGamma) *})
&= \lim_{i\in I} (\id\otimes\Xi_1(\omega_i))(\ww^{D(\bbGamma) *})
= \lim_{i\in I} \Xi\big( (\id\otimes \omega_i)(\ww^{D(\bbGamma) *}) \big) \\
&= \Xi\big( \lim_{i\in I} (\id\otimes \omega_i)(\ww^{D(\bbGamma) *}) \big)
= \Xi\big( (\id\otimes\mu)(\ww^{D(\bbGamma) *}) \big),
\end{align*}
as required.  Here we use the defining relation between $\Xi_1$ and $\Xi$, and that $\Xi$ is weak$^*$-continuous on $\LL^\infty(D(\bbGamma))$.
\end{proof}

Consider $\A(\wh\GGamma\backslash D(\GGamma)/\wh\GGamma)$, the image of $\A(D(\bbGamma))$ under the averaging map $\Xi$.  By the definition of the maps involved, and Lemma~\ref{lem:Xi0adj}, we have the commutative diagrams
\[ \begin{CD}
\A(D(\GGamma))   @>{\Xi}>> \A(\wh\GGamma\backslash D(\GGamma)/\wh\GGamma) \\
@A{\lambda_{\wh{D(\bbGamma)}}}A{\cong}A  @A{\cong}A{\lambda_{\wh{D(\bbGamma)}}}A \\
\LL^1(\wh{D(\bbGamma)})   @>{\Xi_1}>>  \Xi_1\big( \LL^1(\wh{D(\bbGamma)}) \big).
\end{CD}
\qquad\qquad
\begin{CD}
\B_r(D(\GGamma))   @>{\Xi}>> \B_r(\wh\GGamma\backslash D(\GGamma)/\wh\GGamma) \\
@AA{\cong}A  @A{\cong}AA \\
\mrm{C}_0(\wh{D(\bbGamma)})^*   @>{\Xi_0^*}>>  \Xi_0^*\big( \mrm{C}_0(\wh{D(\bbGamma)})^* \big).
\end{CD}
\]

As $\LL^\infty(\wh\bbGamma)\otimes\I \subseteq \LL^\infty(\wh{D(\bbGamma)})$, the map $\LL^1(\wh{D(\bbGamma)}) \to \LL^1(\wh\bbGamma); \mu\mapsto \mu(\cdot\otimes\I)$ is well-defined, and hence the map $\iota$ in the following proposition is well-defined.  Similarly, as $\mrm{C}(\wh\bbGamma)\otimes\I \subseteq \M(\mrm{C}_0(\wh{D(\bbGamma)}))$ and bounded functionals on a \cst-algebra $\mc{A}$ extend uniquely to strictly continuous functionals on $\M(\mc{A})$ with the same norm, the map $\iota^r$ in the following is well-defined.

\begin{proposition}\label{prop:iotamaps}
Define $\iota \colon \Xi_1(\LL^1(\wh{D(\bbGamma)})) \to \LL^1(\wh\bbGamma)$ by $\iota(\mu) = \mu(\cdot\,\otimes\I)$.  Then $\iota$ maps into $\mc Z\LL^1(\wh\bbGamma)$ and
\[ \wh\lambda(\iota(\mu))\otimes\I = \lambda_{\wh{D(\bbGamma)}}(\mu) \qquad
( \mu\in\Xi_1(\LL^1(\wh{D(\bbGamma)})) ). \]
As such, $\A(\wh\GGamma\backslash D(\GGamma)/\wh\GGamma) \subseteq \mc Z\A(\bbGamma)\otimes\I$.
Similarly, the map $\iota^r \colon \Xi_0^*(\mrm{C}_0(\wh{D(\bbGamma)})^*) \to \mrm{C}(\wh\bbGamma)^*$ given by $\iota^r(\mu) = \mu(\cdot\otimes\I)$ maps into $\mc{Z}(\mrm{C}(\wh\bbGamma)^*)$ and satisfies 
\[ (\id\otimes\iota^r(\mu))(\ww^*) \otimes \I = (\id\otimes\mu)(\ww^{D(\bbGamma)*})
\qquad (\mu\in \Xi_0^*(\mrm{C}_0(\wh{D(\bbGamma)})) ). \]
As such, $\B_r(\wh\GGamma\backslash D(\GGamma)/\wh\GGamma) \subseteq \mc Z\B_r(\bbGamma)\otimes\I$.
\end{proposition}

\begin{proof}
We first claim that
\begin{equation} \label{eq:xiinfaction}
\Xi_\infty\big( (\omega\otimes\wh\omega\otimes\id_{D(\bbGamma)})(\ww^{D(\bbGamma)*}) \big)
= \wh\omega(\I) p_z ((\omega\otimes\id)(\ww^*)\otimes\I) p_z
\qquad (\omega\in\ell^1(\bbGamma), \wh\omega\in\LL^1(\wh\bbGamma)).
\end{equation}
Indeed, by \eqref{eq:formxiinfty}, if we set $x=\wh\lambda(\wh\omega)$, then $x_e = \wh\omega(\I)$ and so
\begin{align*}
\Xi_\infty\big( (\omega\otimes\wh\omega\otimes\id_{D(\bbGamma)})(\ww^{D(\bbGamma)}) \big)
&= \Xi_\infty(\lambda_{D(\bbGamma)}(\omega\otimes\wh\omega)) \\
&= x_e p_z(\lambda(\omega)\otimes\I)p_z
= \wh\omega(\I)p_z((\omega\otimes\id)(\ww)\otimes\I)p_z.
\end{align*}
Taking adjoints, and using that $\Xi_\infty$ is a $*$-map, yields \eqref{eq:xiinfaction}.

Let $\mu\in\Xi_1(\LL^1(\wh{D(\bbGamma)}))$.
As $\mu = \Xi_1(\mu) = \mu\circ\Xi_\infty$, using \eqref{eq:xiinfaction},
\begin{align*}
\la (\omega\otimes\wh\omega\otimes\id_{D(\bbGamma)})(\ww^{D(\bbGamma)*}), \mu \ra
&= \wh\omega(\I) \la p_z ((\omega\otimes\id)(\ww^*)\otimes\I) p_z, \mu \ra \\
&= \wh\omega(\I) \la (\omega\otimes\id)(\ww^*)\otimes\I, \mu (p_z\cdot p_z) \ra \\
&= \wh\omega(\I) \la (\omega\otimes\id)(\ww^*)\otimes\I, \mu \ra,
\end{align*}
where in the final step we again use that $\Xi_1(\mu)=\mu$.  By the definition of $\iota$ this equals
\begin{align*}
&\quad\;
\la (\omega\otimes\wh\omega\otimes\id_{D(\bbGamma)})(\ww^{D(\bbGamma)*}), \mu \ra=
\la (\omega\otimes\id)(\ww^*), \iota(\mu) \ra \wh\omega(\I)\\
&= \la (\id\otimes\iota(\mu))(\ww^*), \omega \ra \wh\omega(\I)
= \la \wh\lambda(\iota(\mu))\otimes\I, \omega\otimes\wh\omega \ra.
\end{align*}
As this holds for all $\omega,\wh\omega$, we have shown that
\[ \lambda_{\wh{D(\bbGamma)}}(\mu)
= (\id_{D(\bbGamma)}\otimes\mu)(\ww^{D(\bbGamma)*})
= \wh\lambda(\iota(\mu))\otimes \I, \]
as claimed.

As $\lambda_{\wh{D(\bbGamma)}}\circ\Xi_1 = \Xi\circ\lambda_{\wh{D(\bbGamma)}}$, we have shown that $\A(\wh\GGamma\backslash D(\GGamma)/\wh\GGamma) = \Xi(\A(D(\bbGamma)))$ is equal to
\[ \{ \wh\lambda(\iota(\mu)) \otimes \I \,|\, \mu\in \LL^1(\wh{D(\bbGamma)}),\,\mu=\Xi_1(\mu) \} \subseteq \A(\bbGamma) \otimes \I. \]
However, by Proposition~\ref{averagingdd}, $\A(\wh\GGamma\backslash D(\GGamma)/\wh\GGamma) \subseteq \Xi(\mrm{C}_0(D(\bbGamma))) = \mc Z\mrm{c}_0(\bbGamma)\otimes\I$.  Hence $\wh\lambda(\iota(\mu)) \in \A(\bbGamma)\cap\mc Z\mrm{c}_0(\bbGamma) = \mc Z\A(\bbGamma)$ for each $\mu=\Xi_1(\mu)$, and so $\iota(\mu)\in\mc Z\LL^1(\wh\bbGamma)$ for such $\mu$.

Exactly the same argument works for $\mu \in \Xi_0^*( \mrm{C}_0(\wh{D(\bbGamma)})^* )$ yielding the claims about $\iota^r$ and $\B_r(\wh\GGamma\backslash D(\GGamma)/\wh\GGamma)$.
\end{proof}

We can now give a simple criterion for when $\A(\wh\GGamma\backslash D(\GGamma)/\wh\GGamma) = \mc Z\A(\bbGamma)\otimes\I$.

\begin{theorem}\label{thm:main_fourierinv}
For any discrete quantum group $\bbGamma$ the following are equivalent:
\begin{enumerate}
\item[1)] $\iota\colon \Xi_1(\LL^1(\wh{D(\bbGamma)})) \to \mc{Z}\LL^1(\wh\bbGamma)$ is an isomorphism;
\item[2)] $\A(\wh\GGamma\backslash D(\GGamma)/\wh\GGamma) = \mc Z\A(\bbGamma)\otimes\I$.
\end{enumerate}
Furthermore, the following are equivalent:
\begin{enumerate}
\item[3)] $\iota^r\colon \Xi_0^*(\mrm{C}_0(\wh{D(\GGamma)})^*)\rightarrow \mc{Z} (\mrm{C}(\wh\GGamma)^*)$ is an isomorphism;
\item[4)] $\B_r(\wh\GGamma\backslash D(\GGamma)/\wh\GGamma) = \mc Z\B_r(\bbGamma)\otimes\I$.
\end{enumerate}
\end{theorem}
\begin{proof}
$ 1) \Leftrightarrow 2) $ As in Proposition~\ref{prop:iotamaps}, we identify $\A(\wh\GGamma\backslash D(\GGamma)/\wh\GGamma)$ with the image of $\Xi_1$, and identify $\mc Z\A(\wh\bbGamma) \otimes \I$ with $\mc Z\LL^1(\wh\bbGamma)$, so that $\iota$ gives the inclusion $
\A(\wh\GGamma\backslash D(\GGamma)/\wh\GGamma) \subseteq \mc Z\A(\bbGamma)\otimes\I$.  Thus $\iota$ is an isomorphism if and only if $\A(\wh\GGamma\backslash D(\GGamma)/\wh\GGamma) = \mc Z\A(\bbGamma)\otimes\I$. \\
$ 3) \Leftrightarrow 4) $ The proof is entirely analogous.
\end{proof}

We wish to identify the dual space of the image of $\Xi_1$, for which the next (well-known and standard) lemma is helpful.

\begin{lemma}\label{lem:dualprojimage}
Let $E$ be a Banach space, let $P\in\B(E)$ be a contractive idempotent and set $F=P(E)$. Then $F$ is a closed subspace of $E$ and $F^* \cong P^*(E^*)$ isometrically, for the natural pairing 
between $P^*(E^*)\subseteq E^*$ and $F\subseteq E$.
\end{lemma}

\begin{proof}
By the Hahn-Banach Theorem, we identify $F^*$ with $E^* / F^\perp$ where $F^\perp = \{ \mu\in E^* \,|\, \mu(x) = 0 \ (x\in F) \}$.  Notice that for $\mu\in E^*$, we have that $\mu(x)=0$ for all $x\in F$ if and only if $P^*(\mu) = \mu\circ P = 0$, and so $F^\perp = \ker P^*$.  As we have the algebraic direct sum $E^* = P^*(E^*) \oplus \ker P^*$, we identify $E^* / F^\perp$ with $P^*(E^*)$ in a way which respects the dual pairing.  It remains to show that this identification is isometric.  For $\mu\in E^*$ we have
\begin{align*}
\|P^*(\mu)\| &= \sup\{ |P^*(\mu)(x)| \,|\, x\in E, \|x\|\leq 1 \}
= \sup\{ |\mu(P(x))| \,|\, x\in E, \|x\|\leq 1 \} \\
&= \sup\{ |\mu(y)|\,|\, y\in F, \|y\|\leq 1 \},
\end{align*}
as $\{ P(x)\,|\, x\in E, \|x\|\leq 1 \} = \{ y\in F \,|\, \|y\|\leq 1 \}$ because $P$ is contractive.  Hence $\|P^*(\mu)\| = \|\mu|_F\| = \|\mu+F^\perp\|_{E^*/F^\perp}$ as claimed.
\end{proof}

In the following, for a Banach space $E$ and closed subspaces $F\subseteq E$ and $G\subseteq E^*$ we define
\[ F^\perp = \{ \mu\in E^* \,|\, \mu(x)=0 \ (x\in F) \}, \qquad
{}^\perp G = \{ x\in E \,|\, \mu(x)=0 \ (\mu\in G) \}. \]
By the Hahn-Banach Theorem, we have the natural isomorphism $F^* = E^*/F^\perp$, and that $({}^\perp G)^\perp$ is the weak$^*$-closure of $G$. Furthermore, there is a canonical isometric identification $(E/F)^*=F^{\perp}$.

\begin{proposition}\label{prop:iotaontorange}
The Banach space adjoint $\iota^* \colon (\mc Z\LL^1(\wh\bbGamma))^* \to \Xi_\infty(\LL^\infty(\wh{D(\bbGamma)}))$ is the map
\[ (\mc Z\LL^1(\wh\bbGamma))^*  = \LL^\infty(\wh\bbGamma) / (\mc Z\LL^1(\wh\bbGamma))^\perp
\ni \wh x + (\mc Z\LL^1(\wh\bbGamma))^\perp \mapsto p_z(\wh x\otimes\I)p_z, \]
which has weak$^*$-dense image.

Furthermore, $\iota^r$ is weak$^*$-weak$^*$-continuous, and the Banach space pre-adjoint map is
\[ \iota^r_* \colon \mrm{C}(\wh\bbGamma) / {}^\perp(\mc{Z}(\mrm{C}(\wh\bbGamma)^*)) \to
\Xi_0(\mrm{C}_0(\wh{D(\bbGamma)})); \quad \wh x+{}^\perp(\mc Z(\mrm{C}(\wh\bbGamma)^*))
\mapsto p_z(\wh x\otimes\I)p_z, \]
which has norm-dense image.
\end{proposition}
\begin{proof}
By definition, $\iota^* \colon \LL^\infty(\wh\bbGamma) \to \Xi_1(\LL^1(\wh{D(\bbGamma)}))^*$, and so the identification of $\Xi_1(\LL^1(\wh{D(\bbGamma)}))^*$ with $\Xi_\infty(\LL^\infty(\wh{D(\bbGamma)}))$ given by Lemma~\ref{lem:dualprojimage} shows that $\iota^*$ has the stated codomain.  The given formula for $\iota^*$ now follows easily from the definition of $\iota$.  From Proposition~\ref{prop:xidd}, the image of $\Xi_\infty$ is the weak$^*$-closure of $p_z(\LL^\infty(\wh\bbGamma)\otimes\I)p_z$, and thus $\iota^*$ has weak$^*$-dense image.

We observe that $\iota^r$ is weak$^*$-weak$^*$-continuous by definition, and that $\mrm{C}(\wh\bbGamma) / {}^\perp(\mc{Z}(\mrm{C}(\wh\bbGamma)^*))$ is the canonical Banach space pre-dual of $\mc{Z}(\mrm{C}(\wh\bbGamma)^*)$. One now gives the analogous argument for $\iota^r$.
\end{proof}

Motivated by Proposition~\ref{prop:iotaontorange}, we wish to study the map $\wh x \mapsto p_z(\wh x\otimes\I) p_z$ more closely.  By \cite[Section~7.1]{Brannan} there exists a normal UCP map $\wh\Delta^\sharp \colon \LL^\infty(\wh\bbGamma) \bar\otimes \LL^\infty(\wh\bbGamma) \to \LL^\infty(\wh\bbGamma)$ which satisfies
\[ \wh\Delta^\sharp\big( U^\alpha_{i,j} \otimes U^\beta_{k,l} \big)
= \delta_{\alpha,\beta} \delta_{j,k} \frac{1}{\dim_q(\alpha)} U^\alpha_{i,l}. \]
Define $\mf Q \colon \LL^\infty(\wh\bbGamma) \to \LL^\infty(\wh\bbGamma)$ by $\mf Q = \wh\Delta^\sharp \circ \wh\Delta^{\oon{op}}$.  Thus $\mf Q$ is normal UCP with
\begin{equation}\label{eq:defn_mfQ}
\mf Q(U^\alpha_{i,j}) = \delta_{i,j} \frac{1}{\dim_q(\alpha)} \sum_{k=1}^{\dim(\alpha)} U^\alpha_{k,k}
= \delta_{i,j} \frac{1}{\dim_q(\alpha)} \chi_\alpha.
\end{equation}
From this formula, it is clear that $\mf Q$ restricts to a UCP map $\mf Q_0 \colon \mrm{C}(\wh\bbGamma) \to \mrm{C}(\wh\bbGamma)$.

\begin{proposition}\label{prop:p_z_to_Q}
For any $\wh x\in \LL^\infty(\wh\bbGamma)$ we have that
\[ p_z(\wh x\otimes\I)p_z = Z^*(\mf Q(\wh x)\otimes p_e)Z. \]
\end{proposition}

\begin{proof}
Pre- and post-multiplying by $Z$ and $Z^*$, respectively, shows that the claim is equivalent to
\[ (\I\otimes p_e) Z(\wh x\otimes\I) Z^* (\I\otimes p_e) = \mf Q(\wh x)\otimes p_e
\qquad (\wh x\in \LL^\infty(\wh\bbGamma)). \]
We have $p_e \xi= ( \Lambda_h(\I) | \xi )\Lambda_h(\I)\,(\xi\in \LL^2(\wh\bbGamma))$, and so for any $T\in\B(\LL^2(\wh\bbGamma))$ we have $p_e T p_e = (\Lambda_h(\I)|T \Lambda_h(\I)) p_e$.  Thus the claim is equivalent to
\[ (\id\otimes \omega_{\Lambda_h(\I)}) \big(  Z(\wh x\otimes\I) Z^* \big) = \mf Q(\wh x)
\qquad (\wh x\in \LL^\infty(\wh\bbGamma)). \]
By weak$^*$-continuity, it suffices to check this claim for $\wh x\in\Pol(\wh\bbGamma)$, and then by linearity, it suffices to consider $\wh x = U^\alpha_{i,j}$ for $\alpha\in\Irr(\wh\bbGamma),1\le i,j\le \dim(\alpha)$.

Using
\begin{equation}\label{eq:op_comult}
\ww(\wh x\otimes\I)\ww^* = \Sigma \wh\ww^*(\I\otimes\wh x)\wh\ww \Sigma 
= \wh\Delta^{\oon{op}}(\wh x) \qquad (\wh x \in \LL^\infty(\wh\bbGamma))
\end{equation}
we compute
\begin{align*}
Z(U^\alpha_{i,j}\otimes\I)Z^*
&= (J\otimes\wh J)\ww (J\otimes\wh J) \ww(U^\alpha_{i,j}\otimes\I)\ww^* (J\otimes\wh J)\ww^* (J\otimes\wh J) \\
&= (J\otimes\wh J) \ww \Big(\sum_{k=1}^{\dim(\alpha)} J U^\alpha_{k,j} J \otimes \wh J U^\alpha_{i,k} \wh J\,\Big)
\ww^* (J\otimes\wh J).
\end{align*}
We next use that
\begin{equation}\label{eq14}
 J U^\alpha_{k,j} J = \wh R(U^\alpha_{k,j})^* = \wh\tau_{-i/2} (\wh S(U^\alpha_{k,j})^*) = \wh\tau_{-i/2}(U^\alpha_{j,k})
= \sum_{l,m=1}^{\dim(\alpha)} (\uprho_\alpha^{1/2})_{j,l}  (\uprho_\alpha^{-1/2})_{m,k} U^\alpha_{l,m},  
\end{equation}
and also observe that $\wh J U^\alpha_{i,k} \wh J \in \LL^\infty(\wh\bbGamma)'$ so $\I\otimes \wh J U^\alpha_{i,k} \wh J$ commutes with $\ww^*$. 
Using \eqref{eq:op_comult} again, we thus obtain, 
\begin{align}
Z(U^{\alpha}_{i,j}\otimes \I)Z^*
&=\sum_{k,l,m=1}^{\dim(\alpha)}  \overline{(\uprho_\alpha^{1/2})_{j,l}} \overline{(\uprho_\alpha^{-1/2})_{m,k}}
(J\otimes\wh J) \ww (U^\alpha_{l,m}\otimes\I) \ww^* (\I\otimes\wh J U^\alpha_{i,k} \wh J)(J\otimes\wh J) \notag \\
&= \sum_{k,l,m,n=1}^{\dim(\alpha)} (\uprho_\alpha^{1/2})_{l,j}  (\uprho_\alpha^{-1/2})_{k,m}
(J\otimes\wh J) (U^\alpha_{n,m} \otimes U^\alpha_{l,n})
(J\otimes\wh J)(\I\otimes U^\alpha_{i,k}) \notag \\
&= \sum_{k,l,m,n=1}^{\dim(\alpha)} (\uprho_\alpha^{1/2})_{l,j}  (\uprho_\alpha^{-1/2})_{k,m}
J U^\alpha_{n,m} J \otimes \wh JU^\alpha_{l,n}\wh J U^\alpha_{i,k} \notag \\ 
&= \sum_{k,l,m,n,r,s=1}^{\dim(\alpha)} (\uprho_\alpha^{1/2})_{l,j}  (\uprho_\alpha^{-1/2})_{k,m}
(\uprho_\alpha^{1/2})_{m,r} (\uprho_\alpha^{-1/2})_{s,n} U^\alpha_{r,s}
\otimes \wh JU^\alpha_{l,n}\wh J U^\alpha_{i,k}.
\label{eq15}
\end{align}
Here we used that $\uprho_\alpha$ is a self-adjoint matrix and equation \eqref{eq14}.
Next observe that $\wh J  \Lambda_h(a) = \Lambda_h(\wh\sigma_{i/2}(a)^*)$ for $a\in\Pol(\wh\bbGamma)$, and thus
\begin{align*}
\big( \Lambda_h(\I ) \big| \wh JU^\alpha_{l,n}\wh J U^\alpha_{i,k} \Lambda_h(\I) \big)
&= \big( \Lambda_h(\I ) \big| U^\alpha_{i,k} \wh JU^\alpha_{l,n}\wh J \Lambda_h(\I) \big) 
= \big( \Lambda_h((U^\alpha_{i,k})^*) \big| \Lambda_h(\wh \sigma_{i/2}(U^\alpha_{l,n})^*) \big) \\
&= h\big(U^\alpha_{i,k} \wh\sigma_{i/2}(U^\alpha_{l,n})^* \big)
= \sum_{t,u=1}^{\dim(\alpha)}  \overline{(\uprho_\alpha^{-1/2})_{l,t} (\uprho_\alpha^{-1/2})_{u,n}}
h\big( U^\alpha_{i,k} (U^\alpha_{t,u})^* \big) \\
&= \sum_{t,u=1}^{\dim(\alpha)}  (\uprho_\alpha^{-1/2})_{t,l} (\uprho_\alpha^{-1/2})_{n,u}
\delta_{i,t} \frac{(\uprho_\alpha)_{u,k}}{\dim_q(\alpha)}
= (\uprho_\alpha^{-1/2})_{i,l} \frac{(\uprho_\alpha^{1/2})_{n,k}}{\dim_q(\alpha)}.
\end{align*}
Hence applying $\id\otimes\omega_{\Lambda_h(\I)}$ to \eqref{eq15} gives
\begin{align*}
&\quad\;
(\id\otimes \omega_{\Lambda_h(\I)})\bigl( 
Z(U^{\alpha}_{i,j}\otimes \I)Z^*\bigr) \\
&=
\sum_{k,l,m,n,r,s=1}^{\dim(\alpha)}  (\uprho_\alpha^{1/2})_{l,j}  (\uprho_\alpha^{-1/2})_{k,m}
(\uprho_\alpha^{1/2})_{m,r} (\uprho_\alpha^{-1/2})_{s,n} U^\alpha_{r,s}
(\uprho_\alpha^{-1/2})_{i,l} \frac{(\uprho_\alpha^{1/2})_{n,k}}{\dim_q(\alpha)} \\
&= \frac{1}{\dim_q(\alpha)} \delta_{i,j}
\sum_{k=1}^{\dim(\alpha)} U^\alpha_{k,k} = \mf Q(U^\alpha_{i,j}),
\end{align*}
as required.
\end{proof}

\begin{corollary}\label{corr:links_centreL1_Q}
The maps $\iota^*$ and $\iota^r_*$ are given by
\begin{gather*}
\iota^* \colon (\mc Z\LL^1(\wh\bbGamma))^* = \LL^\infty(\wh\bbGamma) / (\mc Z\LL^1(\wh\bbGamma))^\perp \to \Xi_\infty(\LL^\infty(\wh{D(\bbGamma)})); \quad
\wh x + (\mc Z\LL^1(\wh\bbGamma))^\perp \mapsto Z^*(\mf Q(\wh x)\otimes p_e)Z, \\
\iota^r_* \colon \mrm{C}(\wh\bbGamma) / {}^\perp(\mc{Z}(\mrm{C}(\wh\bbGamma)^*)) \to
\Xi_0(\mrm{C}_0(\wh{D(\bbGamma)})); \quad \wh x+{}^\perp(\mc Z(\mrm{C}(\wh\bbGamma)^*))
\mapsto Z^*(\mf Q_0(\wh x)\otimes\I)Z.
\end{gather*}
In particular $(\mc Z\LL^1(\wh\bbGamma))^\perp \subseteq \ker\mf Q$ and ${}^\perp(\mc{Z}(\mrm{C}(\wh\bbGamma)^*)) \subseteq \ker\mf Q_0$.
Furthermore, 
\begin{itemize}
\item $\iota^*$ is injective if and only if $(\mc Z\LL^1(\wh\bbGamma))^\perp = \ker\mf Q$,
\item $\iota_r^*$ is injective if and only if ${}^\perp(\mc{Z}(\mrm{C}(\wh\bbGamma)^*)) = \ker\mf Q_0$,
\end{itemize}
and
\begin{itemize}
\item $\iota^*$ is surjective if and only if $\mf Q$ has weak$^*$-closed image,
\item $\iota_r^*$ is surjective if and only if $\mf Q_0$ has norm-closed image.
\end{itemize}
\end{corollary}
\begin{proof}
The formulas for $\iota^*$ and $\iota^r_*$ follow immediately from Proposition~\ref{prop:p_z_to_Q} and Proposition~\ref{prop:iotaontorange}. In particular, as $\iota^*$ is well-defined, it follows that $(\mc Z\LL^1(\wh\bbGamma))^\perp \subseteq \ker\mf Q$, and injectivity of $ \iota^* $ is equivalent to $(\mc Z\LL^1(\wh\bbGamma))^\perp = \ker\mf Q$. 
As $\iota^*$ has weak$^*$-dense image, and tensoring with $p_e$ and conjugating by $Z$ does not change being weak$^*$-closed, it follows that $\iota^*$ is surjective if and only if $\mf Q$ has weak$^*$-closed image.  The claims for $\mf Q_0$ are analogous.
\end{proof}

\subsection{One-sided averaging}

We continue Remark~\ref{rem:dd_one_sided}, where we showed that $\mrm{C}_0(D(\bbGamma)/\wh\bbGamma) = \{ (\id\otimes h)(x)\otimes\I \,|\, x\in\mrm{C}_0(D(\bbGamma)) \} = \mrm{c}_0(\bbGamma) \otimes \I$, and noticed that the averaging map from $\mrm{C}_0(D(\bbGamma)/\wh\bbGamma)$ to $\mrm{C}_0(\wh\bbGamma \backslash D(\bbGamma)/\wh\bbGamma)$ is given by $x\otimes \I \mapsto A(x)\otimes \I$, where $A$ is the map discussed at the start of Section~\ref{sec:unimodular}. 

The one-sided averaging map $ \Xi^r\colon \mrm{C}_0(D(\bbGamma)) \to \mrm{C}_0(D(\bbGamma)/\wh\bbGamma)$ is associated to the map $\Xi^r_1 \colon$ $ \LL^1(\wh{D(\bbGamma)}) \to \LL^1(\wh{D(\bbGamma)})$ given by $\mu\mapsto \mu(p_z\cdot) = \mu p_z$; compare with Remark~\ref{rem:onesided4} and Proposition~\ref{prop:xidd}.
The adjoint map is $\Xi^r_\infty \colon \LL^\infty(\wh{D(\bbGamma)}) \to \LL^\infty(\wh{D(\bbGamma)}); x\mapsto p_zx$ which maps $Z^*(\I\otimes x)Z(\wh x\otimes \I)$ to $x_e p_z(\wh x\otimes\I)$, and hence the image of $\Xi^r_\infty$ is the weak$^*$-closure of $\{ p_z(\wh x\otimes\I) \,|\, \wh x\in \LL^\infty(\wh\bbGamma)\}$. 

It seems difficult to express elements of the form $p_z(\wh x\otimes\I)$ in a simple way, so there seems to be no version of Proposition~\ref{prop:p_z_to_Q} in this setting.  However, we do have the following version of Proposition~\ref{prop:iotamaps}.

\begin{proposition}\label{prop2}
Define $\alpha \colon \{ \mu p_z \,|\, \mu\in \LL^1(\wh{D(\bbGamma)}) \} \to \LL^1(\wh\bbGamma)$ by $\alpha(\mu) = \mu(\cdot\otimes\I)$.  Then
\[ \wh\lambda(\alpha(\mu)) \otimes \I = \lambda_{\wh{D(\bbGamma)}}(\mu)
\qquad (\mu = \mu p_z \in \Xi_1^r(\LL^1(\wh{D(\bbGamma)}))). \]
As such,
\begin{equation}\label{eq16}
\A(D(\bbGamma)/\wh\bbGamma) \subseteq \A(\bbGamma) \otimes \I,
\end{equation}
with equality if and only if $\alpha$ is surjective.
\end{proposition}
\begin{proof}
Let $\mu\in \LL^1(\wh{D(\bbGamma)})$ with $\mu = \mu p_z$.  Then we obtain
\begin{align*}
\lambda_{\wh{D(\bbGamma)}}(\mu)
&= (\id\otimes\mu)(\ww^{D(\bbGamma) *})
= (\id\otimes\mu)((p_z)_{34}Z^*_{34} \wh{\ww}_{24}^* Z_{34} \ww_{13}^*) \\
&= (\id\otimes\mu)(Z^*_{34} (\I^{\otimes 3}\otimes p_e)\wh \ww_{24}^* Z_{34} \ww_{13}^*)
= (\id\otimes\mu)(Z^*_{34} (\I^{\otimes 3}\otimes p_e) Z_{34} \ww_{13}^*) \\
&= (\id\otimes\mu)((p_z)_{34} \ww_{13}^*)
= (\id\otimes\mu)(\ww_{13}^*)
= (\id\otimes\alpha(\mu))(\ww^*) \otimes \I
= \wh\lambda(\alpha(\mu)) \otimes \I
\end{align*}
as claimed, where we use that $(\I\otimes p_e)\wh\ww^* = \I\otimes p_e$.

By construction, we have $\Xi^r \circ \lambda_{\wh{D(\bbGamma)}} = \lambda_{\wh{D(\bbGamma)}} \circ \Xi^r_1$, and so
\[ \A(D(\bbGamma)/\wh\bbGamma) = \{ \lambda_{\wh{D(\bbGamma)}}(\mu) \,|\, \mu=\mu p_z \}
= \{ \wh\lambda(\alpha(\mu)) \otimes \I \,|\, \mu=\mu p_z \}
\subseteq \A(\bbGamma) \otimes \I. \]
Moreover, the last inclusion is an equality exactly when $\alpha$ is onto.
\end{proof}

As before, we shall study the Banach space adjoint $\alpha^*$.  As noted in Remark~\ref{rem:dd_one_sided}, we have $\mrm{c}_{00}(\bbGamma)\otimes\I = \Xi^r(\mrm{c}_{00}(\GGamma) \odot \Pol(\wh\GGamma)) \subseteq \Xi^r(\A(D(\bbGamma)))$, and so $\alpha$ has dense range, or equivalently, $\alpha^*$ is injective. Furthermore, we have
\[ \alpha^* \colon \LL^\infty(\wh\bbGamma) \to \{ p_zx \,|\, x\in \LL^\infty(\wh{D(\bbGamma)}) \}
= \{ p_z(\wh x\otimes\I) \,|\, \wh x \in \LL^\infty(\wh\bbGamma) \}^{- \textnormal{weak}^*};
\quad \wh x \mapsto  p_z(\wh x\otimes\I),
\]
here again using Lemma~\ref{lem:dualprojimage} to identify the dual of $\LL^1(\wh{D(\bbGamma)}) p_z$ with $p_z \LL^\infty(\wh{D(\bbGamma)})$. This implies immediately that $\alpha^*$ has weak$^*$-dense range, or equivalently, $\alpha$ is injective. So $\alpha$ is surjective if and only if $\alpha$ is an isomorphism, by the Open Mapping Theorem, and this in turn is equivalent to $\alpha^*$ being an isomorphism, equivalently, $\alpha^*$ being bounded below.
With this in mind, notice that for $\wh x \in \LL^\infty(\wh\bbGamma)$ we have
\begin{equation} \label{eq:alpha_star_Q1}
\|\alpha^*(\wh x)\|^2 = \|p_z(\wh x\otimes\I)\|^2
= \| p_z(\wh x\wh x^*\otimes\I)p_z \|
= \| \mf Q(\wh x\wh x^*) \|,
\end{equation}
the last equality following readily from Proposition~\ref{prop:p_z_to_Q}.

Recall from Section~\ref{sec:unimodularsub} that if $\bbGamma$ is unimodular then $\mf Q = F$ is the Haar-trace-preserving conditional expectation 
of $\LL^\infty(\wh\bbGamma)$ onto $\mathscr{C}_{\wh\bbGamma}$.

\begin{lemma}\label{lem:norm_to_cstar_ineq}
Let $\bbGamma$ be unimodular.  For $\delta>0$, the following are equivalent:
\begin{enumerate}[label=\arabic*)]
\item\label{lem:norm_to_cstar_ineq:one}
$\|\alpha^*(\wh x)\| \geq \delta \|\wh x\|$ for each $\wh x\in \LL^\infty(\wh\bbGamma)$;
\item\label{lem:norm_to_cstar_ineq:onea}
$\|\mf Q(\wh x)\| \geq \delta^2 \|\wh x\|$ for each $\wh x\in \LL^\infty(\wh\bbGamma)$ with $\wh x\geq 0$;
\item\label{lem:norm_to_cstar_ineq:oneb}
$\mf Q(\wh x) \geq \delta^2 \wh x$ for each $\wh x\in \LL^\infty(\wh\bbGamma)$ with $\wh x\geq 0$;
\item\label{lem:norm_to_cstar_ineq:two}
$\|\alpha^*(\wh x)\| \geq \delta \|\wh x\|$ for each $\wh x\in \mrm{C}(\wh\bbGamma)$;
\item\label{lem:norm_to_cstar_ineq:twoa}
$\|\mf Q_0(\wh x)\| \geq \delta^2 \|\wh x\|$ for each $\wh x\in \mrm{C}(\wh\bbGamma)$ with $\wh x\geq 0$;
\item\label{lem:norm_to_cstar_ineq:twob}
$\mf Q_0(\wh x) \geq \delta^2 \wh x$ for each $\wh x\in \mrm{C}(\wh\bbGamma)$ with $\wh x\geq 0$;
\end{enumerate}
\end{lemma}
\begin{proof}
\ref{lem:norm_to_cstar_ineq:one} $ \Leftrightarrow $ \ref{lem:norm_to_cstar_ineq:onea} and \ref{lem:norm_to_cstar_ineq:two} $\Leftrightarrow$ \ref{lem:norm_to_cstar_ineq:twoa}
follow from \eqref{eq:alpha_star_Q1}. 

\ref{lem:norm_to_cstar_ineq:onea} $\Rightarrow$ \ref{lem:norm_to_cstar_ineq:twoa} follows from the fact that $\mf Q$ restricts to a map $\mf{Q}_0\colon \mrm{C}(\wh\bbGamma) \to \mrm{C}(\wh\bbGamma)$.

\ref{lem:norm_to_cstar_ineq:onea} $\Rightarrow$ \ref{lem:norm_to_cstar_ineq:oneb} Given $\wh x\in\LL^\infty(\wh\bbGamma)$ and $\epsilon>0$, notice that
\[ (\epsilon \I + \mf Q(\wh x^*\wh x))^{-1/2} 
\mf Q(\wh x^*\wh x)
(\epsilon \I+ \mf Q(\wh x^*\wh x))^{-1/2} \leq \I. \]
As $\mf Q$ is a conditional expectation, it is a bimodule map.  Set $\wh y = \wh x (\epsilon \I+ \mf Q(\wh x^*\wh x))^{-1/2}$, so the previous equation shows that $\mf Q(\wh y^* \wh y) \leq \I$.  Thus $1 \geq \|\mf Q(\wh y^* \wh y)\| \geq \delta^2 \|\wh y^*\wh y\|$, so $\wh y^*\wh y \leq \delta^{-2} \I$.  Multiplying by $(\epsilon \I+ \mf Q(\wh x^*\wh x))^{1/2}$ on both sides yields
\[ \wh x^*\wh x \leq \delta^{-2} (\epsilon \I+ \mf Q(\wh x^*\wh x)) \quad\text{for all }\epsilon>0, 
\]
and hence $\mf Q(\wh x^*\wh x) \geq \delta^2 \wh x^*\wh x $ as claimed. 

\ref{lem:norm_to_cstar_ineq:twoa} $\Rightarrow $ \ref{lem:norm_to_cstar_ineq:twob} is analogous, and the reverse implications \ref{lem:norm_to_cstar_ineq:oneb} $\Rightarrow$ \ref{lem:norm_to_cstar_ineq:onea} and \ref{lem:norm_to_cstar_ineq:twob} $\Rightarrow$ \ref{lem:norm_to_cstar_ineq:twoa} are clear.

\ref{lem:norm_to_cstar_ineq:twob}$ \Rightarrow $\ref{lem:norm_to_cstar_ineq:oneb} follows from Kaplansky density.  Indeed, given $\wh x\in\LL^\infty(\wh\bbGamma)$ with $\wh x\geq 0$, there is a bounded net of positive elements $(a_i)_{i\in I}$ in $\mrm{C}(\wh\bbGamma)$ with $a_i \xrightarrow[i\in I]{} \wh x$ strongly, see \cite[Corollary~5.3.6]{KadisonRingrose} for example.  Then $(\mf Q(a_i)-\delta^2 a_i)_{i\in I}$ is a net of positive elements, by \ref{lem:norm_to_cstar_ineq:twob}, which converges weak$^*$ to $\mf Q(\wh x)-\delta^2 \wh x$, showing that $\mf Q(\wh x)-\delta^2 \wh x\geq 0$, as required.
\end{proof}

Now let $G$ be a compact group (always assumed to be Hausdorff) and set $\bbGamma = \wh G$.  Then $\mf Q$ is the conditional expectation of $\LL^\infty(G)$ onto the class function algebra
\[ \mathscr{C}_{G} = \{ f\in \LL^\infty(G) \,|\, f(t\cdot t^{-1}) = f(\cdot) \ (t\in G) \}, \]
and similarly at the $C^*$-algebra level, leading to $\mathscr{C}^0_G \subseteq \mrm{C}(G)$.  It is now readily seen that $\mf Q$ takes the form
\[ \mf Q(f) = \int_G f(t\cdot t^{-1}) \md t \qquad (f\in \LL^\infty(G)), \]
see \cite[Lemma~1.2]{KW_ClassFuncs} for example.
For $r\in G$ we denote by $\conjc(r) = \{ trt^{-1} \,|\, t\in G\}$ the conjugacy class of $r$, and we write $C_G(r) = \{ t\in G \,|\, tr=rt \}$ for the centraliser of $r$ in $G$. Since conjugation yields a continuous action of $G$ on itself, $\conjc(r)$ is a closed subset of $G$, and the stabiliser $C_G(r)$ of $ r \in G $ is a closed subgroup of $G$. 
The Orbit--Stabiliser theorem gives a bijection $G / C_G(r) \to \conjc(r); t C_G(r) \mapsto trt^{-1}$.

\begin{proposition}\label{prop:for_cmpt_gp}
Let $G$ be a compact group.  For $\delta>0$, the following are equivalent:
\begin{enumerate}[label=\arabic*)]
\item\label{prop:for_cmpt_gp:one}
$\|\mf Q(f)\| \geq \delta \|f\|$ for $f\in \mrm{C}(G)^+$;
\item\label{prop:for_cmpt_gp:two}
Each conjugacy class satisfies $|\!\conjc(r)| \leq \delta^{-1}$;
\item\label{prop:for_cmpt_gp:three}
Each centraliser $C_G(r)$ has finite index in $G$, with $[ G : C_G(r) ] \leq \delta^{-1}$.
\end{enumerate}
\end{proposition}
\begin{proof}
By the equivalence of \ref{lem:norm_to_cstar_ineq:twoa} and \ref{lem:norm_to_cstar_ineq:twob} in Lemma~\ref{lem:norm_to_cstar_ineq}, we see that condition \ref{prop:for_cmpt_gp:one} is equivalent to
\begin{equation}\label{eq:for_cmpt_gp:1}
\int_G f(trt^{-1}) \md t \geq \delta f(r) \qquad (f\in \mrm{C}(G)^+, r\in G).
\end{equation}

Denote by $\nu=\md t$ the Haar measure on $G$, which is regular (we follow Rudin's conventions \cite[Definition 2.15]{RudinRC}). Fix $r\in G$, and define $\theta \colon G \to G$ by $\theta(t) = trt^{-1}$, a continuous map.  Let $\mu\in M(G)$ be the pushforward of the Haar measure, $\mu = \nu\circ\theta^{-1}$. 

Then $\mu$ is regular, a fact surely known, but for completeness we give a quick proof. As $\nu$ is finite and $G$ compact, we only need to check that $\mu$ is inner regular (note $\mu(E)=1-\mu(G\setminus E)$ for every Borel set $E\subseteq G$). Let $E\subseteq G$ be Borel and $\epsilon>0$, so as $\nu$ is regular, there is a compact $L\subseteq \theta^{-1}(E)$ with $\nu(\theta^{-1}(E)\setminus L) < \epsilon$.  Then $K = \theta(L)$ is compact, as $\theta$ is continuous, and $L\subseteq \theta^{-1}(E) \implies K \subseteq E$.  Also $L \subseteq \theta^{-1}(K)$, and so $\mu(E\setminus K) = \nu(\theta^{-1}(E)\setminus\theta^{-1}(K)) \leq \nu(\theta^{-1}(E)\setminus L) <\epsilon$. We conclude that $\mu$ is regular as claimed. 

It follows from Riesz's theorem \cite[Theorem 6.19]{RudinRC} that $\mu$ is the unique regular measure on $ G $ with
\[ \int_G f(s) \md\mu(s) = \int_G f(\theta(t)) \md t = \int_G f(trt^{-1}) \md t\qquad(f\in\mrm{C}(G)). \]

By regularity of $\mu$ we find for each $\epsilon>0$ an open set $U$ with $r\in U$ and $\mu(U \setminus \{r\}) < \epsilon$. As $U$ is open and contains $r$, there is $f\colon G\to[0,1]$ continuous with $f(r)=1$ and $f$ supported in $U$.  Then
\begin{align*}
\nu(C_G(r)) \leq \int_G f(trt^{-1}) \md t
&\leq \nu\big( \{ t\in G \,|\, trt^{-1}\in U\} \big)
= \nu(\theta^{-1}(U)) = \mu(U) \\
&< \mu(\{r\}) + \epsilon
= \nu(\theta^{-1}(\{r\})) + \epsilon
= \nu(C_G(r)) + \epsilon.
\end{align*}
Together with \eqref{eq:for_cmpt_gp:1}, this gives
\[ \delta = \delta f(r) \leq \int_G f(trt^{-1}) \md t < \nu(C_G(r)) + \epsilon,
\]
and since $\epsilon>0$ was arbitrary we conclude $\delta \leq \nu(C_G(r))$.

Let $x_1=e, x_2,\dots,x_n\in G$ be such that the cosets $(x_j C_G(r))_{j=1}^n$ are disjoint.  By invariance of the Haar measure,
\[ 1 = \nu(G) \geq \sum_{j=1}^n \nu\big( x_j C_G(r) \big) = n \,\nu(C_G(r)) \geq n\delta, \]
so $n\leq \delta^{-1}$.  Thus $[ G : C_G(r) ] \leq \delta^{-1}$, and by the Orbit--Stabiliser theorem we get $|\!\conjc(r)| \leq \delta^{-1}$. 
This argument shows that \ref{prop:for_cmpt_gp:one} $\Rightarrow$ \ref{prop:for_cmpt_gp:two} $\Leftrightarrow$ \ref{prop:for_cmpt_gp:three}.

Conversely, assume that \ref{prop:for_cmpt_gp:two} holds.  For $r\in G$, let $(x_i C_G(r))_{i=1}^n$ be a complete set of cosets of $C_G(r)$, equivalently, $\{ x_i r x_i^{-1} \,|\, 1\leq i\leq n\} = \conjc(r)$.  Then for $f\in \mrm{C}(G)^+$,
\begin{align*}
\int_G f(trt^{-1}) \md t
&= \sum_{i=1}^n \int_{C_G(r)} f(x_i t r t^{-1} x_i^{-1}) \md t
= \sum_{i=1}^n \nu(C_G(r)) f(x_i r x_i^{-1}) \\
&= \nu(C_G(r)) \sum_{s\in \conjc(r)} f(s)
\geq \nu(C_G(r)) f(r).
\end{align*}
As $1 = \nu(G) = n\, \nu(C_G(r))$ we see that $\nu(C_G(r)) = n^{-1} \geq \delta$ and so the inequality in \eqref{eq:for_cmpt_gp:1} holds, thus condition \ref{prop:for_cmpt_gp:one} holds.
\end{proof}

Recall that the derived subgroup of an abstract group $H$ is the subgroup generated by all commutators $[g,h] = g^{-1}h^{-1}gh$ for $g,h \in H$, denoted by $H'$ or $[H,H]$. Since $t [g,h] t^{-1} = [tgt^{-1}, tht^{-1}]$ for $g,h,t\in H$ the subgroup $H'$ is normal.  As $[g,h] \in H'$ by definition, we see that $gh H' = hg H'$ and so $H/H'$ is abelian.  Notice that $t^{-1}rt = [t,r^{-1}] r \in H' r$, and so if $H'$ is finite, then $|\!\conjc(r)| \leq |H'r| = |H'|$, for any $r$.  A remarkable theorem of Neumann shows the converse.

\begin{theorem}[{\cite[Theorem~3.1]{neumann_gps_perm_subsets}}]
Let $H$ be a group such that there is a constant $n\in\mathbb N$ with $|\!\conjc(r)| \leq n$ for each $r\in H$.  Then $H'$ is finite.
\end{theorem}

Combining this result with Proposition \ref{prop:for_cmpt_gp}, Lemma \ref{lem:norm_to_cstar_ineq} and Proposition \ref{prop2} we obtain a characterisation of the compact groups with $\A(D(\wh{G})/G)=\A(\wh{G})\otimes \I$.

\begin{corollary}
Let $G$ be a compact group.  Then $\A(D(\wh G)/G) = \A(\wh G) \otimes \I$ if and only if $G'$ is finite.
\end{corollary}

It follows that there are many compact groups $G$ for which $\A(D(\wh G)/G)$ is a strict subset of $\A(\wh G)\otimes \I$, for example $G=\SU(n)\,$ for $n\ge 2$. 
Note that when $\bbGamma = \Gamma$ is a discrete group we have $\mrm{C}(\wh\bbGamma) = \mrm{C}^*_r(\Gamma)$, hence $\mathscr{C}_{\wh\Gamma} = \LL^\infty(\wh\bbGamma)$ and $\mf Q = \id$. So, in this case one trivially gets $\A(D(\Gamma)/\wh\Gamma) = \A(\Gamma) \otimes \I$.

Next we consider duals of free orthogonal quantum groups.

\begin{proposition}\label{prop1}
Let $N\ge 2,F\in \oon{GL}(N,\CC)$ with $\ov{F}F\in \RR\I$ and $\wh{\bbGamma}=O_F^+$ be the associated free orthogonal quantum group. If $\bbGamma$ is not unimodular, then $\A(D(\bbGamma)/\wh{\bbGamma})\subsetneq \A(\bbGamma)\otimes \I$. This conclusion in particular holds for $\wh{\bbGamma}=\SU_q(2)\,(q\in \left]-1,1\right[\setminus\{0\})$.
\end{proposition}

\begin{proof}
The compact quantum group $O_F^+$ was introduced in \cite{UQG}; let us recall some of its properties, see for instance \cite{TimmermannBook, Banica_On}. One can identify $\Irr(O_F^+)$ with $\ZZ_+$ in such a way that the fundamental representation corresponds to $1$, and the fusion rules are 
\begin{equation}\label{eq18}
n\tp m\simeq |n-m|\oplus (|n-m|+2)\oplus\cdots\oplus (n+m)\quad(n,m\in\ZZ_+).
\end{equation}
Furthermore every finite dimensional representation of $O_F^+$ is equivalent to its contragradient. From \eqref{eq18} it follows that the classical and quantum dimension functions are given by $\dim(n)=[n+1]_{q_c}$ and $\dim_q(n)=[n+1]_{q_q}$ for $n\in\ZZ_+$ and some $0<q_q\le q_c\le 1$, where $[n]_{q_x}$ are the $q$-numbers given by $[n]_{q_x}=\tfrac{q_x^{-n}-q_x^n}{q_x^{-1}-q_x}$ if $0<q_x<1$ and $[n]_{1}=n$. Since we assume that $O_F^+$ is not of Kac type, we have $q_q<q_c$. 

From Proposition \ref{prop2} we know that $\A(D(\bbGamma) / \wh\bbGamma)\subseteq \A(\bbGamma)\otimes \I$, and that equality of these two vector spaces is equivalent to 
\[
\alpha^*\colon \LL^{\infty}(\wh\bbGamma)\rightarrow 
\{p_z x \,|\, x\in \LL^{\infty}(\wh{D(\bbGamma)})\}; \;\wh{x}\mapsto p_z (\wh{x}\otimes \I)
\]
being an isomorphism. We will show that there is no $\delta>0$ for which the inequalities
\begin{equation}\label{eq22}
\|\mf{Q}(\chi_n\chi_n^*)\|\ge \delta \|\chi_n\chi_n^*\|\qquad(n\in\NN)
\end{equation}
hold, which by \eqref{eq:alpha_star_Q1} shows that $\alpha^*$ is not bounded from below and proves the claim.

Fix $n\in\NN$ and note that
\begin{equation}\label{eq19}
\mf{Q}(\chi_n\chi_n^*)=\mf{Q}(\chi_{n\stp n})=\sum_{k=0}^{n} 
\mf{Q}(\chi_{2k})=\sum_{k=0}^{n}
\frac{\dim(2k)}{\dim_q(2k)}\chi_{2k}=
\sum_{k=0}^{n}
\frac{[2k+1]_{q_c}}{[2k+1]_{q_q}}\chi_{2k}.
\end{equation}
To calculate norm of this element in $\mscr{C}_{O_F^+}$ we note that the spectrum of $\chi_1$ in $\mscr{C}_{O_F^+}$ is equal to $\left[-2,2\right]$ and the restricted Haar integral is the semicircle law, see~\cite[Corollary 6.4.12]{TimmermannBook}. Functional calculus then establishes an isomorphism $\mscr{C}_{O_F^+}\simeq \LL^{\infty}(\left[-2,2\right])$; 
in the unimodular case this observation was recorded e.g.~in \cite{RadialMASA}. 

Choose any $0<q<1$. The quantum group $\SU_q(2)$ can be constructed as an orthogonal quantum group for an appropriate matrix (\cite[Proposition 6.4.8]{TimmermannBook}), so by the above we obtain a uniquely determined isomorphism $\mscr{C}_{O_F^+}\simeq \mscr{C}_{\SU_q(2)}$ which maps the character $\chi_1$ to the character $\chi_1^{\SU_q(2)}\in \LL^{\infty}(\SU_q(2))$. Using this isomorphism and \eqref{eq19} we can calculate the norm of $\mf{Q}(\chi_n\chi_n^*)$ as follows:
\begin{equation}\label{eq20}
\|\mf{Q}(\chi_n\chi_n^*)\|=\bigl\|
\sum_{k=0}^{n} \frac{[2k+1]_{q_c}}{[2k+1]_{q_q}}\chi_{2k}\bigr\|_{\mscr{C}_{O_F^+}}\!=
\bigl\|\sum_{k=0}^{n} \frac{[2k+1]_{q_c}}{[2k+1]_{q_q}}\chi_{2k}^{\SU_q(2)}\bigr\|_{\mscr{C}_{\SU_q(2)}}\!=
\sum_{k=0}^{n} \frac{[2k+1]_{q_c}}{[2k+1]_{q_q}}(2k+1),
\end{equation}
where in the last equality we use that $\mrm{C}(\SU_q(2))$ has a continuous counit. Before we bound \eqref{eq20}, let us calculate the norm of $\chi_{n}\chi_n^*$ in a similar way:
\begin{equation}\label{eq21}
\|\chi_n\chi_n^*\|_{\mscr{C}_{O_F^+}}=
\|\chi_n\|^2_{\mscr{C}_{O_F^+}}=
\| \chi_{n}^{\SU_q(2)}\|^2_{\mscr{C}_{\SU_q(2)}}=
(n+1)^2.
\end{equation}
It is an elementary exercise that $\NN\ni n\mapsto \frac{[2n+1]_{q_c}}{[2n+1]_{q_q}}\in \RR$ is decreasing for large enough $n$, say for $n\ge n_0$. Assume $n\ge n_0^2$. We have by \eqref{eq20}
\begin{equation}\label{eq23}
\|\mf{Q}(\chi_n\chi_n^*)\| \le 
\sum_{k=0}^{\lfloor \sqrt{n}\rfloor} 
(2k+1) +
\frac{[2\lfloor \sqrt{n}\rfloor +1]_{q_c}}{[2\lfloor \sqrt{n}\rfloor +1]_{q_q}}
\sum_{k=\lfloor \sqrt{n}\rfloor +1}^{n}
(2k+1).
\end{equation}
If $q_c=1$, we can continue as follows:
\[\begin{split}
\|\mf{Q}(\chi_n\chi_n^*)\|&\le 
2 (\lfloor \sqrt{n}\rfloor +1)\tfrac{\lfloor \sqrt{n}\rfloor }
{2}+\lfloor \sqrt{n}\rfloor +1\\
&+
(2\lfloor \sqrt{n}\rfloor +1)\tfrac{q_q^{-1} - q_q}{q_q^{-2\lfloor \sqrt{n}\rfloor -1}-q_q^{2\lfloor \sqrt{n}\rfloor +1}}\bigl(
2 (n-\lfloor \sqrt{n}\rfloor ) \tfrac{n+\lfloor \sqrt{n}\rfloor +1}{2}+n-\lfloor \sqrt{n}\rfloor
\bigr),
\end{split}\]
consequently $\|\mf{Q}(\chi_n\chi_n^*)\|=\mc{O}(n)$ as $n\to\infty$. If $q_c<1$ we similarly obtain from \eqref{eq23}
\[\begin{split}
\|\mf{Q}(\chi_n\chi_n^*)\|&\le 
2 (\lfloor \sqrt{n}\rfloor +1)\tfrac{\lfloor \sqrt{n}\rfloor }
{2}+\lfloor \sqrt{n}\rfloor +1\\
&+
\tfrac{q_c^{-2\lfloor \sqrt{n}\rfloor -1}-q_c^{2\lfloor \sqrt{n}\rfloor +1}}{q_c^{-1} - q_c}
\,\tfrac{q_q^{-1} - q_q}{q_q^{-2\lfloor \sqrt{n}\rfloor -1}-q_q^{2\lfloor \sqrt{n}\rfloor +1}}\bigl(
2 (n-\lfloor \sqrt{n}\rfloor ) \tfrac{n+\lfloor \sqrt{n}\rfloor +1}{2}+n-\lfloor \sqrt{n}\rfloor
\bigr),
\end{split}\]
and, since the second summand is bounded by $ (\tfrac{q_c}{q_q})^{-2\lfloor \sqrt{n}\rfloor -1} $ times a polynomial in $ n $, we get 
$\|\mf{Q}(\chi_n\chi_n^*)\|=\mc{O}(n)$ as $n\to\infty$ in this case too. Together with \eqref{eq21} this shows that \eqref{eq22} cannot hold for any $\delta>0$. 
\end{proof}

Another class of examples is provided by certain infinite direct sums. For notational convenience we consider only countably-infinite direct sums, but the proof below can be easily extended to the general case.

\begin{proposition}
Let $(\bbGamma_n)_{n\in\NN}$ be a sequence of non-classical discrete quantum groups and let $\bbGamma=\bigoplus_{n=1}^{\infty}\bbGamma_n$ be its infinite direct sum. Then $\A(D(\bbGamma) / \wh\bbGamma)\subsetneq \A(\bbGamma)\otimes \I$.
\end{proposition}

\begin{proof}
As in the proof of Proposition \ref{prop1}, we will show that $\alpha^*$ is not bounded from below using \eqref{eq18}. Fix $n\in\NN$. Since $\bbGamma_n$ is not classical, there is $\beta\in \Irr(\wh\bbGamma_n)$ with $\dim(\beta)\ge 2$. 
Assume first that the real part of the off-diagonal matrix coefficient $U^{\beta}_{1,\dim(\beta)}$, taken with respect to any orthonormal basis $e_1,\dots, e_{\dim(\beta)}$, is nonzero. Define self-adjoint $\wh{y}_n\in \Pol(\wh\bbGamma_n)$ via $\wh{y}_n=t (U^{\beta}_{1,\dim(\beta)} + U^{\beta *}_{1,\dim(\beta)})$, where $t\in\RR$ is chosen such that $\|\wh{y}_n\|=1\in \oon{Sp}(\wh{y}_n)$. Observe that $\mf{Q}(\wh{y}_n)=0$ as $1\neq \dim(\beta)$ and furthermore $2\in \oon{Sp}(\I+\wh{y}_n)\subseteq \left[0,2\right]$. It follows that $\|\I+\wh{y}_n\|=2$. If the real part of $U^{\beta}_{1,\dim(\beta)}$ is zero, then we construct an element $\wh{y}_n$ with the same properties by using the imaginary part of $U^{\beta}_{1,\dim(\beta)}$ instead. 

For the definition and properties of $\bbGamma=\bigoplus_{n=1}^{\infty} \bbGamma_n$ we refer to \cite{WangTensor} (see also \cite[Section 3]{KrajczokSoltanExamples}). Let us only recall that the irreducible representations of $\wh\bbGamma$ are given by
\[
\{e^{\boxtimes\infty}, \beta_1\boxtimes \cdots\boxtimes \beta_K\boxtimes e^{\boxtimes\infty}\,|\,K\in\NN, \beta_1\in \Irr(\wh{\bbGamma}_1),\dotsc,\beta_K\in \Irr(\wh\bbGamma_K)\},
\]
where $e^{\boxtimes\infty}$ is the trivial representation on $\CC$, while $\beta_1\boxtimes \cdots\boxtimes \beta_K\boxtimes e^{\boxtimes\infty}$ is the representation on $\msf{H}_{\beta_1}\otimes \cdots\otimes \msf{H}_{\beta_K}$ with 
\begin{equation}\label{eq17}
U^{\beta_1\boxtimes \cdots\boxtimes \beta_K\boxtimes e^{\boxtimes\infty} }_{(i_1,\dotsc,i_K),(j_1,\dotsc,j_K)}=
U^{\beta_1}_{i_1,j_1}\otimes \cdots\otimes U^{\beta_K}_{i_K,j_K}\otimes \I^{\otimes\infty}.
\end{equation}
Now, choose $N\in\NN$ and consider $(\I+\wh{y}_1)\otimes \cdots\otimes (\I+\wh{y}_N)\otimes \I^{\otimes\infty}$ in $\Pol(\wh\bbGamma)$. Since each $\I+\wh{y}_n$ is positive, we can write 
this element as $\wh{x}\wh{x}^*$ for some $\wh{x}\in \mrm{C}(\wh\bbGamma)$. Then we have
\[
\|\wh x\|^2=\|\wh{x}\wh{x}^*\|=\|
(\I+\wh{y}_1)\otimes \cdots\otimes (\I+\wh{y}_N)\otimes \I^{\otimes\infty}\|=
\|\I+\wh{y}_1\|\cdots
\|\I+\wh{y}_N\|=2^N.
\]
Next observe that 
\[
\mf{Q}(z_1\otimes \cdots \otimes z_K \otimes \I^{\otimes \infty})=\mf{Q}(z_1)\otimes \cdots\otimes \mf{Q}(z_K)\otimes \I^{\otimes \infty}\quad(K\in\NN, z_1\in \Pol(\wh{\bbGamma}_1),\dotsc,z_K\in \Pol( \wh\bbGamma_K)),
\]
which easily follows from \eqref{eq17}. Using this description and $\mf{Q}(\wh{y}_n)=0$ we have
\[\begin{split}
&\quad\;
\mf{Q}(\wh{x}\wh{x}^*)=
\mf{Q}\bigl( (\I+\wh{y}_1)\otimes\cdots\otimes (\I+\wh{y}_N)\otimes \I^{\otimes\infty}\bigr)\\
&=
\sum_{z_1\in \{\I,\wh{y}_1\}}\cdots \sum_{z_N\in \{\I,\wh{y}_N\}} \mf{Q}(z_1)\otimes \cdots\otimes \mf{Q}(z_N)\otimes \I^{\otimes \infty}=\I^{\otimes \infty}.
\end{split}\]
We thus obtain $\|\wh{x}\|=2^{N-1}$ on the one hand, and $ \|\alpha^*(\wh{x})\|=\|\mf{Q}(\wh x \wh{x}^*)\|^{1/2}=1$ by using \eqref{eq:alpha_star_Q1} on the other. 
Hence $\alpha^*$ is not bounded from below.
\end{proof}

As a final remark, we consider briefly the second averaging map $\mrm{c}_0(\bbGamma)\otimes\I = \mrm{C}_0(D(\bbGamma)/\wh\bbGamma) \to \mrm{C}_0(\wh\bbGamma \backslash D(\bbGamma)/\wh\bbGamma)$ given by $x\otimes \I \mapsto A(x)\otimes \I$. As explained in Remark~\ref{rem:A_bad_fourier_nonunimod}, there exist non-unimodular $\bbGamma$ for which there are elements $x\in \A(\bbGamma)$ with $A(x)\not\in \A(\bbGamma)$. In this situation we cannot have $\A(D(\bbGamma) / \wh\bbGamma) = \A(\bbGamma)\otimes \I$, since $A(x)\otimes\I \not\in \mc Z\A(\bbGamma)\otimes \I$ while $\A(\wh\bbGamma \backslash D(\bbGamma)/\wh\bbGamma) \subseteq \mc Z\A(\bbGamma)\otimes \I$ by Proposition~\ref{prop:iotamaps}.

\subsection{Density of central algebras} \label{subsecdensity}

Recall that the \emph{quantum characters} are defined by
\[ \chi_\alpha^q = \sum_{i,j=1}^{\dim(\alpha)} (\uprho_\alpha)_{j,i} U^\alpha_{i,j}=
\wh\sigma_{-i/2}(\chi_\alpha) \qquad (\alpha\in\Irr(\wh\bbGamma)). \]
Then for any $\beta\in\Irr(\wh\bbGamma),1\le i,j\le\dim(\beta)$ we see that
\begin{equation}\label{eq:qc_adjoint}
h(\chi_\alpha^{q *} U^\beta_{i,j}) = 
\sum_{k,l=1}^{\dim(\alpha)} (\uprho_\alpha)_{k,l} h\big( (U^\alpha_{k,l})^* U^\beta_{i,j} \big)
=\frac{ \delta_{\alpha,\beta} }{\dim_q(\alpha)} \sum_{k=1}^{\dim(\alpha)} (\uprho_\alpha)_{k,j} (\uprho_\alpha^{-1})_{i,k} = \delta_{\alpha,\beta} \frac{\delta_{i,j}}{\dim_q(\alpha)}.
\end{equation}
Thus $\wh\lambda(h(\chi_\alpha^{q *} \cdot)) = \dim_q(\alpha)^{-1} p_\alpha$, and so
\[ \mc Z\mrm{c}_{00}(\bbGamma) = \wh\lambda\big( \lin\{ h(\chi_\alpha^{q *} \cdot)  \,|\, \alpha\in\Irr(\wh\bbGamma) \} \big). \]
We say that $\lin \{ h(\chi_\alpha^{q *} \cdot)  \,|\, \alpha\in\Irr(\wh\bbGamma) \} \subseteq \mc Z\LL^1(\wh\bbGamma)$ is the space of \emph{finitely-supported} elements in $\mc Z\LL^1(\wh\bbGamma)$.

\begin{lemma}\label{lem:preadjoint_Q}
As $\mf Q$ is normal, it has a pre-adjoint map $\mf Q_* \colon \LL^1(\wh\bbGamma) \to \LL^1(\wh\bbGamma)$, which satisfies
\begin{equation}\label{eq:preadjoint_Q}
\mf Q_*\big( h(U^{\alpha *}_{i,j}\cdot) \big)
= \frac{(\uprho_\alpha^{-1})_{j,i}}{\dim_q(\alpha)}
h(\chi_\alpha^{q *}\cdot)
\qquad (\alpha\in\Irr(\wh\bbGamma),1\leq i,j\leq\dim(\alpha))
\end{equation}
so that the space of finitely-supported element of $\mc Z\LL^1(\wh\bbGamma)$ is contained in the image of $\mf Q_*$ and also
\begin{equation}\label{eq:preadjoint_Q:2}
\mf Q_*\big( h(\chi_\alpha^{q *}\cdot) \big)
= \frac{\dim(\alpha)}{\dim_q(\alpha)}
h(\chi_\alpha^{q *}\cdot) \qquad (\alpha\in\Irr(\wh\bbGamma)).
\end{equation}
Let $\omega\in \mc Z\LL^1(\wh\bbGamma)$ be finitely-supported, and let $x\in\ker\mf Q$.  Then $\omega(x)=0$.
\end{lemma}

\begin{proof}
For $\beta\in\Irr(\wh\bbGamma),1\le s,t\le \dim(\beta)$ we compute that
\[ \mf Q_*\big( h(U^{\alpha *}_{i,j}\cdot) \big)(U^\beta_{s,t})
= h\big( U^{\alpha *}_{i,j}  \mf Q(U^\beta_{s,t}) \big)
= \frac{\delta_{s,t}}{\dim_q(\beta)} \sum_{r=1}^{\dim(\beta)}h\big( U^{\alpha *}_{i,j} U^\beta_{r,r} \big)
= \delta_{\alpha,\beta} \frac{\delta_{s,t}}{\dim_q(\alpha)^2} (\uprho_\alpha^{-1})_{j,i}, \]
while
\begin{align*}
\frac{(\uprho_\alpha^{-1})_{j,i}}{\dim_q(\alpha)}
h(\chi_\alpha^{q *}U^\beta_{s,t})
= \delta_{\alpha,\beta} \frac{\delta_{s,t}}{\dim_q(\alpha)^2} (\uprho_\alpha^{-1})_{j,i}.
\end{align*}
by \eqref{eq:qc_adjoint}. As elements $U^\beta_{s,t}$ form a basis of $\Pol(\wh\bbGamma)$, and $\Pol(\wh\bbGamma)$ is weak$^*$-dense in $\LL^\infty(\wh\bbGamma)$, this establishes \eqref{eq:preadjoint_Q}. It follows that
\[ \mf Q_*\big( h(\chi_\alpha^{q *}\cdot) \big)
= \sum_{i,j=1}^{\dim(\alpha)} (\uprho_\alpha)_{i,j} \mf Q_*\big( h(U^{\alpha *}_{i,j}\cdot) \big)
= \sum_{i,j=1}^{\dim(\alpha)} (\uprho_\alpha)_{i,j} \frac{(\uprho_\alpha^{-1})_{j,i}}{\dim_q(\alpha)}
h(\chi_\alpha^{q *}\cdot)
= \frac{\dim(\alpha)}{\dim_q(\alpha)} h(\chi_\alpha^{q *}\cdot), \]
which is \eqref{eq:preadjoint_Q:2}.

To show that any finitely-supported element $\omega\in\mc Z\LL^1(\wh\bbGamma)$ annihilates $\ker\mf Q$, it suffices to show that $h(\chi_\alpha^{q *} x) = 0$ for each $\alpha\in\Irr(\wh\bbGamma)$ and $x\in\ker\mf Q$.  However, with $\omega = h(\chi_\alpha^{q *}\cdot)$, we have
\[ h(\chi_\alpha^{q *} x) = \omega(x) =
\frac{\dim_q(\alpha)}{\dim(\alpha)} \mf Q_*(\omega)(x)
= \frac{\dim_q(\alpha)}{\dim(\alpha)} \omega(\mf Q(x)) = 0, \]
as required.
\end{proof}

The following is interesting in view of Lemma~\ref{lemma:dense_in_centre_fourier} above.

\begin{proposition}\label{prop:kerQ_to_centrec00}
The following are equivalent:
\begin{enumerate}
\item[1)]\label{prop:kerQ_to_centrec00:one}
$\mc Z\mrm{c}_{00}(\bbGamma)$ is dense in $\mc Z\A(\bbGamma)$ for the $\A(\bbGamma)$ norm;
\item[2)]\label{prop:kerQ_to_centrec00:two}
$(\mc Z\LL^1(\wh\bbGamma))^\perp = \ker\mf Q$.
\end{enumerate}
Furthermore, the following are equivalent:
\begin{enumerate}[resume]
\item[3)]\label{prop:kerQ_to_centrec00:onea}
$\mc Z\mrm{c}_{00}(\bbGamma)$ is weak$^*$-dense in $\mc Z\B_r(\bbGamma)$;
\item[4)]\label{prop:kerQ_to_centrec00:twoa}
${}^\perp (\mc Z(\mrm{C}(\wh\bbGamma)^*)) = \ker\mf Q_0$.
\end{enumerate}
\end{proposition}
\begin{proof}
$ 1) \Rightarrow 2) $ By assumption, the finitely-supported elements of $\mc Z\LL^1(\wh\bbGamma)$ are dense in $\mc Z\LL^1(\wh\bbGamma)$. 
Take $\omega \in \mc Z\LL^1(\wh\bbGamma)$, and let $(\omega_n)_{n\in\NN}$ be a sequence of finitely-supported central elements converging in norm to $\omega$. 
For $x\in\ker\mf Q$, Lemma~\ref{lem:preadjoint_Q} then implies $0 = \lim_{n\to\infty} \omega_n(x) = \omega(x)$.  
We conclude $\ker\mf Q \subseteq (\mc Z\LL^1(\wh\bbGamma))^\perp$.
Corollary~\ref{corr:links_centreL1_Q} gives the reverse inclusion $(\mc Z\LL^1(\wh\bbGamma))^\perp \subseteq \ker\mf Q$. 

$ 2) \Rightarrow 1) $ 
Towards a contradiction, suppose that $1)$ does not hold. Then, by Hahn-Banach, there is $x\in \LL^\infty(\wh\bbGamma)$ which annihilates the finitely-supported elements in $\mc Z\LL^1(\wh\bbGamma)$, such that there is $\omega\in \mc Z\LL^1(\wh\bbGamma)$ with $\omega(x)=1$. Thus $h(\chi_\alpha^{q *}x)=0$ for each $\alpha$.  From \eqref{eq:preadjoint_Q}, it follows that
\[
h(U^{\alpha *}_{i,j}\cdot)(\mf{Q}(x))= \mf Q_*\big( h(U^{\alpha *}_{i,j}\cdot) \big)(x)
= \frac{(\uprho_\alpha^{-1})_{j,i}}{\dim_q(\alpha)}
h(\chi_\alpha^{q *}x) = 0, \]
for each $\alpha,i,j$.  As $ \{ h(a^*\cdot) \,|\, a\in\Pol(\wh\bbGamma) \}$ is dense in $\LL^1(\wh\bbGamma)$ it follows that $\mf Q(x)=0$ and hence $x\in\ker\mf Q$.
By assumption $x\in (\mc{Z}\LL^1(\wh\bbGamma))^{\perp}$, which is a contradiction.

$3) \Rightarrow 4)$ By assumption, the finitely-supported elements in $\mc Z\LL^1(\wh\bbGamma)$ are weak$^*$-dense in $\mc Z(\mrm{C}(\wh\bbGamma)^*)$. Then given $\mu\in \mc Z(\mrm{C}(\wh\bbGamma)^*)$ we pick a net $(\omega_i)_{i\in I}$ of finitely-supported elements in $\mc Z\LL^1(\wh\bbGamma)$ converging weak$^*$ to $\mu$. For $x\in\ker\mf Q_0 \subseteq \ker\mf Q$ we again have $0 = \lim_{i\in I} \omega_i(x) = \mu(x)$. Thus ${}^\perp(\mc{Z}(\mrm{C}(\wh\bbGamma)^*)) \supseteq \ker\mf Q_0$, and together with Corollary~\ref{corr:links_centreL1_Q} 
this shows $4)$.

$4) \Rightarrow 3)$ Assume that $4)$ holds but $3)$ does not. This means that the space of finitely-supported elements of $\mc{Z}\LL^1(\wh\bbGamma)$ is not weak$^*$ dense in $\mc{Z}(\mrm{C}(\wh\bbGamma)^*)$. The weak$^*$ topology of $\mc{Z}(\mrm{C}(\wh\bbGamma)^*)$ is given by the canonical predual $\mrm{C}(\wh\bbGamma) / {}^{\perp}(\mc{Z}(\mrm{C}(\wh\bbGamma)^*))$, hence there is $\mu\in \mc Z(\mrm{C}(\wh\bbGamma)^*)$ and $x\in \mrm{C}(\wh\bbGamma)$ which annihilates the finitely-supported elements of $\mc Z\LL^1(\wh\bbGamma)$ and satisfies $\mu(x)=1$.  By $4)$, it follows that $x\not\in\ker\mf Q_0$ so $\mf Q(x)\not=0$, but the same argument as before shows that $\la x, \mf Q_*(\omega) \ra = 0$ for each $\omega\in\LL^1(\wh\bbGamma)$. This implies $\mf Q(x)=0$, which is a contradiction.
\end{proof}

\begin{corollary}\label{cor1}
If $\A(\wh\GGamma\backslash D(\GGamma)/\wh\GGamma) = \mc Z\A(\bbGamma)\otimes\I$ then $\mc Z\mrm{c}_{00}(\bbGamma)$ is dense in $\mc Z\A(\bbGamma)$.  Furthermore, if $\B_r(\wh\GGamma\backslash D(\GGamma)/\wh\GGamma) = \mc Z\B_r(\bbGamma)\otimes\I$ then $\mc Z\mrm{c}_{00}(\bbGamma)$ is weak$^*$-dense in $\mc Z\B_r(\bbGamma)$.
\end{corollary}

\begin{proof}
Assume that $\A(\wh\GGamma\backslash D(\GGamma)/\wh\GGamma) = \mc Z\A(\bbGamma)\otimes\I$. By Theorem~\ref{thm:main_fourierinv}, $\iota$ and hence also $\iota^*$ are isomorphisms. Using Corollary~\ref{corr:links_centreL1_Q} we see that $(\mc{Z}\LL^1(\wh\bbGamma))^{\perp}=\ker \mf{Q} $, and therefore Proposition~\ref{prop:kerQ_to_centrec00} gives the result. The case of 
the reduced Fourier-Stieltjes algebra is analogous.
\end{proof}

If $\GGamma$ is unimodular, the Haar state $ h $ on $\LL^\infty(\wh\bbGamma)$ is a trace we have the normal conditional expectation $F\colon \LL^\infty(\wh\bbGamma) \rightarrow \mathscr{C}_{\wh\bbGamma}$, considered already in Section \ref{sec:unimodularsub}. Comparing the definition of $\mf Q$ in \eqref{eq:defn_mfQ} with formula \eqref{eq:cond_exp_classfns} shows that $\mf Q = F$ in this case.

\begin{theorem}\label{thm:unimod_inv_fourier}
Let $\bbGamma$ be unimodular.  Then $\mf Q$ has weak$^*$-closed image, and $\iota^*$ is an isomorphism.  As such, $\A(\wh\GGamma\backslash D(\GGamma)/\wh\GGamma) = \mc Z\A(\bbGamma)\otimes\I$. 
\end{theorem}

\begin{proof}
The map $\mf Q = F$ is a weak$^*$-continuous idempotent and so has weak$^*$-closed image. According to Corollary~\ref{corr:links_centreL1_Q} this means that $\iota^*$ is surjective. 
Moreover, Proposition~\ref{unimodularaveraging} shows that $\mf Q_* = F_*$ is a projection onto $\mc Z\LL^1(\wh\bbGamma)$, and hence $\ker\mf Q = \ker F = (\mc Z\LL^1(\wh\bbGamma))^\perp$; 
for this we could also use Lemma~\ref{lemma:dense_in_centre_fourier} together with Proposition~\ref{prop:kerQ_to_centrec00}. 
Due to Corollary~\ref{corr:links_centreL1_Q} we conclude that $\iota^*$ is injective. The claim now follows from Theorem~\ref{thm:main_fourierinv}.
\end{proof}

Our next goal is to prove that density of $\mc{Z}\mrm{c}_{00}(\bbGamma)$ in $\mc{Z}\A(\bbGamma)$ holds under the assumption that $\bbGamma$ has central AP, 
and a similar claim for the Fourier-Stieltjes algebra. As a preparation we need to establish a couple of results concerning the weak$^*$ topology of $\M^l_{cb}(\A(\bbGamma))$ and central AP.

Let us first recall some facts from \cite[Section~4]{DKV_ApproxLCQG}. If $\GG$ is a locally compact quantum group and $\msf{H}$ a Hilbert space, then 
for $x\in \mrm{C}_0(\wh\GG)\otimes\K(\msf H)$ and $\omega\in \LL^1(\wh\GG)\wh\otimes \B(\msf H)_*$ the bounded functional
\[ \Omega_{x,\omega}\colon \M^l_{cb}(\A(\GG)) \rightarrow \mathbb C;\,
a \mapsto \langle (\Theta^l(a)\otimes\id)x, \omega \rangle \]
lies in the predual $Q^l(\A(\GG))$. Moreover, every element of $Q^l(\A(\GG))$ arises in this way, see \cite[Proposition~3.8]{DKV_ApproxLCQG}. We shall extend this result 
as follows, improving \cite[Proposition~1.3]{HaagerupKraus} in the classical setting.

\begin{lemma}\label{lemma3}
Let $\GG$ be a locally compact quantum group and let $\msf{H}$ be a Hilbert space. 
For any $x\in \mrm{C}_0(\wh\GG)\otimes\K(\msf H)$ and $\mu\in (\mrm{C}_0(\wh\GG)\otimes\K(\msf H))^*$, the bounded linear functional
\[ \Omega_{x,\mu}\colon \M^l_{cb}(\A(\GG)) \rightarrow \mathbb C;\,
a \mapsto \langle \mu, (\Theta^l(a)\otimes\id)x \rangle \]
is contained in $Q^l(\A(\GG))$.
\end{lemma}

\begin{proof}
By continuity and linearity, it suffices to prove the result for $x=y\otimes\theta$ with $y\in \mrm{C}_0(\whG), \theta\in\K(\msf H)$. Let $\mu' \in \mrm{C}_0(\whG)^*$ 
be given by $\langle \mu', b \rangle = \langle \mu, b\otimes\theta \rangle$ for $b\in\mrm{C}_0(\whG)$, so that
\[ \Omega_{x,\mu}(a) = \langle \mu, \Theta^l(a)(y) \otimes \theta \rangle
= \langle \mu', \Theta^l(a)(y) \rangle. \]
By a further continuity argument, we may suppose that $y = (\omega\otimes\id)(\ww^*)$ for some $\omega\in \LL^1(\GG)$.  As $(a\otimes\I)\ww^* = (\id\otimes\Theta^l(a))(\ww^*)$, it follows that $\Theta^l(a)(y) = (\omega\otimes\id)((a\otimes\I)\ww^*)
= (\omega a\otimes\id)(\ww^*)$ for all $a\in\M^l_{cb}(\A(\GG))$. Consider $z = (\id\otimes\mu')(\ww^*)\in \LL^\infty(\GG)$, that is, $\langle z, \omega' \rangle = \langle \mu', (\omega'\otimes\id)(\ww^*) \rangle$ for any $\omega'\in\LL^1(\GG)$. Then we have
\begin{align*}
\Omega_{x,\mu}(a)
&= \langle \mu', (\omega a\otimes\id)(\ww^*) \rangle
= \langle z, \omega a \rangle
= \langle a, z\omega \rangle.
\end{align*}
Thus $\Omega_{x,\mu}$ agrees with the action of $z\omega\in \LL^1(\GG)$, considered as an element of $Q^l(\A(\GG))$.  It follows that $\Omega_{x,\mu}\in Q^l(\A(\GG))$, as required.
\end{proof}

The next result should be compared with \cite[Theorem~4.4]{DKV_ApproxLCQG} and \cite[Theorem~1.9(c)]{HaagerupKraus}.

\begin{lemma}\label{lemma4}
Let $\GG$ be a locally compact quantum group and let $X \subseteq \M^l_{cb}(\A(\GG))$ be a convex subset. If there is a net $(a_i)_{i\in I}$ in $X$ with $a_i \xrightarrow[i\in I]{} \I$ weak$^*$, then there is a net $(b_j)_{j\in J}$ in $X$ with $(\Theta^l(b_j)\otimes\id)(x) \xrightarrow[j\in J]{} x$ in norm, for each $x\in\mrm{C}_0(\whG)\otimes\K(\msf H)$ and any Hilbert space $\msf H$.
\end{lemma}

\begin{proof}
By assumption and Lemma~\ref{lemma3}, for any $x\in \mrm{C}_0(\wh\GG)\otimes\K(\msf H)$ and $\mu\in (\mrm{C}_0(\wh\GG)\otimes\K(\msf H))^*$ we have that
\[ \lim_{i\in I} \langle \mu, (\Theta^l(a_i)\otimes\id)(x) \rangle
= \lim_{i\in I} \langle a_i, \Omega_{x,\mu} \rangle
= \langle \I, \Omega_{x,\mu} \rangle
= \langle \mu, x \rangle. \]
Hence $(\Theta^l(a_i)\otimes\id)(x) \xrightarrow[i\in I]{} x$ weakly for each $x$. Consequently, for any $n\in\NN$ and $x_1,\dots,x_n \in \mrm{C}_0(\whG)\otimes\K(\msf H)$, the weak closure  of the convex set $\bigl\{ \bigl((\Theta^l(a)\otimes\id)(x_k)\bigr)_{k=1}^{n} \,\big.|\, a\in X \bigr\}$, in the $n$-fold product of $\mrm{C}_0(\wh\GG)\otimes\K(\msf H)$ with itself, contains $(x_1,\dots,x_n)$. By Hahn-Banach, the same is true for the norm closure. In a standard way, we may now construct a (possibly) new net $(b_j)_{j\in J}$ in $X$ so that $(\Theta^l(b_j)\otimes\id)(x) \xrightarrow[j\in J]{} x$ in norm for each $x \in \mrm{C}_0(\wh\GG)\otimes\K(\msf H)$.
\end{proof}

\begin{proposition}\label{thm:approx_fs_fourier_central}
Let $\bbGamma$ be a discrete quantum group with the central approximation property.  Given $b\in \mc{Z} \B_r(\bbGamma)$, there is a net $(b_i)_{i\in I}$ in $\mc{Z} \mrm{c}_{00}(\bbGamma)\subseteq \mc{Z} \B_r(\GGamma)$ with $b_i\xrightarrow[i\in I]{} b$ weak$^*$.

Suppose further that $\bbGamma$ is centrally weakly amenable.  Then we can choose the net $(b_i)_{i\in I}$ to be bounded; more precisely, we can choose the net 
so that $\|b_i\| \leq \mc{Z}\Lambda_{cb}(\bbGamma) \|b\|$ for all $i\in I$.
\end{proposition}

\begin{proof}
Consider the convex set $\mc{Z} \mrm{c}_{00}(\bbGamma) \subseteq \M_{cb}^l(\A(\bbGamma))$. By assumption, this set contains a net converging weak$^*$ to $\I$, and so by Lemma~\ref{lemma4} there is a net $(a_i)_{i\in I}$ in $\mc{Z}\mrm{c}_{00}(\bbGamma)$ with $\Theta^l(a_i)(x) \xrightarrow[i\in I]{} x$ in norm for each $x\in \mrm{C}(\wh\bbGamma)$. For every $i\in I$, there is $\wh\omega_i\in\LL^1(\wh\bbGamma)$ with $a_i = \wh\lambda(\wh\omega_i)$, and so $\Theta^l(a_i)(x) = (\wh\omega_i\otimes\id)\Delta_{\wh\GGamma}(x)$ for $x\in \mrm{C}(\wh\bbGamma)$. It follows that $\wh\omega_i \star \mu \xrightarrow[i\in I]{} \mu$ weak$^*$ for each $\mu\in\mrm{C}(\wh\bbGamma)^*$. As $\LL^1(\wh\bbGamma)$ is an ideal in $\mrm{C}(\wh\bbGamma)^*$, given $b = \wh\lambda(\mu)\in \mc{Z}\B_r(\bbGamma)$, we obtain an approximating net in $\mc{Z} \mrm{c}_{00}(\bbGamma)\subseteq \mc{Z} \B_r(\GGamma)$ by setting $b_i = \wh\lambda(\wh\omega_ i\star \mu) = a_i b$, which converges weak$^*$ to $b$.

We now consider the case when $\bbGamma$ is centrally weakly amenable.  From the previous paragraph, it follows that
\[ \wh\omega_i \star \mu = \mu \circ \Theta^l(a_i)
\quad \implies \quad
\|b_i\|_{\A(\GG)} = \|a_ib\|_{\A(\GG)} =
\|\wh\omega_i \star \mu\| \leq \|\mu\| \|a_i\|_{cb}. \]
Hence we can work with the convex set $\{ a\in\mc Z\mrm{c}_{00}(\GGamma) \,|\, \|a\|_{cb} \leq  \mc{Z}\Lambda_{cb}(\bbGamma) \}$, and hence suppose that $\|a_i\|_{cb} \leq \mc Z\Lambda_{cb}(\bbGamma)$ for each $i$. Setting $b_i = a_i b$ as before then yields a net $(b_i)_{i\in I}$ with the desired properties.
\end{proof}

\begin{theorem}\label{prop:centralAP_density}
Let $\bbGamma$ be a discrete quantum group with the central approximation property. Then $(\mc Z\LL^1(\wh\bbGamma))^\perp = \ker\mf Q$ and $\mc Z\mrm{c}_{00}(\bbGamma)$ is dense in $\mc Z\A(\bbGamma)$. Furthermore, ${}^\perp (\mc Z(\mrm{C}(\wh\bbGamma)^*)) = \ker\mf Q_0$ and $\mc Z\mrm{c}_{00}(\bbGamma)$ is weak$^*$-dense in $\mc Z\B_r(\bbGamma)$.
\end{theorem}

\begin{proof}
We claim that there is a net $(a_i)_{i\in I}$ in $\mc Z\mrm{c}_{00}(\bbGamma)$ with $\Theta^l(a_i)(x) \xrightarrow[i\in I]{} x$ weak$^*$ for each $x\in\LL^\infty(\wh\bbGamma)$.  This can be shown by adapting the proof of \cite[Theorem~4.4]{DKV_ApproxLCQG}, which shows that when $(b_i)_{i\in I}$ is a net converging weak$^*$ to $\I$ in $\M_{cb}^l(\A(\bbGamma))$, and $f\in\ell^1(\bbGamma)$, then the net given by $a_i = b_i\star f$ will satisfy the property we need.  As $\bbGamma$ is discrete, we may take $f=\eps$ the counit, and then if $b_i$ is central also 
$a_i$ will be central.  For each $i$ let $\omega_i\in\mc Z\LL^1(\wh\bbGamma)$ with $\wh\lambda(\omega_i) = a_i$, so that in particular $\omega_i$ has finite support.

Let $x\in\ker\mf Q$, and $\omega\in\mc Z\LL^1(\wh\bbGamma)$.  Then $\omega_i \star \omega$ also has finite support, as $\wh\lambda(\omega_i \star \omega) = a_i \wh\lambda(\omega) \in \mc Z\mrm{c}_{00}(\bbGamma)$.  For each $i$, by \eqref{eq:preadjoint_Q:2} it follows that $\omega_i \star \omega = \mf Q_*(\nu_i)$ for some $\nu_i \in \mc Z\LL^1(\wh\bbGamma)$ with finite support.  Thus
\[ \omega(x) = \lim_{i\in I} \la \Theta^l(a_i)(x), \omega \ra
= \lim_{i\in I} \la x \star \omega_i, \omega \ra
= \lim_{i\in I} \la x, \mf Q_*(\nu_i) \ra = \lim_{i\in I} \nu_i(\mf Q(x)) = 0. \]
As $\omega\in\mc Z\LL^1(\wh\bbGamma)$ was arbitrary it follows that $x\in (\mc Z\LL^1(\wh\bbGamma))^\perp$. We thus have $\ker\mf Q \subseteq (\mc Z\LL^1(\wh\bbGamma))^\perp$, and 
hence $\ker\mf Q = (\mc Z\LL^1(\wh\bbGamma))^\perp$ by Corollary~\ref{corr:links_centreL1_Q}. Therefore Proposition~\ref{prop:kerQ_to_centrec00} yields the first claim.

Now let $\mu\in \mc Z(\mrm{C}(\wh\bbGamma)^*)$. Taking $X=\mc{Z}\mrm{c}_{00}(\bbGamma)$ in Lemma \ref{lemma4} we can find a net $(a_i)_{i\in I}$ in $\mc{Z}\mrm{c}_{00}(\bbGamma)$ such that $\Theta^l(a_i)(x)\xrightarrow[i\in I]{} x$ in norm for all $x\in \mrm{C}(\wh\bbGamma)$. We can write $a_i=\wh\lambda(\omega_i)$ with $\omega_i\in \mc{Z}\LL^1(\wh\bbGamma)$ finitely supported. Then $\omega_i\star\mu$ is also finitely supported in $\mc{Z}\LL^1(\wh\bbGamma)$. Consequently, for $x\in \ker\mf Q_0 \subseteq \ker \mf Q$ we get $\mu(x)=0$ by an analogous argument as above. We thus have $\ker\mf Q_0 \subseteq {}^\perp (\mc Z(\mrm{C}(\wh\bbGamma)^*))$, and the second claim now follows again from Corollary~\ref{corr:links_centreL1_Q} and Proposition~\ref{prop:kerQ_to_centrec00}.
\end{proof}

\begin{remark}
It is shown in \cite[Equation~(3.2)]{AlaghmandanCrann} that when $\bbGamma$ has the central ACPAP, see \cite[Definition 3]{DFY_CCAP}, then $\mc{Z} \LL^1(\wh\bbGamma)$ is the closed linear span of functionals of the form $h(\cdot \chi^q_{\alpha})$.  Recall that we can express the quantum character as $\chi_\alpha^q=\wh{\sigma}_{-i/2}(\chi_\alpha)$, so in particular, 
$\chi_{\ov\alpha}^{ q *} = \wh{\sigma}_{-i/2}(\chi_{\ov\alpha})^*
= \wh{\sigma}_{i/2}(\chi_{\ov\alpha}^*) = \wh{\sigma}_{i/2}(\chi_{\alpha})$.  Thus for $x\in\Pol(\wh\bbGamma)$ we see that
\[ h(x \chi^q_{\alpha})
= h(x \wh{\sigma}_{-i/2}(\chi_{\alpha}))
= h(x \wh{\sigma}_{-i}(\wh{\sigma}_{i/2}(\chi_{\alpha})) )
= h(x \wh{\sigma}_{-i}(\chi_{\ov\alpha}^{ q *}) )
= h(\chi_{\ov\alpha}^{q *} x). \]
Hence $h(\cdot \chi_{\alpha}^{q})=h(\chi_{\ov\alpha}^{q *}\cdot)$, and so the linear span of the functionals $h(\cdot \chi_{\alpha}^{q})$ agrees with the finitely-supported functionals in $\mc{Z} \LL^1(\wh\bbGamma)$, compare the discussion before Lemma~\ref{lem:preadjoint_Q}. Thus Theorem~\ref{prop:centralAP_density} improves this result from \cite{AlaghmandanCrann} by showing that it holds whenever $\bbGamma$ merely has the central AP.
\end{remark}

With reference to Theorem~\ref{thm:main_fourierinv}, it would be interesting to know if there is any relation between $\iota$ being an isomorphism, and $\iota^r$ being an isomorphism.  In general, we do not know of any such relation, though in special cases we can say something, as follows.

\begin{proposition}\label{prop:ws_dense_im}
Suppose there exists $K>0$ such that, for each $\mu\in \mc Z(\mrm{C}(\wh\bbGamma)^*)$, there is a net $(\omega_i)_{i\in I}$ in $\mc Z\LL^1(\wh\bbGamma)$ converging weak$^*$ to $\mu$ and with $\|\omega_i\| \leq K\|\mu\|$ for each $i\in I$.  If $\iota$ is an isomorphism, then also $\iota^r$ is an isomorphism.
\end{proposition}

Note that by Proposition \ref{thm:approx_fs_fourier_central} the assumption in Proposition \ref{prop:ws_dense_im} holds whenever $\bbGamma$ is centrally weakly amenable.

\begin{proof}
Recall that we can identify the dual space of $\mrm{C}(\wh\bbGamma)/{}^{\perp}(\mc{Z}(\mrm{C}(\wh\bbGamma)^*))$ with $\mc{Z}(\mrm{C}(\wh\bbGamma)^*)$. Consequently, the hypothesis implies that
\[ \sup\big\{ |\omega(\wh a)| \,|\, \omega\in \mc Z\LL^1(\wh\bbGamma), \|\omega\|\leq 1\big\}
\geq K^{-1} \big\| \wh a + {}^\perp(\mc{Z}(\mrm{C}(\wh\bbGamma)^*)) \big\|_{ \mrm{C}(\wh\bbGamma) / {}^\perp(\mc{Z}(\mrm{C}(\wh\bbGamma)^*)) } \]
for each $\wh a \in \mrm{C}(\wh\bbGamma)$.  In turn, this shows that the natural map
\[ \mrm{C}(\wh\bbGamma) / {}^\perp(\mc{Z}(\mrm{C}(\wh\bbGamma)^*)) \to \LL^\infty(\wh\bbGamma) / (\mc Z\LL^1(\wh\bbGamma))^\perp=(\mc Z\LL^1(\wh\bbGamma))^* \]
is bounded below.  We have the commutative diagram
\[ \begin{CD}
  \mrm{C}(\wh\bbGamma) / {}^\perp(\mc{Z}(\mrm{C}(\wh\bbGamma)^*))
  @>{\iota^r_*}>>
  \Xi_0(\mrm{C}_0(\wh{D(\bbGamma)})) \\
  @VVV @VVV \\
  \LL^\infty(\wh\bbGamma) / (\mc Z\LL^1(\wh\bbGamma))^\perp
  @>{\iota^*}>>
  \Xi_\infty(\LL^\infty(\wh{D(\bbGamma)}))
\end{CD} \]
As just remarked, the left-hand vertical map is bounded below, while by construction (compare Proposition~\ref{prop:xidd}) the right-hand vertical map is isometric.  It follows that if $\iota^*$ is bounded below, then also $\iota^r_*$ is bounded below.  As $\iota_*^r$ has norm-dense image by Proposition~\ref{prop:iotaontorange}, we conclude that $\iota^r$ is an isomorphism as claimed.
\end{proof}

\subsection{Counter-examples}

We now show that there are non-unimodular discrete quantum groups for which $\A(\wh\bbGamma \backslash D(\bbGamma)/\wh\bbGamma)\subsetneq \mc{Z}\A(\bbGamma)\otimes \I$. 
This is the case in particular for $\bbGamma=\wh{\SU_q(2)}$, see Theorem \ref{thm1} below. 

Let us begin by estimating the norms of characters in certain quotient spaces. For more on the link between strong amenability and spectral properties of characters, in particular the Kesten amenability criterion, we refer to \cite[Section~3.4]{Brannan}.

\begin{lemma}\label{lemma2}
Let $\bbGamma$ be a discrete quantum group and let $\alpha\in\Irr(\wh\bbGamma)$.
\begin{enumerate}
\item[1)] The inequalities
\[
\|\chi_\alpha + (\mc{Z}\LL^1(\wh\bbGamma))^{\perp}\|_{\LL^{\infty}(\wh\bbGamma)/ (\mc{Z}\LL^1(\wh\bbGamma))^{\perp}}\le \|\chi_\alpha + {}^{\perp}(\mc{Z}(\mrm{C}(\wh\bbGamma)^*))\|_{\mrm{C}(\wh\bbGamma)/ {}^{\perp}(\mc{Z}(\mrm{C}(\wh\bbGamma)^*))}\le \|\chi_\alpha\|\le \dim(\alpha)
\]
hold.
\item[2)] If $\bbGamma$ is strongly amenable, then $\|\chi_\alpha + {}^{\perp}(\mc{Z}(\mrm{C}(\wh\bbGamma)^*))\|_{\mrm{C}(\wh\bbGamma)/ {}^{\perp}(\mc{Z}(\mrm{C}(\wh\bbGamma)^*))}=\dim(\alpha)$. 
\item[3)] If $\bbGamma$ is strongly amenable and centrally  weakly amenable, then
\[
\tfrac{\dim(\alpha)}{\mc{Z} \Lambda_{cb}(\bbGamma)}\le \|\chi_\alpha + (\mc{Z}\LL^1(\wh\bbGamma))^{\perp}\|_{\LL^{\infty}(\wh\bbGamma)/ (\mc{Z}\LL^1(\wh\bbGamma))^{\perp}}.
\]
\end{enumerate}
\end{lemma}

\begin{proof}
$ 1) $ Since $\mc{Z}\LL^1(\wh\bbGamma)\subseteq \mc{Z}( \mrm{C}(\wh\bbGamma)^*)$, we have ${}^{\perp}(\mc{Z}(\mrm{C}(\wh\bbGamma)^*))\subseteq (\mc{Z}\LL^1(\wh\bbGamma))^{\perp}$ and the first point easily follows. 

$ 2) $ Assume that $\bbGamma$ is strongly amenable. Then the counit defines an element $\eps\in \mrm{C}(\wh\bbGamma)^*$ of norm $1$. Observe that $\eps\in \mc{Z} (\mrm{C}(\wh\bbGamma)^*)$, thus for $x\in {}^{\perp} (\mc{Z}(\mrm{C}(\wh\bbGamma)^*))\subseteq \mrm{C}(\wh\bbGamma)$ we have $\eps(x)=0$. Consequently we can define a functional $\tilde{\eps}$ of norm $1$ by 
\[
\tilde{\eps}\colon \mrm{C}(\wh\bbGamma)/ {}^{\perp}(\mc{Z}(\mrm{C}(\wh\bbGamma)^*))\ni y + 
 {}^{\perp}(\mc{Z}(\mrm{C}(\wh\bbGamma)^*))
\mapsto \eps(y)\in \CC.
\]
Hence
\[
\|\chi_\alpha + {}^{\perp}(\mc{Z}(\mrm{C}(\wh\bbGamma)^*))\|_{\mrm{C}(\wh\bbGamma)/ {}^{\perp}(\mc{Z}(\mrm{C}(\wh\bbGamma)^*))}\ge 
\bigl|\tilde{\eps}\bigl(
\chi_\alpha + {}^{\perp}(\mc{Z}(\mrm{C}(\wh\bbGamma)^*)) \bigr)\bigr|=\dim(\alpha),
\]
and in conjunction with part $1)$ this completes the proof of $2)$. 

$3) $ Assume in addition that $\bbGamma$ is centrally weakly amenable, so that there is a net $(\omega_\lambda)_{\lambda\in\Lambda}$ in $\mc Z\LL^1(\wh\bbGamma)$ such that with $a_\lambda = \wh\lambda(\omega_\lambda)$ we have $\|\Theta^l(a_\lambda)\|_{cb} \leq \mc{Z} \Lambda_{cb}(\bbGamma)$ for each $\lambda$, and $a_\lambda \xrightarrow[\lambda\in \Lambda]{} \I$ pointwise. Define $\mu\in\ell^1(\bbGamma)$ by $\mu(e^\beta_{i,j}) = \delta_{\alpha,\beta} \delta_{i,j}$.  By \eqref{eq:wcmpt} we find $(\mu\otimes\id)(\ww^*) = (\id\otimes\mu)(\wh\ww) = \chi_\alpha$, and so
\[ \omega_\lambda(\chi_\alpha) = (\mu\otimes\omega_\lambda)(\ww^*) = \mu(a_\lambda) \xrightarrow[\lambda\in\Lambda]{} \mu(\I) = \dim(\alpha). \]
As $\Theta^l(a_\lambda)(x) = (\omega_\lambda\otimes\id)\wh\Delta(x)$ for $x\in \LL^\infty(\wh\bbGamma)$ we see that $\omega_\lambda = \omega_\lambda \star \eps = \eps \circ \Theta^l(a_\lambda)$ as a functional on $\mrm{C}(\wh\bbGamma)$, so in particular $\|\omega_\lambda\| \leq \mc{Z} \Lambda_{cb}(\bbGamma)$.  In fact, as $\bbGamma$ is strongly amenable, it follows from \cite[Proposition~3.1]{CBMultipliers} that $\|\omega_\lambda\| = \|\Theta^l(a_\lambda)\|_{cb}$.  Since $\omega_\lambda$ is central we can view it as a functional on $(\mc Z\LL^1(\wh\bbGamma))^* = \LL^{\infty}(\wh\bbGamma)/(\mc{Z}\LL^1(\wh\bbGamma))^{\perp}$ with norm $\|\omega_\lambda\|$, and so
\[ \|\chi_\alpha + (\mc{Z}\LL^1(\wh\bbGamma))^{\perp}\|_{\LL^{\infty}(\wh\bbGamma)/
(\mc{Z}\LL^1(\wh\bbGamma))^{\perp}}
\geq \limsup_{\lambda\in \Lambda} \la \chi_\alpha, \tfrac{\omega_\lambda}{\|\omega_\lambda\|} \ra
\geq \tfrac{\dim(\alpha)}{\mc{Z}\Lambda_{cb}(\bbGamma)}, \]
which establishes $3)$.
\end{proof}

The next theorem shows that the properties described in Theorem~\ref{thm:main_fourierinv} do not always hold. In particular, since the dual of $\SU_q(2)$ 
is strongly amenable and centrally weakly amenable they do not hold for $\bbGamma=\wh{\SU_q(2)}$ with $q\in \left]-1,1\right[\setminus\{0\}$.

\begin{theorem}\label{thm1}
Let $\bbGamma$ be a discrete quantum group which is strongly amenable and non-unimodular.
\begin{enumerate}
\item[1)] If $\bbGamma$ has central AP then $\B_r(\wh\bbGamma\backslash D(\bbGamma) / \wh\bbGamma )\subsetneq \mc{Z}\B_r(\bbGamma)\otimes \I$.
\item[2)]
If $\bbGamma$ is centrally weakly amenable then $\A(\wh\bbGamma\backslash D(\bbGamma) / \wh\bbGamma)\subsetneq \mc{Z}\A(\bbGamma)\otimes \I$.
\end{enumerate}
\end{theorem}

\begin{proof}
By Corollary~\ref{corr:links_centreL1_Q}, $\iota^r_*$ is a surjection if and only if $\mf{Q}_0$ has norm-closed image. If this is the case, then $\mf Q_0$ drops to an isomorphism $\mrm{C}(\wh\bbGamma) / \ker\mf Q_0 \to \mf Q_0(\mrm{C}(\wh\bbGamma))$ by the Open Mapping Theorem. In particular, there exists $\delta>0$ with
\[ \|\mf Q_0(x) \| \geq \delta \| x + \ker\mf Q_0 \|_{\mrm{C}(\wh\bbGamma) / \ker\mf Q_0} \qquad (x\in \mrm{C}(\wh\bbGamma)). \]
By \eqref{eq:defn_mfQ}, this shows that
\[ \frac{\dim(\alpha)}{\dim_q(\alpha)} \|\chi_\alpha\| \geq \delta \| \chi_\alpha + \ker\mf Q_0 \|_{\mrm{C}(\wh\bbGamma) / \ker\mf Q_0} \qquad (\alpha\in\Irr(\wh\bbGamma)). \]

$1)$ Suppose $\bbGamma$ has central AP, so by Theorem~\ref{prop:centralAP_density}, $\ker\mf Q_0 = {}^\perp (\mc Z(\mrm{C}(\wh\bbGamma)^*))$. Using Lemma~\ref{lemma2} and 
our assumption that $\bbGamma$ is strongly amenable, we hence conclude that
\[ \frac{\dim(\alpha)^2}{\dim_q(\alpha)} 
= \frac{\dim(\alpha)}{\dim_q(\alpha)} \|\chi_\alpha\|
\geq \delta \| \chi_\alpha + \ker\mf Q_0 \| = \delta \dim(\alpha)
\implies
\dim(\alpha) \geq \delta\dim_q(\alpha), \]
for each $\alpha\in\Irr(\wh\bbGamma)$, which means that $\bbGamma$ must be unimodular.
Indeed, given any $\alpha\in\Irr(\wh\bbGamma)$ and $n\geq 1$, let $\alpha^{\stp n}$ have decomposition $\alpha^{\stp n} = \bigoplus_{i\in I} \beta_i$ for some finite index set $I$ and $\beta_i\in\Irr(\wh\bbGamma)$.  As the usual dimension and quantum dimension are additive and multiplicative and respect equivalence,
\[ \dim_q(\alpha)^n = \dim_q(\alpha^{\stp n})
= \sum_{i\in I} \dim_q(\beta_i)
\leq \sum_{i\in I} \delta^{-1} \dim(\beta_i)
= \delta^{-1} \dim(\alpha^{\stp n})
= \delta^{-1} \dim(\alpha)^n. \]
Letting $n\rightarrow\infty$ shows that $\dim_q(\alpha) = \dim(\alpha)$, whence $\bbGamma$ is unimodular.
This is a contradiction to our assumptions, so $\iota^r_*$ cannot be surjective, and hence cannot be an isomorphism. This means that $\iota^r$ cannot be an isomorphism either (see for instance \cite[Chapter~VI, Proposition~1.9]{conway}), 
and hence Theorem~\ref{thm:main_fourierinv} yields the claim. 

$2)$ Suppose now that $\bbGamma$ is centrally weakly amenable. We again proceed by contradiction, assuming that $\iota^*$ is a surjection. Corollary~\ref{corr:links_centreL1_Q} then implies that $\mf Q$ has weak$^*$-closed image, so $\mf Q$ has also norm-closed image (see for example \cite[Chapter~VI, Theorem~1.10]{conway}). Hence there exists $\delta>0$ with 
\[ \frac{\dim(\alpha)}{\dim_q(\alpha)} \|\chi_\alpha\| \geq \delta \| \chi_\alpha + \ker\mf Q \|_{\LL^\infty(\wh\bbGamma) / \ker\mf Q} \qquad (\alpha\in\Irr(\wh\bbGamma)). \]
By Theorem~\ref{prop:centralAP_density} we have $(\mc Z\LL^1(\wh\bbGamma))^\perp = \ker\mf Q$, and so by Lemma~\ref{lemma2} we obtain
\[ \frac{\dim(\alpha)^2}{\dim_q(\alpha)} = 
\frac{\dim(\alpha)}{\dim_q(\alpha)} \|\chi_\alpha\|
\geq \delta \frac{\dim(\alpha)}{\mc Z\Lambda_{cb}(\bbGamma)}
\implies
\dim(\alpha) \geq \frac{\delta}{\mc Z\Lambda_{cb}(\bbGamma)} \dim_q(\alpha). \]
This gives again a contradiction to our assumption that $\bbGamma$ is not unimodular. In the same way as above, it follows that $\iota^*$ cannot be surjective, and hence cannot be an isomorphism. We conclude that $\iota$ cannot be an isomorphism, and then Theorem~\ref{thm:main_fourierinv} completes again the proof. 
\end{proof}

\begin{remark}
As already indicated above, Theorem \ref{thm1} applies in particular to $\bbGamma=\wh{\SU_q(2)}$. One can obtain the conclusion $\A(\wh\bbGamma \backslash D(\bbGamma) / \wh\bbGamma)\subsetneq \mc{Z}\A(\bbGamma)\otimes \I$ also in another way in this case, which we now briefly sketch. By \cite[Remark 31]{DFY_CCAP}, there exists a bounded central functional $\omega\in \mc{Z} (\mrm{C}(\wh\bbGamma)^*)$ which cannot be written as a linear combination of positive functionals in $\mc{Z} (\mrm{C}(\wh\bbGamma)^*)$. 
Based on this one can show $\B_r(\wh\bbGamma \backslash D(\bbGamma) / \wh\bbGamma)\subsetneq \mc{Z}\B_r(\bbGamma)\otimes \I$, and as $\mc{Z}\Lambda_{cb}(\bbGamma)=1$ by \cite[Theorem 24]{DFY_CCAP}, Proposition \ref{prop:ws_dense_im} and Theorem \ref{thm:main_fourierinv} imply $\A(\wh\bbGamma \backslash D(\bbGamma) / \wh\bbGamma)\subsetneq \mc{Z}\A(\bbGamma)\otimes \I$.
\end{remark}

\subsection{Summary}

We conclude with a brief summary of our main results in this section. By Theorem~\ref{thm:main_fourierinv}, we have $\A(\wh\GGamma\backslash D(\GGamma)/\wh\GGamma) = \mc Z\A(\bbGamma)\otimes\I$ if and only if $\iota$ is an isomorphism.  This is equivalent to $\iota^*$ being an isomorphism, a problem which can be split into two subproblems:
\begin{itemize}
\item Injectivity of $\iota^*$. By Corollary~\ref{corr:links_centreL1_Q} this is equivalent to $(\mc Z\LL^1(\wh\bbGamma))^\perp = \ker\mf Q$, which is further equivalent to $\mc Z\mrm{c}_{00}(\bbGamma)$ being dense in $\mc Z\A(\bbGamma)$, see Proposition~\ref{prop:kerQ_to_centrec00}. These conditions hold for unimodular $\bbGamma$, Theorem~\ref{thm:unimod_inv_fourier}, and when $\bbGamma$ has the central AP, Theorem~\ref{prop:centralAP_density}.
\item Surjectivity of $\iota^*$. This is equivalent to $\iota^*$ having weak$^*$-closed image and, by Corollary~\ref{corr:links_centreL1_Q}, also equivalent to $\mf Q$ having weak$^*$-closed image. It follows from Theorem~\ref{thm1} that this property does not always hold.
\end{itemize}
The situation for $\B_r(\wh\GGamma\backslash D(\GGamma)/\wh\GGamma)$ and $\mc Z\B_r(\bbGamma)\otimes\I $ is entirely analogous.  We leave open whether it is always true that $(\mc Z\LL^1(\wh\bbGamma))^\perp = \ker\mf Q$, and whether $\iota^*$ being surjective could be equivalent to $\bbGamma$ being unimodular.

\bibliographystyle{plain}
\bibliography{bibliografia}

\end{document}